\PassOptionsToPackage{unicode}{hyperref}
\PassOptionsToPackage{hyphens}{url}
\PassOptionsToPackage{dvipsnames,svgnames,x11names}{xcolor}
\documentclass[
  12pt]{article}

\usepackage{amsmath,amssymb}
\usepackage{iftex}
\usepackage{subcaption}
\usepackage{amsthm}
\theoremstyle{definition}
\newtheorem{definition}{Definition}
\newtheorem{example}{Example}
\theoremstyle{remark}
\newtheorem{remark}{Remark}
\theoremstyle{plain}
\newtheorem{theorem}{Theorem}[section]
\newtheorem{corollary}[theorem]{Corollary}
\theoremstyle{definition}
\theoremstyle{remark}
\theoremstyle{plain}
\newtheorem{assumption}{Assumption}[section]
\usepackage{multirow}
\usepackage{diagbox}
\numberwithin{equation}{section}

\usepackage{multirow}
\usepackage{diagbox}

\usepackage{subcaption}

\usepackage[ruled, noend]{algorithm2e}

\ifPDFTeX
  \usepackage[T1]{fontenc}
  \usepackage[utf8]{inputenc}
  \usepackage{textcomp} 
\else 
  \usepackage{unicode-math}
  \defaultfontfeatures{Scale=MatchLowercase}
  \defaultfontfeatures[\rmfamily]{Ligatures=TeX,Scale=1}
\fi
\usepackage{lmodern}
\ifPDFTeX\else  
\fi
\IfFileExists{upquote.sty}{\usepackage{upquote}}{}
\IfFileExists{microtype.sty}{
  \usepackage[]{microtype}
  \UseMicrotypeSet[protrusion]{basicmath} 
}{}
\makeatletter
\@ifundefined{KOMAClassName}{
  \IfFileExists{parskip.sty}{%
    \usepackage{parskip}
  }{
    \setlength{\parindent}{0pt}
    \setlength{\parskip}{6pt plus 2pt minus 1pt}}
}{
  \KOMAoptions{parskip=half}}
\makeatother
\usepackage{xcolor}
\makeatletter
\ifx\paragraph\undefined\else
  \let\oldparagraph\paragraph
  \renewcommand{\paragraph}{
    \@ifstar
      \xxxParagraphStar
      \xxxParagraphNoStar
  }
  \newcommand{\xxxParagraphStar}[1]{\oldparagraph*{#1}\mbox{}}
  \newcommand{\xxxParagraphNoStar}[1]{\oldparagraph{#1}\mbox{}}
\fi
\ifx\subparagraph\undefined\else
  \let\oldsubparagraph\subparagraph
  \renewcommand{\subparagraph}{
    \@ifstar
      \xxxSubParagraphStar
      \xxxSubParagraphNoStar
  }
  \newcommand{\xxxSubParagraphStar}[1]{\oldsubparagraph*{#1}\mbox{}}
  \newcommand{\xxxSubParagraphNoStar}[1]{\oldsubparagraph{#1}\mbox{}}
\fi
\makeatother

\usepackage{longtable,booktabs,array}
\usepackage{calc} 
\usepackage{etoolbox}
\makeatletter
\patchcmd\longtable{\par}{\if@noskipsec\mbox{}\fi\par}{}{}
\makeatother
\IfFileExists{footnotehyper.sty}{\usepackage{footnotehyper}}{\usepackage{footnote}}
\makesavenoteenv{longtable}
\usepackage{graphicx}
\makeatletter
\def\maxwidth{\ifdim\Gin@nat@width>\linewidth\linewidth\else\Gin@nat@width\fi}
\def\maxheight{\ifdim\Gin@nat@height>\textheight\textheight\else\Gin@nat@height\fi}
\makeatother
\setkeys{Gin}{width=\maxwidth,height=\maxheight,keepaspectratio}
\makeatletter
\def\fps@figure{htbp}
\makeatother

\makeatletter
\@ifpackageloaded{caption}{}{\usepackage{caption}}
\AtBeginDocument{%
\ifdefined\contentsname
  \renewcommand*\contentsname{Table of contents}
\else
  \newcommand\contentsname{Table of contents}
\fi
\ifdefined\listfigurename
  \renewcommand*\listfigurename{List of Figures}
\else
  \newcommand\listfigurename{List of Figures}
\fi
\ifdefined\listtablename
  \renewcommand*\listtablename{List of Tables}
\else
  \newcommand\listtablename{List of Tables}
\fi
\ifdefined\figurename
  \renewcommand*\figurename{Figure}
\else
  \newcommand\figurename{Figure}
\fi
\ifdefined\tablename
  \renewcommand*\tablename{Table}
\else
  \newcommand\tablename{Table}
\fi
}
\@ifpackageloaded{float}{}{\usepackage{float}}
\floatstyle{ruled}
\@ifundefined{c@chapter}{\newfloat{codelisting}{h}{lop}}{\newfloat{codelisting}{h}{lop}[chapter]}
\floatname{codelisting}{Listing}

\makeatother
\makeatletter
\@ifpackageloaded{caption}{}{\usepackage{caption}}
\@ifpackageloaded{subcaption}{}{\usepackage{subcaption}}
\makeatother

\ifLuaTeX
  \usepackage{selnolig}  
\fi
\usepackage[]{natbib}
\usepackage{bookmark}

\IfFileExists{xurl.sty}{\usepackage{xurl}}{} 
\hypersetup{
  pdftitle={Title},
  pdfauthor={Author 1; Author 2},
  pdfkeywords={3 to 6 keywords, that do not appear in the title},
  colorlinks=true,
  linkcolor={blue},
  filecolor={Maroon},
  citecolor={Blue},
  urlcolor={Blue},
  pdfcreator={LaTeX via pandoc}}

 \title{\bf Tests for Independence of High-Dimensional Nonstationary Time Series}
  \author{Yunyi Zhang\\
  School of Data Science, The Chinese University of Hong Kong, Shenzhen,\\ Shenzhen, Guangdong 518172, China\\
  zhangyunyi@cuhk.edu.cn
  }

\begin{document}

\def\spacingset#1{\renewcommand{\baselinestretch}%
{#1}\small\normalsize} \spacingset{1}


  \maketitle

\bigskip
\begin{abstract}
This manuscript studies the problem of  independence testing between two high-dimensional time series without assuming weak stationarity, allowing their autocovariances to vary over time. To this end, we propose a bimodal weighted-average test statistic that removes the bias induced by temporal dependence under the null hypothesis, thereby avoiding the need to whiten the time series prior to hypothesis testing---a procedure that is challenging in high-dimensional and nonstationary settings. To facilitate statistical inference, we develop a dependent wild bootstrap procedure. On the theoretical side, we derive a concentration inequality for quadratic forms of time series data stemming from a class of high-dimensional, nonlinear, and  nonstationary processes. This result enables us to derive the asymptotic null distribution of the proposed test statistic and to establish the validity of the bootstrap algorithm. Numerical results show that the proposed test attains desired size and good power performance even when the dimension exceeds the sample size or when the data-generating process exhibits time-varying autocovariances.  In contrast, tests based  on whitening time series fail to maintain correct size in the presence of unstable autocovariance structures. Since nonstationary autocovariances commonly arise in real-life time series data, our work offers a robust procedure for independence testing. 
\end{abstract}

\noindent%
\vfill

\newpage
\spacingset{1.8} 

\section{Introduction}

\subsection{Overview}

Suppose that we observe a mean-zero, vector-valued time series $\mathbf{z}_t = \left(\mathbf{x}_t^\top, \mathbf{y}_t^\top\right)^\top\in\mathbf{R}^{d_1 + d_2}$ for $t = 1,2,\cdots, T.$ We consider testing the null hypothesis $H_0$ that 
the sequences $\left\{\mathbf{x}_t\right\}_{t = 1,\cdots, T}$ and $\left\{\mathbf{y}_t\right\}_{t = 1,\cdots, T}$ are independent, against the alternative hypothesis $H_1$ that they are dependent.

Independence testing in time series has attracted considerable attention in economics, financial data analysis, machine learning, and biological studies, as evidenced by the growing body of literature, including \cite{https://doi.org/10.1111/j.1540-6261.1987.tb04368.x, MR2472050, NIPS2013_ae5e3ce4, MR3569280,  https://doi.org/10.1111/ecog.04360, MR4270386, NEURIPS2022_6739d8df, 10.1371/journal.pbio.3002758, MR5027588}. Several reasons motivate this line of research. For example, \cite{MR3416594} leveraged such testing procedures to investigate the relationship between inflation and currency exchange rates. In addition, when a time series comprises many features, practitioners may assess dependence structures among subsets of variables, where independence suggests separation into distinct communities. Furthermore, independence testing aids model simplification \cite{MR4270386}. If two time series are found to be independent, they may be modeled separately rather than jointly, leading to a substantial reduction in model complexity.

Early contributions, such as \cite{MR845894, MR1423878,  MR1410877, MR1613453, MR1997176, MR2472050}, primarily considered independence testing for scalar processes, while subsequent studies \cite{MR1463322, MR2245711, MR3416594, MR4270386, MR4384115, MR4530924} generalized the hypothesis testing problem to multivariate and functional time series. 
A common approach in the literature --- see, for example, \cite{MR2472050, MR3416594} ---  assumes that $\mathbf{x}_t$ and $\mathbf{y}_t$ follow autoregressive models with independent innovations and reduces the hypothesis testing problem to testing independence between fitted innovations. Although effective in low-dimensional settings, such an approach faces two challenges in high dimensions or when autocovariance stationarity is violated. First, in high-dimensional settings, consistent estimation of vector autoregressive models typically requires structural assumptions such as sparsity or low rank in the coefficient matrices \cite{MR3450535, MR3357870, MR4278792, zhang2023statistical, MR4480716}, which may be restrictive in practice. Second, when the autocovariance structure is nonstationary, fitting autoregressive models no longer recovers the underlying innovations, thus invalidating distributional approximations. Moreover, assessing autocovariance stationarity can be difficult in high dimensions \cite{10.1093/jrsssb/qkaf064}.

To address these challenges, this manuscript  proposes a bimodal weighted-average test statistic to test independence between two high-dimensional time series.  The proposed method allows the dimensions $d_1$ and $d_2$  to  exceed the sample size $T,$  imposes no stationarity assumptions on the autocovariance structure, and does not rely on vector autoregressive modeling. By avoiding these restrictions, the procedure is broadly applicable and robust to model misspecification.

The remainder of this manuscript is organized as follows: Section \ref{section.background_intro} reviews related literature.  Section \ref{section.methodology} introduces the test statistic and a dependent wild bootstrap algorithm that assists statistical inference via computer simulations. Section \ref{section.asymptotic} establishes the asymptotic distribution of the test statistic under $H_0,$ and proves the validity of the proposed bootstrap algorithm. Sections \ref{section.numerical_simulation} and \ref{section.applications} respectively  demonstrate the finite-sample performance and the applications of the test through numerical experiments and real data examples. Additional results, including further numerical experiments and detailed proofs of the theoretical results, are postponed to the online supplement.

\subsection{Related literature}
\label{section.background_intro}
This section provides a review of the literature on independence testing and its practical applications; further review of methods for analyzing high-dimensional vector autoregressive (VAR) processes is deferred to the online supplement. 

\textbf{Independence testing.} Research on testing the independence of two stationary scalar time series dates back to \cite{MR418379, Pierce01031977}, with more recent studies including \cite{MR4513134, MR4708634, Leong03042025}.  This line of research has been extended to vector time series; see, for example, \cite{MR2245711, MR2354407, MR3416594, MR4270386, dette2024detecting}. We also refer to \cite{MR4384115, MR4829509} for independence testing in the functional data setting, and to \cite{MR3782387, MR3798874, MR4185812, MR4404921, MR4797902} for independence testing with high-dimensional independent observations. Despite this extensive body of literature, to the best of our knowledge, relatively few studies have been conducted on independence testing for high-dimensional nonstationary time series, where the temporal dependence can affect the null distribution of the test statistic under $H_0$, potentially leading to size inflation.

\textbf{Applications.}    Independence testing serves as a key component in many complex algorithms. For example, as shown in \cite{NIPS2008_f7664060, spirtes2012causation, NIPS2013_47d1e990}, it complements causal discovery methods by determining the direction of causality. In addition, the works of \cite{10.1145/1273496.1273600, MR2930643} employ independence testing for feature selection. We also refer readers to \cite{doi:10.1073/pnas.1420291112, PhysRevResearch.3.013145, communication_nature, 10.1371/journal.pbio.3002758, DOUBOVIKOV2025121334}, among others, for several real-world data applications.

\subsection{Frequently-used notations}
This manuscript uses the standard order notations $O(\cdot), o(\cdot), O_p(\cdot), o_p(\cdot):$ For two numerical sequences $a_t, b_t, t = 1,2, \cdots,$ the notation $a_t = O(b_t)$ indicates that there exists a constant $C > 0$ such that $\vert a_t\vert\leq C\vert b_t\vert$ for any $t\in\mathbf{Z};$ and $a_t = o(b_t)$ indicates $\lim_{t\to\infty} \left(\frac{a_t}{b_t}\right) = 0.$ In particular, we say $a_t\asymp b_t$ if the conditions $a_t = O(b_t)$ and $b_t = O(a_t)$ hold true simultaneously. For two random variable sequences $X_t, Y_t,$ we say $X_t = O_p(Y_t)$ if for any given $0 < \varepsilon < 1,$ there exists a constant $C_\varepsilon > 0$ (depending only on the value of $\varepsilon $) such that $\mathbf{pr}\left(\vert X_t\vert\leq C_\varepsilon\vert Y_t\vert\right)\geq 1 - \varepsilon$ for any $t;$ and $X_t = o_p(Y_t)$ if $X_t / Y_t\to_p 0,$ where $\to_p$ denotes convergence in probability. The reader can refer to Definition 1.9 of \cite{MR2002723} for a detailed introduction to these notations. For a random variable $X\in\mathbf{R}$ and a real number $m\geq 1,$ define its $m$-norm $\Vert X\Vert_m = \mathbf{E}\left[\vert X\vert^m\right]^{1/m}.$  We use the notations $\vee,\wedge$ to represent the ``minimum'' and ``maximum'' of two numbers, i.e., $a\wedge b = \min(a,b),$ $a\vee b = \max(a,b).$ We use $C,C_0,C_1,C_2,\cdots$ to represent general constants, whose values are independent of the sample size $T.$

In the manuscript, the bold (Greek) lower-case letters, such as $\mathbf{a}$ or $\boldsymbol{\omega},$ refer to (random) vectors, while the bold (Greek) upper-case letters, such as $\mathbf{A}$ or $\boldsymbol{\Omega},$ refer to (random) matrices. Superscripts in brackets represent the specific elements of the vector. For example, a vector $\mathbf{a} = \left(\mathbf{a}^{(1)},\cdots, \mathbf{a}^{(d)}\right)^\top\in\mathbf{R}^d$ and a matrix $\mathbf{A} = \left\{\mathbf{A}^{(i,j)}\right\}_{i = 1,\cdots, d_1, j =1,\cdots, d_2}\in\mathbf{R}^{d_1\times d_2}.$ For a vector $\mathbf{a}\in\mathbf{R}^d$ and a number $m\geq 1,$ we denote its $m$-norm $\vert\mathbf{a}\vert_m = \left(\sum_{j = 1}^d\vert\mathbf{a}^{(j)}\vert^m\right)^{1/m}$ and $\vert\mathbf{a}\vert_\infty = \max_{j = 1,\cdots, d}\left\vert\mathbf{a}^{(j)}\right\vert.$ For a $d\times d$ matrix, we denote its spectral norm by $\vert\mathbf{A}\vert_2$ as its largest singular value, and its Frobenius norm $\vert\mathbf{A}\vert_F = \sqrt{\sum_{i = 1}^d\sum_{j = 1}^d \mathbf{A}^{(i,j)2}}.$ For a matrix $\mathbf{A}\in\mathbf{R}^{d\times d},$ define the trace of $\mathbf{A}$ by  $\mathrm{Tr}\left(\mathbf{A}\right) = \sum_{j = 1}^d \mathbf{A}^{(j,j)}.$

\section{Setting and methodology}
This section introduces the bimodal weighted-average test statistic, along with a corresponding dependent wild bootstrap algorithm for hypothesis testing.
\label{section.methodology}

\subsection{Constructing the test statistic}
Suppose the observed time series $\mathbf{z}_t = \left(\mathbf{x}_t^\top, \mathbf{y}_t^\top\right)^\top\in\mathbf{R}^{d_1 + d_2}$ satisfies $\mathbf{E}\left[\mathbf{z}_t\right] = 0,$        and the objective is to test the independence between the sequences $\left\{\mathbf{x}_t\right\}_{t\in\mathbf{Z}}$ and  $\left\{\mathbf{y}_t\right\}_{t\in\mathbf{Z}}.$ Unlike in the independent-data setting \cite{NIPS2007_d5cfead9, MR4404921}, the presence of nonzero autocovariances in time series induces bias in the test statistic, inflating its magnitude even under  $H_0.$ Existing literature, such as \cite{MR3416594, MR4270386} removed the autocovariances by 
whitening the data via fitting a vector autoregressive model. However, such methods are generally unsuitable for  nonstationary time series.

This work attenuates autocovariances by excluding pairs with small time lags. Specifically, for two given integer bandwidths $1 < \lambda_1 < \lambda_2 < T$ and a given lag $h\geq 0,$ we define the test statistic as
\begin{equation}
\begin{aligned}
    \widehat{R} = \sum_{t = 1 + \lambda_1}^T O_t,
    \quad O_t &= \frac{1}{\mathcal{U}}\sum_{t_1 = (t - \lambda_2)\vee 1}^{t - \lambda_1}\sum_{s = 0}^{h\wedge (t_1 - 1)}\mathbf{x}_{t}^\top\mathbf{x}_{t_1}\mathbf{y}_{t - s}^\top\mathbf{y}_{t_1 - s}\\
     & + \frac{1}{\mathcal{U}}\sum_{t_1 = (t - \lambda_2)\vee 1}^{t - \lambda_1}\sum_{s = 1}^{h\wedge(t_1 - 1)}\mathbf{x}_{t - s}^\top\mathbf{x}_{t_1 - s}\mathbf{y}_{t}^\top\mathbf{y}_{t_1},
\end{aligned}
\label{eq.def_R}
\end{equation}
where $\mathcal{U} = 2\sqrt{d_1d_2T(\lambda_2 - \lambda_1)h}$ normalizes $\widehat{R}$ to remain order $O_p(1)$ under $H_0$. As in \cite{MR3416594, MR4270386}, we sum over lags $s$ up to $h$ to  detect serial dependence up to lag $h.$ Allowing $h$ to diverge with $T$ enables the test statistic to capture cross-series dependence over increasingly long lags. We further demonstrate in Remark \ref{remark.underH0} that  a sufficiently large $\lambda_1$ reduces the bias induced by temporal dependence. The idea of excluding pairs to mitigate bias has also been explored in the literature, for example in \cite{MR2604697, MR3911118} in the context of independent data.

We refer to $\widehat{R}$ as  ``bimodal weighted-average'' because $\mathcal{U}\widehat{R}$ can be written as
\begin{align*}
&\frac{1}{2}\sum_{t_1 = 1}^T\sum_{t_2 = 1}^T f(t_1 - t_2)\sum_{s = 0}^{h\wedge (t_1 \wedge t_2 - 1)}\mathbf{x}_{t_1}^\top\mathbf{x}_{t_2}\mathbf{y}_{t_1 - s}^\top\mathbf{y}_{t_2 - s}\\
&+ \frac{1}{2}\sum_{t_1 = 1}^T\sum_{t_2 = 1}^T f(t_1 - t_2)\sum_{s = 1}^{h\wedge(t_1\wedge t_2 - 1)}\mathbf{x}_{t_1 - s}^\top\mathbf{x}_{t_2 - s}\mathbf{y}_{t_1}^\top\mathbf{y}_{t_2},
\end{align*}
where $f(x) = 1$ if $\lambda_1\leq \vert x\vert \leq \lambda_2,$ and  $f(x) = 0$ otherwise. Viewing $f(\cdot)$ as a kernel function, it is even and exhibits a bimodal shape: it vanishes at the origin and for large $|x|$, while attaining its maximum at intermediate values in $(0,\infty).$ This contrasts with standard kernel functions, which typically peak at zero and decay as $|x|$ increases. In this sense, the proposed test statistic is constructed using a piecewise constant kernel $f.$ 
Remarks \ref{remark.underH0} and \ref{remark.underH1} derive the expectations of the test statistic under both $H_0$ and $H_1,$ thereby explaining why  $\widehat{R}$ distinguishes between temporal dependence and dependence across  two time series.

\begin{remark}[Under $H_0$]
    \label{remark.underH0} 
Under independence between the sequences $\{\mathbf{x}_t\}_{t\in\mathbf{Z}}$ and $\{\mathbf{y}_t\}_{t\in\mathbf{Z}},$
temporal dependence dominates the bias of $\widehat{R}.$ We demonstrate in Remark \ref{remark.underH0_appendix} of the online supplement that the expectation of $\widehat{R}$ satisfies 
\begin{equation}
\begin{aligned}
   &\left\vert\mathbf{E}\left[ \widehat{R} \right]\right\vert\leq \frac{CTd_1d_2h}{\mathcal{U}\lambda_1^{2\alpha - 1}}
\end{aligned}
\label{eq.E_R_hat}
\end{equation}
under $H_0.$ Therefore, for sufficiently large $\lambda_1,$ the bias arising from temporal dependence becomes asymptotically negligible compared to its stochastic error. 

\end{remark}

\begin{remark}[Under $H_1$]
    \label{remark.underH1}
    Under the alternative hypothesis, due to dependence between the two time series, $\mathbf{E}\left[\mathbf{x}_{t}^\top\mathbf{x}_{t_1}\mathbf{y}_{t - s}^\top\mathbf{y}_{t_1 - s}\right]\neq 
\mathbf{E}\left[\mathbf{x}_{t}^\top\mathbf{x}_{t_1}\right]\mathbf{E}\left[\mathbf{y}_{t - s}^\top\mathbf{y}_{t_1 - s}\right]$ in general. For any $t\in\mathbf{Z}$ and $s\geq 0,$ we define the cross-covariance matrices $\boldsymbol{\Pi}_{t,s,1} = \mathbf{E}\left[\mathbf{x}_t\mathbf{y}_{t-s}^\top\right]\in\mathbf{R}^{d_1\times d_2}$ and $\boldsymbol{\Pi}_{t,s,2} = \mathbf{E}\left[\mathbf{y}_t\mathbf{x}_{t-s}^\top\right]\in\mathbf{R}^{d_2\times d_1}.$ We demonstrate in Remark \ref{remark.underH1_appendix} of the online supplement that the expectation $\mathbf{E}\left[O_t\right]$ contains contributions involving the products $\mathrm{Tr}\left(\boldsymbol{\Pi}_{t_1,s,1}\boldsymbol{\Pi}_{t,s,1}^\top\right)$ and $\mathrm{Tr}\left(\boldsymbol{\Pi}_{t_1,s,2}\boldsymbol{\Pi}_{t,s,2}^\top\right),$ where $\mathrm{Tr}(\cdot)$ denotes the trace of a matrix. Consequently, the proposed statistic accumulates these nonzero contributions, leading to a non-negligible bias under $H_1.$ This mechanism allows the test to distinguish cross-series dependence from temporal dependence within each series.
\end{remark}

Remark \ref{remark.underH1} also highlights the effect of nonstationarity.  If the time series $\mathbf{x}_t$ and $\mathbf{y}_t$ have stationary autocovariances, then $\boldsymbol{\Pi}_{t,s,1} = \boldsymbol{\Pi}_{1,s,1}$ for all $t,$ and 
$
    \mathrm{Tr}\left(\boldsymbol{\Pi}_{t_1,s,1}\boldsymbol{\Pi}_{t,s,1}^\top\right) = \mathrm{Tr}\left(\boldsymbol{\Pi}_{1,s,1}\boldsymbol{\Pi}_{1,s,1}^\top\right)\geq 0,
$
which makes $\mathbf{E}\left[O_t\right]$ positive. In contrast, without autocovariance stationarity, this bias can be negative, and a two-sided (rather than one-sided) alternative is required, which is different from the tests for stationary time series considered in \cite{MR3416594}.

Figure \ref{figure.bias_plot} plots the dynamics of the bias of the test statistic under both $H_0$ and $H_1$, in agreement with Remarks \ref{remark.underH0} and \ref{remark.underH1}. Under the null, the bias is driven by the autocovariances, and increasing $\lambda_1$ renders this bias negligible. In contrast, under $H_1$, the presence of nonzero cross-covariance matrices induces extra bias in the test statistic, thereby inflating its absolute value.

\begin{figure}[htbp]
    \centering
    \begin{subfigure}[b]{0.45\textwidth}
        \centering
        \includegraphics[width=\linewidth]{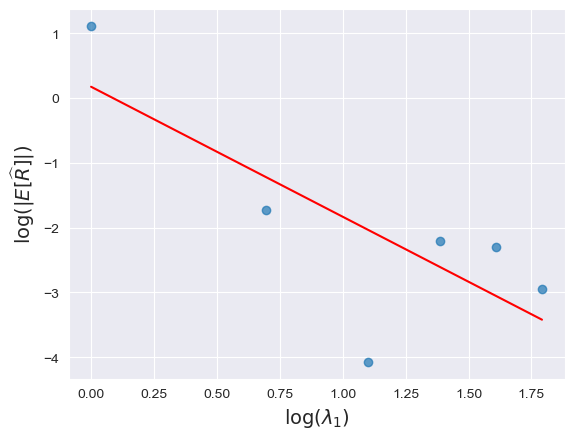}
        \caption{$\log(\mathrm{Bias})$ under $H_0$.}
        \label{null_situation}
    \end{subfigure}
    \hfill
    \begin{subfigure}[b]{0.45\textwidth}
        \centering
        \includegraphics[width=\linewidth]{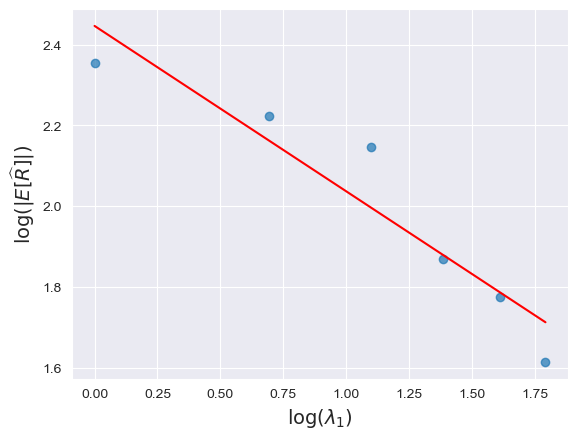}
        \caption{$\log(\mathrm{Bias})$ under $H_1$.}
       \label{fig.test_res}
    \end{subfigure}
    \caption{$\log(\mathrm{Bias})$ of the test statistic under $H_0$ and $H_1.$ Data are generated from the ``nAR(1)'' process in Section \ref{section.simulation_data_res}. Each value is calculated from 5000  independent replications. Blue dots denote the estimated bias, and red lines indicate the fitted regression line.} 
    \label{figure.bias_plot}
\end{figure}

\subsection{Bootstrap-assisted hypothesis testing}

The complex autocovariance structure of nonstationary time series induces a nontrivial variance for the test statistic. To circumvent explicit variance estimation, we develop a bootstrap procedure that approximates the sampling distribution of the test statistic through simulation.

The literature offers a wide range of bootstrap algorithms for statistical inference with time series data. Examples include \cite{MR1310224, MR1872222, MR1947405, MR3909939} for general stationary time series;  \cite{MR1466304, MR3343007, MR3346699} for linear processes; and \cite{MR3310530, MR3798001} for nonstationary time series. \cite{MR2656050} introduced the dependent wild bootstrap for stationary time series, and subsequent works, such as \cite{MR4726969, MR4829492}, have extended it to accommodate nonstationary settings.

The classical dependent wild bootstrap was primarily designed for linear combinations of observations (such as sample mean) and therefore could not fully capture higher-order cumulant information. More recent studies, including \cite{MR4829492, zhang2025anova}, have extended the algorithm to account for fourth-order cumulants.
Building on these advances, we employ a second-order dependent wild bootstrap for hypothesis testing.

The implementation of the bootstrap algorithm needs a kernel function that satisfies the following conditions. 

\begin{definition}[Kernel function]
    Suppose the function $K(\cdot):\mathbf{R}\to [0,\infty)$ satisfies $K(0) = 1,$ $K(\cdot)$ is decreasing on $[0,\infty)$ and is symmetric around $0,$ and $\sup_{x\in\mathbf{R}}\vert K^\prime(x)\vert < \infty.$ Assume that $\int_{\mathbf{R}}K(x)\mathrm{d}x < \infty.$ For any $t\in\mathbf{R},$ define the Fourier transform of $K(\cdot)$ as $\mathcal{F}K(t) = \int_{\mathbf{R}}K(x)\exp(-2\pi\mathrm{i}tx)\mathrm{d}x.$ We assume that $\mathcal{F}K(t)\geq 0$ for all $t\in\mathbf{R}$ and $\int_{\mathbf{R}}\mathcal{F}K(t)\mathrm{d}t < \infty.$
    \label{def.kernel}
\end{definition}

A canonical example is the Gaussian kernel $K(x) = \exp\left(-\frac{x^2}{2}\right),$ which satisfies the conditions of Definition \ref{def.kernel}.

\begin{remark}
The conditions on the Fourier transform $\mathcal{F}K$ ensure that the covariance matrix induced by the kernel is positive semi-definite. Specifically, for any real numbers $c_t, t = 1,\cdots, T,$ from Theorem 8.26 of \cite{MR1681462}, for any variance bandwidth $\mathcal{K}_T > 0,$
    \begin{align*}
        \sum_{t_1 = 1}^T\sum_{t_2 = 1}^T c_{t_1}c_{t_2}K\left(\frac{t_1 - t_2}{\mathcal{K}_T}\right) 
        = \int_{\mathbf{R}}\mathcal{F}K(y)\left\vert\sum_{t = 1}^T c_t \exp\left(\frac{2\pi\mathrm{i}yt_1}{\mathcal{K}_T}\right)\right\vert^2\mathrm{d}y\geq 0,
    \end{align*}
    so the matrix $\left\{K\left(\frac{t_1 - t_2}{\mathcal{K}_T}\right)\right\}_{t_1,t_2 = 1,\cdots,T}$ is positive semi-definite.
\end{remark}

\begin{algorithm}[H]
\caption{Second-order dependent wild bootstrap for hypothesis testing.}
\label{algorithm.dependent_wild_bootstrap}

\SetAlgoLined
\KwIn{Observations $\mathbf{z}_t = \left(\mathbf{x}_t^\top, \mathbf{y}_t^\top\right)^\top\in\mathbf{R}^{d_1 + d_2}$ for $t = 1,\cdots,T,$ integer bandwidths
   $1\leq\lambda_1 < \lambda_2 < T$ and $h\geq 0,$  a kernel function $K(\cdot)$ satisfying Definition \ref{def.kernel}, a variance bandwidth
    $\mathcal{K}_T > 0,$ number of bootstrap replicates $B,$ and the nominal size (Type-I error) $0<\mathcal{A}<1.$}
    
\hspace{1in}

 Derive the test statistic $\widehat{R}$ as well as $O_t$ for $t = 1+\lambda_1,\cdots, T$ as in equation \eqref{eq.def_R}.

\For{$b = 1$ \KwTo $B$}{
    Generate joint normal random variables $\left(e^*_1,\cdots,e^*_T\right)$ such that 
\begin{align*}
 \mathbf{E}\left[e^*_t\right] = 0,\quad \mathrm{Cov}\left(e^*_{t_1}, e^*_{t_2}\right) = K\left(\frac{t_1 - t_2}{\mathcal{K}_T}\right).
\end{align*}

Calculate the linear combination $\widehat{R}^*_b = \sum_{t = 1+ \lambda_1}^T O_te^*_t.$
}
Calculate the $1 - \alpha$ sample quantile $q^*_{1-\alpha}$ of $\left\vert \widehat{R}^*_b\right\vert:$ Sort $\left\vert\widehat{R}^*_b\right\vert$ into $\left\vert\widehat{R}^*_{(1)}\right\vert\leq \left\vert\widehat{R}^*_{(2)}\right\vert\leq \cdots\leq \left\vert\widehat{R}^*_{(B)}\right\vert,$ and the sample quantile is given by $\left\vert\widehat{R}^*_{(q^*)}\right\vert,$ where $q^*$ is the smallest number such that $q^* / B \geq 1 - \mathcal{A}.$

\KwOut{Reject $H_0$ if $\left\vert \widehat{R}\right\vert > q^*_{1-\mathcal{A}}$.}
\end{algorithm}

The use of a two-sided rejection region is necessary due to the effect of autocovariance nonstationarity discussed in Remark \ref{remark.underH1}, under which the bias need not be positive.

To elucidate the mechanism of Algorithm \ref{algorithm.dependent_wild_bootstrap}, note that $e^*_1,\cdots,e^*_T$ have joint normal distribution, so $\widehat{R}^*_b,$ which is a linear combination of $e^*_t,$ also has normal distribution in the bootstrap world (conditional on the observed data). Furthermore, the mean and variance of  $\widehat{R}^*_b$ in the bootstrap world are given by 0 and
$
\sum_{t_1 = 1 + \lambda_1}^T\sum_{t_2 = 1 + \lambda_1}^T O_{t_1}O_{t_2}K\left(\frac{t_1 - t_2}{\mathcal{K}_T}\right)
$
respectively.
Therefore, Algorithm \ref{algorithm.dependent_wild_bootstrap} actually generates mean-zero normal random variables with a specific variance. Furthermore, as demonstrated in Theorem \ref{theorem.distributional_under_H0}, the test statistic has an asymptotic mean-zero normal distribution under $H_0.$ Hence, the distribution of $\widehat{R}^*_b$ in the bootstrap world approximates that of $\widehat{R},$ provided that its variance is close to the asymptotic variance of $\widehat{R}.$ This variance consistency is proved in \eqref{eq.difference_variance} of the online supplement.

\begin{remark}
    The variance of $\widehat{R}^*_b$ in the bootstrap world serves as a heteroskedasticity and autocorrelation consistent (HAC) estimator of the asymptotic variance of $\widehat{R}$. HAC estimators, originating in \cite{MR890864, MR1106513}, have proven useful for heterogeneous data and nonstationary time series; see, for example, \cite{MR2656050, MR3310530, MR4134802, MR4410878, MR4392159, MR4726969}. 
    In our setting, where no parametric structure is imposed on fourth-order cumulants, HAC estimation provides a natural and flexible approach  for estimating variance.
\end{remark}

\section{Asymptotic results}
\label{section.asymptotic}
This section establishes the asymptotic distribution of the test statistic under $H_0,$ as well as the consistency of Algorithm \ref{algorithm.dependent_wild_bootstrap}. To this end, we introduce a class of short-range dependence conditions, named $(M,\alpha)$-short-range dependence, which accommodates nonlinear and nonstationary data-generating processes. Under these conditions, we derive concentration inequalities for quadratic forms of high-dimensional time series. These results are of independent interest, as many statistics in the time series literature---including sample covariance and precision matrices, spectral density \cite{MR4206676}, and sample autoregressive coefficient matrices \cite{MR3620446, zhang2023statistical}---can be reduced to the analysis of quadratic forms. Consequently, the developments in Section \ref{section.malpha_short_range} may be of broader interest beyond the scope of this paper.

\subsection{A short-range dependence condition}
\label{section.malpha_short_range}
Suppose $\left\{e_t\in\mathbf{H}, t\in\mathbf{Z}\right\}$ are independent (but not necessarily identically distributed) random variables taking values in $\mathbf{H}$, where $\mathbf{H}$ is a measurable space. We assume that 
\begin{equation}
    \mathbf{z}_t = g_{t,T}\left(\cdots, e_{t-2}, e_{t-1}\right),
    \label{eq.def_z_t}
\end{equation}
where $g_{t,T}(\cdot):\mathbf{H}^{\mathbf{Z}}\to\mathbf{R}^{d_1+d_2}$ is a measurable function.  The subscript $t,T$ indicates that the function $g_{t,T}(\cdot)$ may vary with respect to the time index $t$ and the sample size $T.$  We define $\left\{e_t^\dagger \in\mathbf{H}, t\in\mathbf{Z}\right\}$ as an independent copy of $\left\{e_t\in\mathbf{H}, t\in\mathbf{Z}\right\}:$ the random variable sequence $\left\{e_t^\dagger \in\mathbf{H}, t\in\mathbf{Z}\right\}$ is independent of $\left\{e_t\in\mathbf{H}, t\in\mathbf{Z}\right\},$ and $e^\dagger_t$ has the same distribution as $e_t$ for any $t.$ For any integer $s\in\mathbf{Z},$ define 
\begin{align*}
    \mathbf{z}_{t,s} = 
    \begin{cases}
    g_{t,T}\left(\cdots, e_{t-s - 2}, e_{t-s - 1}, e_{t-s}^\dagger, e_{t-s+1},\cdots, e_t\right) & \text{if}\quad s\geq 0,\\
    \mathbf{z}_t & \text{if}\quad s < 0.
    \end{cases}
\end{align*}
We define $\delta_{t,s}$ and $\delta_s$ as follows:
\begin{align*}
    \delta_{t,s} = \sup_{\mathbf{a}\in\mathbf{R}^{d_1+d_2},\ \vert\mathbf{a}\vert_2 = 1}\left\Vert\mathbf{a}^\top\left(\mathbf{z}_{t} - \mathbf{z}_{t,s}\right)\right\Vert_M,\quad \delta_s = \sup_{t\in\mathbf{Z}}\delta_{t,s}.
\end{align*}
With a slight abuse of notation, we do not explicitly include $M$ in $\delta_{t,s}$ and $\delta_s.$ Since $M$ is considered as fixed throughout the manuscript, this abuse of notation should not cause confusion. With these notations, we are able to introduce the definition of $(M,\alpha)$-short-range dependence conditions, presented as follows:

\begin{definition}[$(M,\alpha)$-short-range dependence]
    For two given real numbers $M,\alpha>1,$ suppose that the random variables $\mathbf{z}_t, t\in\mathbf{Z}$ satisfy equation \eqref{eq.def_z_t}. Furthermore, we assume that $\mathbf{E}\left[\mathbf{z}_t\right] = 0,$  $\sup_{t\in\mathbf{Z};\ \mathbf{a}\in\mathbf{R}^{d_1+d_2},\ \vert\mathbf{a}\vert_2 = 1}\left\Vert\mathbf{a}^\top\mathbf{z}_{t}\right\Vert_M = O(1),$ and $\sup_{s\geq 0}(1 + s)^\alpha\sum_{q = s}^\infty\delta_q = O(1).$ Then we say that $\mathbf{z}_t,t\in\mathbf{Z}$ satisfies the $(M,\alpha)$-short-range dependence.
    \label{def.m_alpha_short_range_dependence}
\end{definition}

\begin{remark}
    Discussions of processes of the form \eqref{eq.def_z_t} trace back to a conjecture of N. Wiener, later disproved by \cite{MR2493017}. \cite{MR2172215} introduced a short-range dependence condition for processes of the form \eqref{eq.def_z_t}.  Subsequent studies---such as  \cite{MR2351105, MR3021390, MR3466186}---have applied this condition to derive distributional results for various statistics. However,  \cite{MR2172215} assumed that the time series $\mathbf{z}_t$ is strictly stationary, i.e., the joint distribution of $\mathbf{z}_{t+1},\cdots,\mathbf{z}_{t+k}$ is identical to the joint distribution of $\mathbf{z}_{1},\cdots,\mathbf{z}_{k}$ for any $k\geq 1$ and any $t\in\mathbf{Z},$ as in Definition 1.3.3 of \cite{MR1093459}.

   More recent studies---including  \cite{MR3174655, MR4998106, MR4829492, MR3718156, MR4206676, MR3779697}---have extended the dependence conditions of \cite{MR2172215} to accommodate nonstationary and high-dimensional time series. Definition \ref{def.m_alpha_short_range_dependence} also generalizes the conditions of \cite{MR2172215} to the setup of high-dimensional nonstationary time series, but it
   imposes moment bounds on linear combinations $\boldsymbol{a}^\top \mathbf{z}_t,$ rather than on individual entries. This formulation is motivated by the test statistic involving inner products of observations. Related conditions appeared in \cite{MR2604697, zhang2025anova}, where quadratic-type test statistics were considered.

The literature has offered alternative  short-range dependence conditions, including various mixing conditions \cite{MR1312160} and near-epoch dependence (NED) \cite{MR4602904}. The process defined in \eqref{eq.def_z_t} can be viewed as a special case of an NED process, where the innovations $e_t$ are independent rather than mixing. As noted in \cite{MR2172215}, verifying Definition \ref{def.m_alpha_short_range_dependence} is comparatively simple relative to mixing conditions, as it only requires bounding moments of random variables, whereas mixing conditions typically involve controlling joint probabilities. On the other hand, Definition \ref{def.m_alpha_short_range_dependence} relies on a nonlinear moving-average representation of the observations, while mixing conditions generally do not impose a specific functional form on the data-generating process.
\end{remark}

\begin{example}
    Suppose a nonstationary linear process 
    $
        \mathbf{z}_t=  \sum_{j = 0}^\infty \mathbf{A}_{t,j}\boldsymbol{e}_{t-j},
    $
    where $\boldsymbol{e}_j\in\mathbf{R}^{d_1 + d_2}$ are independent random vectors (that is, we choose $\mathbf{H}$ as $ \mathbf{R}^{d_1+d_2}$ with Borel $\sigma$-field) with $\mathbf{E}\left[\boldsymbol{e}_j\right] = 0$ and $\sup_{t\in\mathbf{Z};\ \mathbf{a}\in\mathbf{R}^{d_1+d_2},\ \vert\mathbf{a}\vert_2 = 1}\left\Vert\mathbf{a}^\top\mathbf{e}_{t}\right\Vert_M = O(1)$ (like those in \cite{MR2604697}). We further assume that $\sup_{t\in\mathbf{Z}}\left\vert \mathbf{A}_{t,j}\right\vert_2\leq \frac{C}{(1 + j)^{\alpha + 1}}$ for any $j\geq 0.$ With these assumptions, for any vector $\mathbf{a}\in\mathbf{R}^{d_1 + d_2}$ with $\vert \mathbf{a}\vert_2 = 1,$ we have $\mathbf{E}\left[\mathbf{z}_t\right] = \sum_{j  = 0}^\infty\mathbf{A}_{t,j}\mathbf{E}\left[\boldsymbol{e}_{t-j}\right] = 0,$ 
    \begin{align*}
        \left\Vert \mathbf{a}^\top\mathbf{z}_t\right\Vert_M\leq \sum_{j = 0}^\infty\left\Vert\mathbf{a}^\top \mathbf{A}_{t,j}\boldsymbol{e}_{t-j}\right\Vert_M\leq C\sum_{j = 0}^\infty\left\vert \mathbf{A}_{t,j}^\top\mathbf{a}\right\vert_2\leq \sum_{j = 0}^\infty\frac{C_1}{(1 + j)^{1+\alpha}}\leq C_2,
    \end{align*}
    and for any $j\geq 0,$
    \begin{align*}
        \left\Vert
        \mathbf{a}^\top\left(\mathbf{z}_t - \mathbf{z}_{t,j}\right)
        \right\Vert_M = \left\Vert
        \mathbf{a}^\top\mathbf{A}_{t,j}\left(\mathbf{e}_{t-j} - \mathbf{e}_{t-j}^\dagger\right)
        \right\Vert_M\leq 2\left\Vert
        \mathbf{a}^\top\mathbf{A}_{t,j}\mathbf{e}_{t-j}
        \right\Vert_M\leq \frac{C}{(1 + j)^{1 + \alpha}},
    \end{align*}
    which implies 
    $$
    \sum_{j = s}^\infty\delta_q\leq \sum_{j = s}^\infty \frac{C}{(1 + j)^{1 + \alpha}}\leq \frac{C_1}{(1+s)^\alpha},
    $$
    which verifies the $(M,\alpha)$-short-range dependence of $\mathbf{z}_t.$
\end{example}

Theorem \ref{theorem.linear_and_quadratic} presents concentration inequalities in moments of linear combinations and quadratic forms of $\mathbf{z}_t,$ which will serve as the theoretical foundation of the distributional results for the test statistics in later sections. In the following parts of this section, for any $t\in\mathbf{Z}$ and $s\geq 0,$ we define $\mathcal{F}_{t,s}$ as the $\sigma$-field generated by $e_t,e_{t-1},\cdots,e_{t-s}.$

\begin{theorem}
    \label{theorem.linear_and_quadratic}
    Suppose $\mathbf{z}_t, t\in\mathbf{Z}$ satisfy Definition \ref{def.m_alpha_short_range_dependence} with $M > 8$ and $\alpha > 8,$ and the dimension $d = d_1 + d_2 = O(T),$ where $T$ denotes the sample size. Then there exists a constant $C$ independent of the sample size $T$ and the dimension $d$ such that for any real-value vectors $\mathbf{a}_t\in\mathbf{R}^d, t = 1,\cdots,T$ and any $s\geq 0,$
    \begin{equation}
    \begin{aligned}
        \left\Vert
        \sum_{t = 1}^T \mathbf{a}_t^\top\mathbf{z}_t
        \right\Vert_M & \leq C\sqrt{\sum_{t = 1}^T\vert\mathbf{a}_t\vert^2_2},\\  
        \text{and}\quad \left\Vert
        \sum_{t = 1}^T \mathbf{a}_t^\top\left(\mathbf{z}_t - \mathbf{E}\left[\mathbf{z}_t\mid\mathcal{F}_{t,s}\right]\right)
        \right\Vert_M&\leq \frac{C}{(1+s)^\alpha}\sqrt{\sum_{t = 1}^T\vert\mathbf{a}_t\vert^2_2}.
    \end{aligned}
        \label{eq.linear_combination}
    \end{equation}
    Furthermore, assume that two integer bandwidths $1<\lambda_1<\lambda_2 <T$ satisfy $\lambda_1\asymp T^{\kappa_1}, $ $\lambda_2\asymp T^{\kappa_2},$ where $1 > \kappa_2 >\kappa_1 > \frac{5}{2(\alpha - 2)}, $ then for any real symmetric matrices $\mathbf{A}_{t_1t_2},$ we have 
    \begin{equation}
    \begin{aligned}
        &\left\Vert\sum_{t_1 = 1 + \lambda_1}^T\sum_{t_2 = (t_1 - \lambda_2)\vee 1}^{t_1 - \lambda_1}\left(\mathbf{z}_{t_1}^\top\mathbf{A}_{t_1t_2}\mathbf{z}_{t_2} - \mathbf{E}\left[\mathbf{z}_{t_1}^\top\mathbf{A}_{t_1t_2}\mathbf{z}_{t_2}\right]\right)\right\Vert_{M/2}\\
        &= O\left(\sqrt{d\sum_{t_1 = 1 + \lambda_1}^T\sum_{t_2 = (t_1 - \lambda_2)\vee 1}^{t_1 - \lambda_1}\left\vert \mathbf{A}_{t_1t_2}\right\vert^2_2}\right).
    \end{aligned}
        \label{eq.product_quadratic_forms}
    \end{equation}
    Furthermore, there exists a constant $C>0$ such that for any integer $S > 2\lambda_2,$  
    \begin{equation}
        \begin{aligned}
            &\left\Vert\sum_{t_1 = 1 + \lambda_1}^T\sum_{t_2 = (t_1 - \lambda_2)\vee 1}^{t_1 - \lambda_1}\left(\mathbf{z}_{t_1}^\top\mathbf{A}_{t_1t_2}\mathbf{z}_{t_2} - \mathbf{E}\left[\mathbf{z}_{t_1}^\top\mathbf{A}_{t_1t_2}\mathbf{z}_{t_2}\mid\mathcal{F}_{t_1, S}\right]\right)\right\Vert_{M/2}
            \leq  \frac{CT\sqrt{T}a^\dagger}{(S - \lambda_2)^{\alpha - 2}},
        \end{aligned}
        \label{eq.truncate_whole}
    \end{equation}
    where $a^\dagger = \max_{t_1 = 1+\lambda_1,\cdots, T;\ t_2 = (t_1 - \lambda_2)\vee 1,\cdots, t_1  -\lambda_1}\vert \mathbf{A}_{t_1t_2}\vert_2.$
\end{theorem}

We note that the moment bound in \eqref{eq.product_quadratic_forms} remains valid for mutually independent $\mathbf{z}_t$ with independent entries, as established in \cite{MR0133849}; see also Remark \ref{remark.independent_linear_quadratic} in Appendix \ref{section.proof_linear_quadratic}. Therefore, Theorem \ref{theorem.linear_and_quadratic} extends the results of \cite{MR0133849} to settings involving dependent random variables.

The differing bounds in \eqref{eq.product_quadratic_forms} and \eqref{eq.truncate_whole} reflect the cost of approximating the original quadratic forms by sums of $m$-dependent random variables (see Definition 6.4.3 of \cite{MR2839251}); the technical details are provided in \eqref{eq.second_moment_truncate_part} in the online supplement. To provide some intuition, suppose that $\mathbf{z}_t$ is $m$-dependent with $m < S.$ In this case, $\mathbf{z}_{t_1}$ is $\mathcal{F}_{t_1, S}$-measurable, and 
$$
\mathbf{E}\left[\mathbf{z}_{t_1}^\top\mathbf{A}_{t_1t_2}\mathbf{z}_{t_2}\mid\mathcal{F}_{t_1, S}\right]
= \mathbf{z}_{t_1}^\top\mathbf{A}_{t_1t_2}\mathbf{E}\left[\mathbf{z}_{t_2}\mid\mathcal{F}_{t_1, S}\right]
= \mathbf{z}_{t_1}^\top\mathbf{A}_{t_1t_2}\mathbf{E}\left[\mathbf{z}_{t_2}\right]= 0
$$
if $S < t_1 - t_2.$ This makes $\mathbf{E}\left[\mathbf{z}_{t_1}^\top\mathbf{A}_{t_1t_2}\mathbf{z}_{t_2}\mid\mathcal{F}_{t_1, S}\right]$ a poor approximation when the time lag $t_1 - t_2$ is significantly larger than $S.$ This crude approximation increases the  approximation error of the summation $\sum_{t_1 = 1 + \lambda_1}^T\sum_{t_2 = (t_1 - \lambda_2)\vee 1}^{t_1 - \lambda_1}\mathbf{E}\left[\mathbf{z}_{t_1}^\top\mathbf{A}_{t_1t_2}\mathbf{z}_{t_2}\mid\mathcal{F}_{t_1, S}\right]$  in \eqref{eq.truncate_whole}.

\begin{corollary}
From equation \eqref{eq.linear_combination}, suppose $t_1 > t_2,$
\begin{align*}
     \left\vert \mathbf{E}\left[\mathbf{x}_{t_1}^\top\mathbf{x}_{t_2}\right]\right\vert &\leq \sum_{i = 1}^{d_1}\left\vert
     \mathbf{E}\left[\left(\mathbf{z}_{t_1}^{(i)} - \mathbf{E}\left[\mathbf{z}_{t_1}^{(i)}\mid\mathcal{F}_{t_1, t_1 - t_2 - 1}\right]\right)\mathbf{z}_{t_2}^{(i)}\right]
     \right\vert\\
     &\leq \sum_{i = 1}^{d_1}\left\Vert \mathbf{z}_{t_2}^{(i)}\right\Vert_M\left\Vert \mathbf{z}_{t_1}^{(i)} - \mathbf{E}\left[\mathbf{z}_{t_1}^{(i)}\mid\mathcal{F}_{t_1, t_1 - t_2 - 1}\right] \right\Vert_M\leq \frac{Cd_1}{(t_1 -  t_2)^\alpha},
\end{align*}
and 
\begin{align*}
    \left\vert \mathbf{E}\left[\mathbf{y}_{t_1}^\top\mathbf{y}_{t_2}\right]\right\vert\leq \sum_{i = d_1 + 1}^d \left\vert
     \mathbf{E}\left[\left(\mathbf{z}_{t_1}^{(i)} - \mathbf{E}\left[\mathbf{z}_{t_1}^{(i)}\mid\mathcal{F}_{t_1, t_1 - t_2 - 1}\right]\right)\mathbf{z}_{t_2}^{(i)}\right]
     \right\vert\leq \frac{Cd_2}{(t_1 -  t_2)^\alpha}.
\end{align*}
Therefore, in a heuristic sense, Definition \ref{def.m_alpha_short_range_dependence} ensures that the temporal correlations of $\mathbf{x}_t$ and $\mathbf{y}_t$ decrease at a polynomial rate. The parameter $\alpha$ quantifies the strength of dependence: a smaller $\alpha$ corresponds to a slower decay of covariances, indicating potentially strong dependence.
    \label{corollary.covariances}
\end{corollary}

\subsection{Asymptotic distribution under $H_0$}
By leveraging concentration inequalities developed in Section \ref{section.malpha_short_range}, this section establishes the asymptotic distribution of the test statistic and the validity of the bootstrap algorithm under the null hypothesis, which controls the Type-I error of the test.

Theorem \ref{theorem.distributional_under_H0} builds on ideas from \cite{MR3161448, MR3779697, MR3992401} in controlling approximation errors, but differs from these works in that it derives asymptotic distributional approximations for quadratic forms of dependent data.

\begin{theorem}
\label{theorem.distributional_under_H0}
    Suppose Assumptions \ref{assumption.independent_short_range_structure}, \ref{assumption.lambda_choices}, \ref{assumption.variance} of Appendix \ref{section.technical_details_under_H0} hold true. Then we have 
    \begin{equation}
        \sup_{x\in\mathbf{R}}\left\vert
        \mathbf{pr}\left(\widehat{R}\leq x\right) - \mathbf{pr}\left(\epsilon\leq x\right)
        \right\vert = o(1),
        \label{eq.Prob_close}
    \end{equation}
    where the random variable $\epsilon$ has normal distribution with $\mathbf{E}\left[\epsilon\right] = 0$ and $\mathrm{Var}\left(\epsilon\right) = \mathrm{Var}\left(\widehat{R}^\dagger\right),$ and $\widehat{R}^\dagger$ is the canonical statistic defined in Assumption \ref{assumption.variance} of Appendix \ref{section.technical_details_under_H0}.
\end{theorem}
Theorem \ref{theorem.distributional_under_H0} further justifies the claim in Remark \ref{remark.underH0}: For appropriate choices of $\lambda_1,$ the test statistic has an asymptotic normal distribution with mean zero. The asymptotic normality also makes it reasonable to sample jointly normal random variables in Algorithm \ref{algorithm.dependent_wild_bootstrap}.

The validity of Algorithm \ref{algorithm.dependent_wild_bootstrap} hinges on the fact that the distribution of the bootstrapped statistics $\widehat{R}^*_b$ in the bootstrap world accurately approximates the distribution of $\widehat{R}.$ Specifically, define the ``probability in the bootstrap world''---that is, the conditional probability conditional on all observations---as  $\mathbf{pr}^*\left(\cdot\right) = \mathbf{pr}\left(\cdot\mid \mathbf{z}_t, t= 1,\cdots,T\right).$ According to Theorem 1.2.1 of \cite{MR1707286}, it suffices to establish that 
\begin{equation}
    \begin{aligned}
        \sup_{x\in\mathbf{R}}\left\vert \mathbf{pr}^*\left(\widehat{R}^*_b\leq x\right) - \mathbf{pr}\left(\epsilon\leq x\right)\right\vert = o_p(1),
    \end{aligned}
    \label{eq.bootstrap_validity}
\end{equation}
where $\epsilon$ is defined in Theorem \ref{theorem.distributional_under_H0}. We demonstrate this result in Theorem \ref{theorem.consistency_bootstrap}.

\begin{theorem}
    \label{theorem.validity_of_bootstrap_algorithm}
    Suppose Assumptions \ref{assumption.independent_short_range_structure}, \ref{assumption.lambda_choices}, \ref{assumption.variance}, and \ref{assumption.for_variance_estimation} of Appendix \ref{section.technical_details_under_H0} hold true. Then \eqref{eq.bootstrap_validity} holds true.
    \label{theorem.consistency_bootstrap}
\end{theorem}

\section{Numerical simulations}
\label{section.numerical_simulation}
\subsection{Implementation details}
The proposed testing procedure involves four bandwidth parameters: $\lambda_1$, $\lambda_2$, $h$, and $\mathcal{K}_T$. Among these bandwidths, $h$ determines the maximal lag up to which mutual dependence is captured;  as discussed in \cite{MR1423878, MR1997176, MR2472050, MR3416594}. The parameter $\mathcal{K}_T$ 
corresponds to the bandwidth used in HAC estimation, which is common in the literature; see, for example, \cite{MR2656050, MR4134802, MR4726969}. The literature has proposed various methods for selecting $\mathcal{K}_T$ in different kinds of bootstrap settings \cite{MR1366282, MR2211092, MR2380557}. In particular, the original dependent wild bootstrap of \cite{MR2656050} leveraged \cite{MR2041534} for bandwidth selection.

In contrast, $\lambda_1$ and $\lambda_2$ are associated with the proposed test statistic $\widehat{R}.$ The parameter $\lambda_1$ serves to mitigate the bias induced by temporal dependence, while $\lambda_2$ controls the dependence strength in the products of time series data. Algorithm \ref{algorithm.selection_lambda_1} provides  a data-driven method for selecting $\lambda_1.$  We then set  $\lambda_2 = \iota\lambda_1,$ with different  scale parameters $\iota.$ Numerical experiments indicate that the size of the proposed test remains stable across a wide range of $\lambda_2.$ However, the power can be affected if $\lambda_2$ is chosen to be excessively small or excessively large.

The design of Algorithm \ref{algorithm.selection_lambda_1} is motivated by the need to control the temporal bias of the test statistic.
To illustrate this, suppose the null hypothesis holds true and assume that $\mathbf{z}_t,t\in\mathbf{Z}$ is  weak stationary. From equation \eqref{eq.bias_Ot} in the online supplement, if we set $h = 0,$ then 
$$
    \mathcal{U}\left\vert \mathbf{E}\left[ O_t\right]\right\vert 
    \leq \frac{1}{2} \sum_{t_1 = (t - \lambda_2)\vee 1}^{t - \lambda_1}\left(\left(\mathbf{E}\left[\mathbf{x}_{t}^\top\mathbf{x}_{t_1}\right]\right)^2 + \left(\mathbf{E}\left[\mathbf{y}_{t}^\top\mathbf{y}_{t_1}\right]\right)^2\right).
$$
Therefore, Algorithm \ref{algorithm.selection_lambda_1} chooses the smallest $\lambda_1$ such that   $\left(\mathbf{E}\left[\mathbf{x}_{t}^\top\mathbf{x}_{t_1}\right]\right)^2 + \left(\mathbf{E}\left[\mathbf{y}_{t}^\top\mathbf{y}_{t_1}\right]\right)^2$ is small compared to the total sum of squared sample autocovariances. This choice ensures a small temporal bias. Although $\mathbf{E}\left[\mathbf{x}_{t}^\top\mathbf{x}_{t_1}\right]$ and $\mathbf{E}\left[\mathbf{y}_{t}^\top\mathbf{y}_{t_1}\right]$ are not observable, under stationarity they can be estimated by the sample autocovariances $\widehat{\sigma}_{t-t_1}^{(x)}$ and $\widehat{\sigma}_{t-t_1}^{(y)}$ in Algorithm \ref{algorithm.selection_lambda_1}, leading to the design of Algorithm \ref{algorithm.selection_lambda_1}.

\begin{algorithm}[H]
\caption{Selection of $\lambda_1.$}
\label{algorithm.selection_lambda_1}
\SetAlgoLined
\KwIn{Observations $\mathbf{z}_t = \left(\mathbf{x}_t^\top, \mathbf{y}_t^\top\right)^\top\in\mathbf{R}^{d_1 + d_2}$ for $t = 1,\cdots,T,$ ratio threshold $\mathcal{F}.$ } 
    
\hspace{1in}

Derive the estimators $\widehat{\sigma}_a^{(x)} = \left(\sum_{t = a + 1}^T\mathbf{x}^\top_{t}\mathbf{x}_{t - a}\right) / T$ and $\widehat{\sigma}_a^{(y)} = \left(\sum_{t = a + 1}^T\mathbf{y}^\top_{t}\mathbf{y}_{t - a}\right) / T$
for $a = 0,1,\cdots,  T-1.$

Choose $\lambda_1$ as the smallest integer such that 
\begin{align*}
    \frac{\sum_{a = \lambda_1}^{T-1}\left(\widehat{\sigma}_a^{(x)2} + \widehat{\sigma}_a^{(y)2}\right)}{\sum_{a =  0}^{T-1}\left(\widehat{\sigma}_a^{(x)2} + \widehat{\sigma}_a^{(y)2}\right)}
    \leq \mathcal{F}
\end{align*}

\KwOut{The bandwidth $\lambda_1$.}

\noindent \textbf{Remark: }  We set the ratio threshold to $\mathcal{F} = 0.5\%$ in our experiments.
\end{algorithm}

This experiment sets the scale parameter $\iota$ to $\lfloor 2T^{0.05}\rfloor, \lfloor 2T^{0.1}\rfloor, \lfloor 2T^{0.2}\rfloor,$ with $\lfloor x\rfloor$ denotes the largest integer that is smaller than or equal to $x.$ Notably, the choice $\lfloor 2T^{0.2}\rfloor$ violates Assumption \ref{assumption.for_variance_estimation} in the online supplement. We intentionally include this ratio to  assess the robustness of the proposed test to violations of the underlying assumptions. Following \cite{MR3416594}, we consider three choices of $h = \lfloor \log(T)\rfloor, \lfloor6T^{0.1}\rfloor, \lfloor 6T^{0.3}\rfloor.$

This experiment adopts the method proposed in \cite{MR2041534}, implemented in the R package \texttt{blocklength}, to select $\mathcal{K}_T$. Since the procedure in \cite{MR2041534} is designed for univariate time series, we apply it to the scalar series $\mathbf{z}_t^\top \mathbf{z}_t$ for $t = 1, \ldots, T$.

\subsection{Simulation data}
\label{section.simulation_data_res}
The numerical experiment considers three data-generating processes of $\mathbf{z}_t = (\mathbf{x}_t^\top, \mathbf{y}_t^\top)^\top \in \mathbf{R}^{d_1 + d_2},$  namely
\begin{itemize}
    \item nonstationary autoregressive (nAR(1)) process: 
    \begin{align*}
        \mathbf{z}_t = 
        \begin{cases}
        \boldsymbol{\varepsilon}_t + \mathbf{A}\mathbf{z}_{t-1}  & \text{if}\quad t\mod 2 = 1,\\
        \boldsymbol{\varepsilon}_t  + \mathbf{B}\mathbf{z}_{t-1} & \text{if}\quad t \mod  2 = 0.
        \end{cases}
    \end{align*}
    \item nonstationary autoregressive moving-average (nARMA(1,1)) process:
    \begin{align*}
        \mathbf{z}_t = 
        \begin{cases}
        \mathbf{A}\mathbf{z}_{t-1} + \boldsymbol{\varepsilon}_t +\mathbf{A}\boldsymbol{\varepsilon}_{t-1} & \text{if}\quad t\mod 2 = 1,\\
        \mathbf{B}\mathbf{z}_{t-1} + \boldsymbol{\varepsilon}_t  + \mathbf{B}\boldsymbol{\varepsilon}_{t-1} & \text{if}\quad t \mod  2 = 0.
        \end{cases}
    \end{align*}
    \item nonstationary threshold ARMA (nTARMA(1,1)) process :
    \begin{align*}
        \mathbf{z}_t = 
        \begin{cases}
         \boldsymbol{\varepsilon}_t + \mathbf{A}\mathbf{z}_{t - 1} + \mathbf{A}\boldsymbol{\varepsilon}_{t-1}, &\text{if} \quad t \mod  2 = 1\quad\text{and}\quad \mathbf{z}_{t - 2}^{(1)}\leq 0,\\
         \boldsymbol{\varepsilon}_t - \mathbf{A}\mathbf{z}_{t - 1} - \mathbf{A}\boldsymbol{\varepsilon}_{t-1}, &\text{if} \quad t \mod  2 = 1\quad\text{and}\quad \mathbf{z}_{t - 2}^{(1)} > 0,\\
        \boldsymbol{\varepsilon}_t + \mathbf{B}\mathbf{z}_{t - 1} + \mathbf{B}\boldsymbol{\varepsilon}_{t-1}, &\text{if} \quad t \mod  2 = 0\quad\text{and}\quad \mathbf{z}_{t - 2}^{(1)}\leq 0,\\
         \boldsymbol{\varepsilon}_t - \mathbf{B}\mathbf{z}_{t - 1} - \mathbf{B}\boldsymbol{\varepsilon}_{t-1}, &\text{if} \quad t \mod  2 = 0\quad\text{and}\quad \mathbf{z}_{t - 2}^{(1)} > 0.\\
        \end{cases}
    \end{align*}
\end{itemize}
The innovations satisfy $\boldsymbol{\varepsilon}_t = \boldsymbol{\Pi}\boldsymbol{\eta}_t,$ where $\boldsymbol{\eta}_{t} = (\boldsymbol{\eta}_{t}^{(1)},\cdots, \boldsymbol{\eta}_{t}^{(d_1+d_2)})^\top,$ and each $\boldsymbol{\eta}_{t}^{(j)}, t\in\mathbf{Z},j = 1,\cdots, d_1 + d_2$ satisfy independent $t$-distribution with 20 degrees of freedom. We consider $t$-distribution because it has finite polynomial moments but lacks  exponential moments, which satisfies the setup of Section \ref{section.asymptotic}. The values of the matrices $\mathbf{A},\mathbf{B},\boldsymbol{\Pi}$ are postponed to Section \ref{section.values_parameter_matrices} of the online supplement.

\subsection{Simulation results}
\label{section.simulation_verify}
We choose the sample size $T = 300, 500, 700$ and the dimensions $d_1 = 1.2T,  d_2 = 1.5T.$ The numerical results for $T = 500$ are reported in Table \ref{size_power_performance}, and other results are postponed to Table \ref{size_power_performance_sample_size} of the online supplement. We compare our proposed method with the HSIC-based testing procedure of \cite{MR4270386}. As shown in  Table \ref{size_power_performance}, the size of the proposed hypothesis testing procedure is insensitive to different choices of $\lambda_2$ and $h.$ In particular, even when the choices of $\lambda_2$ and $h$ are chosen to be of higher order than required by Assumption \ref{assumption.variance}, the test sizes remain close to the nominal level.

However, appropriate  choices of $h$ and $\lambda_2$ can improve the power performance. From Table \ref{size_power_performance}, when $\lambda_2$ is  set to a moderate level ($\lfloor 2T^{0.1}\rfloor\lambda_1$ in the experiment), the proposed hypothesis testing procedure achieves the most favorable power performance overall compared to other choices. On the other hand, selecting values of $\lambda_2$  that are either too small or too large leads to a deterioration in power. Numerical results also indicate that adopting a relatively large bandwidth $h$ is beneficial for improving power performance.

Table \ref{size_power_performance} further highlights the necessity of applying independence tests directly to the time series data $\mathbf{z}_t,$ rather than to fitted residuals, in nonstationary settings. The issue arises from potential model incompatibility and is further illustrated by comparison with Table \ref{table.stationary_performance}, where the data-generating process follows a stationary AR(1) model of the form $\mathbf{z}_t = \mathbf{A}\mathbf{z}_{t-1} + \boldsymbol{\epsilon}_t.$

Specifically,  when the time series exhibits a stationary autocovariance structure, practitioners can employ parametric models---such as vector ARMA processes---to remove autocovariances by computing fitted residuals. In contrast, for nonstationary time series with evolving covariance structures, few models are capable of adequately capturing the underlying dependence. As a result, model misspecification induces residual autocorrelation, which in turn deteriorates the size and power performance of the HSIC-based testing procedure proposed in \cite{MR4270386}.

When the data-generating process exhibits nonlinearity, the size of the test remains stable, whereas its power is affected. Table \ref{table.sample_imbalance} further illustrates that relatively large sample sizes are required to achieve satisfactory power under nonstationarity. These findings suggest the need for developing independence tests that are suitable for capturing nonlinear dependence structures, for example by exploiting reproducing kernel methods as in \cite{MR2249882}, which we leave for future research.

\textbf{Necessity of selecting a reasonably large $\lambda_1.$} Equation \eqref{eq.E_R_hat} suggests that, under $H_0,$ choosing a positive $\lambda_1$ removes the bias induced by temporal dependence. This theoretical implication is further supported by the numerical results in Table \ref{table.size_power_different_lambda}, where we perform hypothesis testing with different values of $\lambda_1.$
When $\lambda_1 = 0, $ the bias arising from temporal dependence dominates the stochastic error, causing the test to reject $H_0$ regardless of the underlying hypothesis. When $\lambda_1$ is nonzero but relatively small, the temporal dependence bias still affects the test size---it even partially offsets the bias induced under $H_1,$ resulting in power that can be lower than the size.

\begin{table}[htbp]
	\caption{Size and power performance of the proposed test procedure. The sample size is 500, and the nominal size is 5\%. We repeat the experiment 100 times to evaluate the size and power.   } 
    \scriptsize
{	\begin{tabular}{ccc|rrrr || rrrr}
\hline\hline
		& & & \multicolumn{8}{c}{Different choices of $\delta$}\\
        \hline
        & & & \multicolumn{4}{c||}{Our work}  & \multicolumn{4}{c}{HSIC-based test}\\
	DGP	& $h$ &  $(\lambda_1,\lambda_2, \mathcal{K}_T)$ & 0 (null)  &  0.075  & 0.100 & 0.125 & 0 (null)  &  0.075  & 0.100 & 0.125\\
    \hline
    \multirow{9}{*}{nAR(1)} & \multirow{3}{*}{6} & (2, 4, 4.2) & 6\% &   10\% & 62\% & 99\% & \multirow{3}{*}{1\%} &  \multirow{3}{*}{1\%} & \multirow{3}{*}{0\%} & \multirow{3}{*}{0\%}\\
        &     &   (2, 6, 4.2) & 6\% & 15\% & 65\% & 100\% \\
        &     &   (2, 12, 6.4) & 6\%  & 12\% & 58\% & 99\% \\
        \cline{2-11}
        &  \multirow{3}{*}{11}   &   (2, 4, 4.4) & 6\% & 9\% & 60\% & 100\%  &  \multirow{3}{*}{1\%} &  \multirow{3}{*}{0\%} & \multirow{3}{*}{1\%} & \multirow{3}{*}{0\%}\\
        &                       &   (2, 6, 4.4) &  5\% & 18\% & 67\% & 100\%\\
        &                       &   (2, 12, 4.1) & 6\% & 19\% & 66\% & 100\%\\
        \cline{2-11}
        &   \multirow{3}{*}{38} &   (2, 4, 4.6)  &  7\% & 18\% & 86\% & 100\% & \multirow{3}{*}{0\%} & \multirow{3}{*}{0\%} & \multirow{3}{*}{0\%} & \multirow{3}{*}{0\%}\\
        &                       &   (2, 6, 6.7)  &  4\% & 31\% & 92\% & 100\%\\
        &                       &   (2, 12, 4.2) &  4\% & 13\% & 76\% & 100\%\\
        \hline\hline
    \multirow{9}{*}{nARMA(1,1)}  & \multirow{3}{*}{6} &  (3, 6, 7.1) & 4\% & 27\% & 90\% & 100\% & \multirow{3}{*}{1\%} & \multirow{3}{*}{0\%} & \multirow{3}{*}{0\%} & \multirow{3}{*}{0\%}\\
        &                       &     (3, 9, 7.6)  & 3\% & 32\% & 92\% & 100\%\\
        &                       &     (3, 18, 6.7) & 3\% & 37\% & 86\% & 100\%\\    
        \cline{2-11}
        &    \multirow{3}{*}{11} &     (2, 4, 7.1) & 5\% & 36\% & 95\% & 100\%  & \multirow{3}{*}{1\%} & \multirow{3}{*}{0\%} & \multirow{3}{*}{0\%} & \multirow{3}{*}{0\%}\\    
        &                       &      (2, 6, 7.5) & 8\% & 36\% & 92\% & 100\%\\
        &                       &      (2, 12, 7.3) & 3\% & 28\% & 94\% & 100\% \\
        \cline{2-11}
        &    \multirow{3}{*}{38} & (2, 4, 7.2) & 4\% & 46\% & 98\% & 100\%   & \multirow{3}{*}{0\%} & \multirow{3}{*}{0\%} & \multirow{3}{*}{0\%} & \multirow{3}{*}{0\%}\\
        &                        & (2, 6, 7.3) & 3\% & 46\% & 100\% & 100\% \\
        &                        &  (2, 12, 7.4) & 4\% & 47\% & 98\% & 100\%\\
\hline\hline
\multirow{9}{*}{nTARMA(1,1)} &  \multirow{3}{*}{6} & (1, 2, 8.2) & 6\% & 2\% & 18\% & 92\% & \multirow{3}{*}{1\%} & \multirow{3}{*}{0\%} & \multirow{3}{*}{0\%} & \multirow{3}{*}{0\%} \\
                       &                     & (1, 3, 4.5) & 5\% & 3\% & 17\% & 95\%\\
                       &                     & (1, 6, 6.0) & 5\% & 6\% & 42\% & 99\%\\
\cline{2-11}
                         &  \multirow{3}{*}{11} & (1, 2, 4.0) & 5\% & 5\% & 14\% & 68\%  & \multirow{3}{*}{0\%} & \multirow{3}{*}{0\%} & \multirow{3}{*}{0\%} & \multirow{3}{*}{0\%}\\
                         &                      & (1, 3, 6.9) & 5\% & 6\% & 10\% & 66\% \\
                         &                      & (1, 6, 7.2) & 3\% & 6\% & 23\% & 83\%\\
                         \cline{2-11}
                         &  \multirow{3}{*}{38} & (1, 2, 6.5) & 5\% & 3\% & 4\% & 21\% & \multirow{3}{*}{1\%} & \multirow{3}{*}{1\%} & \multirow{3}{*}{0\%} & \multirow{3}{*}{0\%}\\
                         & & (1, 3, 4.3) & 7\% & 3\% & 6\% & 26\%\\
                         & & (1, 6, 6.8) &  5\% & 5\% & 5\% & 40\% \\
\hline\hline
	\end{tabular}
}
	\label{size_power_performance}
\end{table}

\begin{table}[htbp]
    \caption{Size and power performance of tests on stationary AR(1) model $\mathbf{z}_t = \mathbf{A}\mathbf{z}_{t-1} + \boldsymbol{\epsilon}_t,$ where  $\mathbf{A}$ is defined as in Section \ref{section.values_parameter_matrices} of the online supplement.  The experiment setup coincides with Table \ref{size_power_performance}. We choose the sample size $T = 500$ and $\iota = \lfloor 2T^{0.1}\rfloor.$}
    \centering
    \scriptsize
  {  \begin{tabular}{cc| rrrr || rrrr}
  \hline\hline
         &  &  \multicolumn{8}{c}{Different choices of $\delta$}  \\
         \hline
         &  &  \multicolumn{4}{c||}{Our work} & \multicolumn{4}{c}{HSIC-based test}\\
         $h$ &  $(\lambda_1,\lambda_2, \mathcal{K}_T)$ & 0 (null)  &  0.075  & 0.100 & 0.125 & 0 (null)  &  0.075  & 0.100 & 0.125\\
         \hline
         6  & (2, 6, 4.5) & 7\% & 14\% & 61\% & 99\% & 5\% & 4\% & 0\% & 1\%\\
         11 & (2, 6, 6.2) & 6\% & 15\% & 70\% & 100\% & 6\% & 10\% & 4\% & 3\% \\
         38 & (2, 6, 6.5) & 5\% & 23\% & 94\% & 100\% & 7\% & 6\% & 6\% & 3\%\\
    \hline\hline
    \end{tabular}
}
    \label{table.stationary_performance}
\end{table}

\begin{table}[htbp]
\caption{Size and power performance with respect to  different $\lambda_1.$  We choose the sample size $T = 500,$ and fix $(h,\lambda_2,\mathcal{K}_T) = (38, 6, 6.7),$ which coincides with the choices of corresponding parameters of nAR(1) model. }
\scriptsize
\centering
{
    \begin{tabular}{c|rrr||rrr||rrr}
    \hline\hline
            &   \multicolumn{9}{c}{Different choices of $\delta$}\\
            \hline
        &   \multicolumn{3}{c||}{$\lambda_1 = 0$} &   \multicolumn{3}{c||}{$\lambda_1 = 1$} &  \multicolumn{3}{c}{$\lambda_1 = 2$}\\
         DGP    &   0 (null)  &  0.075  & 0.100 & 0 (null)  &  0.075  & 0.100 & 0 (null)  &  0.075  & 0.100\\
            \hline
    nAR(1)   &  100\%  & 100\% & 100\% & 82\% & 7\% & 92\% & 3\% & 26\% & 86\%\\
    nARMA(1,1) & 100\%  & 100\% & 100\% & 87\% & 17\% & 100\% & 4\% & 49\% & 100\%\\
    \hline\hline
    \end{tabular}
}
\label{table.size_power_different_lambda}
\end{table}

\section{Applications}
\label{section.applications}

This section illustrates the practical use of the proposed test in two applications. The first facilitates causal discovery by refining causal directions after applying causal discovery algorithms (e.g., PC algorithm in Section 5.4.2 of \cite{spirtes2012causation}). The second applies the test to financial data to detect cross-sector dependence in stock returns.

\textbf{Causal discovery. }   After observing a time series $\mathbf{z}_t = \left(\mathbf{x}_t^\top, \mathbf{y}_t^\top\right)^\top$ for $ t = 1,\cdots,T,$ 
one goal of causal discovery involves determining the direction of causality---whether $\mathbf{x}_t$ causes $\mathbf{y}_t$ or $\mathbf{y}_t$ causes $\mathbf{x}_t$---as studied in \cite{NIPS2008_f7664060, MR4119157}.  \cite{zheng2024causal} demonstrated that 
independence tests could complement causal discovery algorithms, such as the PC algorithm in \cite{spirtes2012causation}, when causal directions were not identifiable.  \cite{NIPS2008_f7664060} proposed a three-step procedure for causal discovery: 1. Test if  $\mathbf{x}_t$ and $\mathbf{y}_t$ are independent; 2. If independence is rejected, fit the model $\mathbf{y}_t = \mathbf{f}\left(\mathbf{x}_t\right) + \boldsymbol{\epsilon}_{t,1},$ estimate $\widehat{\mathbf{f}},$ and test independence between $\mathbf{x}_t$ and the residuals $\widehat{\boldsymbol{\epsilon}}_{t,1} = \mathbf{y}_t - \widehat{\mathbf{f}}(\mathbf{x}_t).$ Failure to reject the null hypothesis suggests the causal direction $\mathbf{x}_t\to \mathbf{y}_t .$  3. Repeat Step 2 with the reversed model $\mathbf{x}_t = \mathbf{g}(\mathbf{y}_t) + \boldsymbol{\epsilon}_{t,2}$. This section applies our method to infer causal directions in time series data. We generate 4 mutually independent nAR(1) processes  $\mathbf{x}_{t,i}$ with $i = 1,\cdots,4.$ After that, define the observations $\mathbf{w}_{t,i}, i = 1,2,3,4$ as 
\begin{equation}
\begin{aligned}
    \mathbf{w}_{t,1}  &= \mathbf{x}_{t,1},\quad \mathbf{w}_{t,2}  = \mathbf{A}^{(x)}\mathbf{w}_{t,1} + \mathbf{x}_{t,2},\quad \mathbf{w}_{t,3}  = \mathbf{B}^{(x)} \mathbf{w}_{t,1}+ \mathbf{x}_{t,3},\\
    \mathbf{w}_{t,4}  &= \mathbf{B}^{(x)} \mathbf{w}_{t,2} + \mathbf{A}^{(x)} \mathbf{w}_{t,3} + \mathbf{x}_{t,4},\quad
\end{aligned}
\label{eq.DGP_system}
\end{equation}
where $\mathbf{A}^{(x)}, \mathbf{B}^{(x)}$ are defined in Section \ref{section.values_parameter_matrices} of the online supplement.
This construction induces the causal graph shown in Figure~\ref{fig.real_causal_graph}.

We apply the PC algorithm as implemented in \cite{zheng2024causal} to recover causal relationships among the observations. Since the algorithm is designed for scalar data, we reduce dimensionality at each node via principal component analysis.
The resulting graph is displayed in Figure \ref{fig.causal_discovery}. Due to temporal dependence and information loss during  dimension reduction, the original PC algorithm includes a spurious edge and fails to correctly identify causal directions.

Following \cite{NIPS2008_f7664060}, we use Lasso to estimate $\mathbf{f}$ and $\mathbf{g},$ and apply Algorithm \ref{algorithm.dependent_wild_bootstrap} to calibrate the direction of causality. The results, reported in Table \ref{table.causation_direction} and Figure \ref{figure.causal_graph}, show that the calibrated directions align with the true causal graph. However, the spurious edge $W_3\to W_2$ cannot be removed because of the presence of the common confounder $W_1;$ in this case, regression adjustment fails to eliminate the induced dependence between $\mathbf{w}_{t,3}$ and $\mathbf{w}_{t,2}.$

\begin{figure}[htbp]
    \centering
    \begin{subfigure}[b]{0.32\textwidth}
        \centering
        \includegraphics[width=\linewidth]{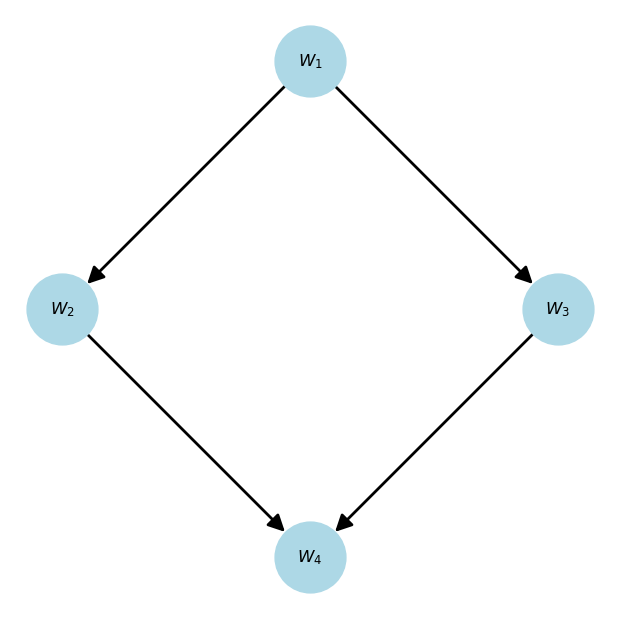}
        \caption{Population causal graph}
        \label{fig.real_causal_graph}
    \end{subfigure}
    \hfill
    \begin{subfigure}[b]{0.32\textwidth}
        \centering
        \includegraphics[width=\linewidth]{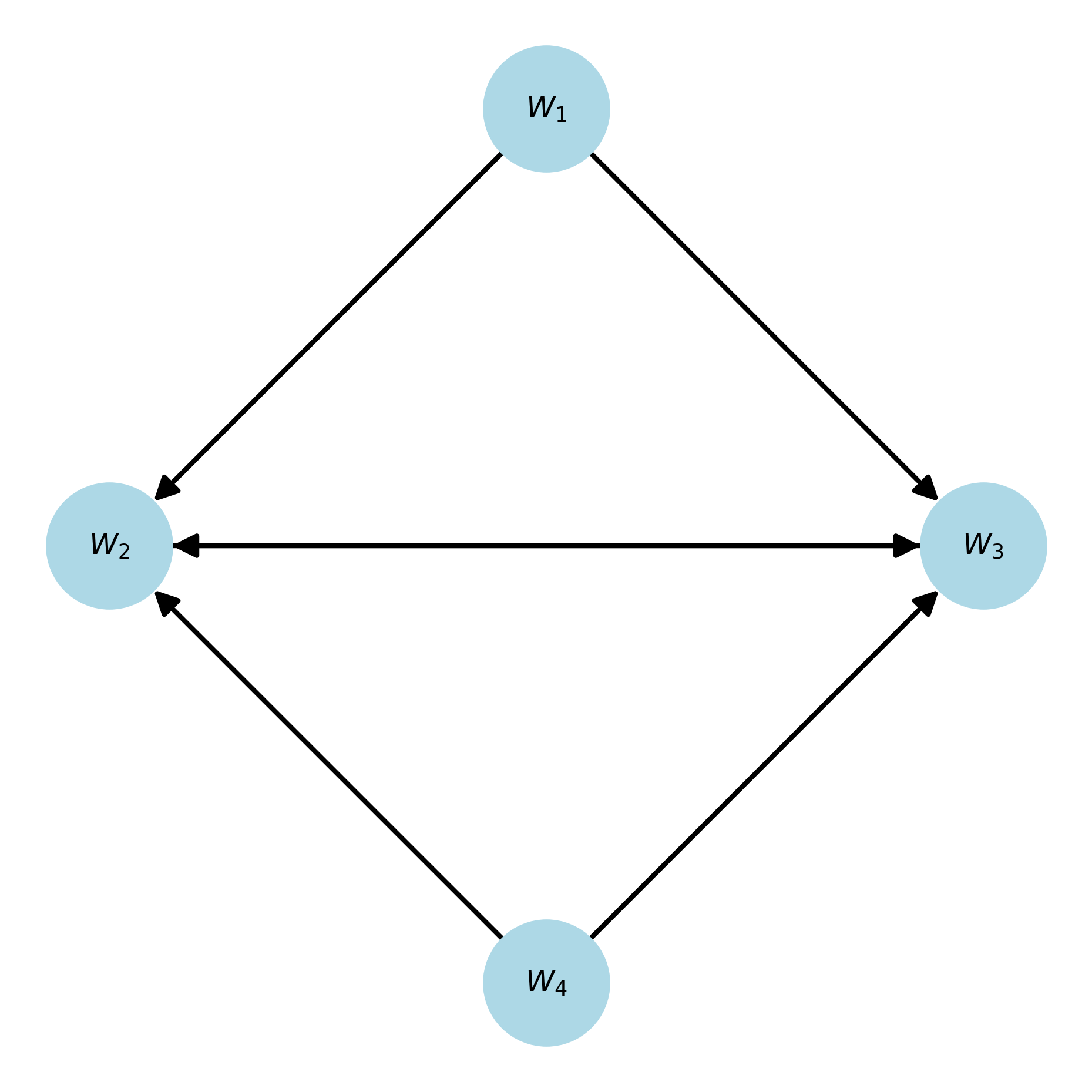}
        \caption{Causal graph by PC algorithm}
       \label{fig.causal_discovery}
    \end{subfigure}
    \begin{subfigure}[b]{0.32\textwidth}
        \centering
        \includegraphics[width=\linewidth]{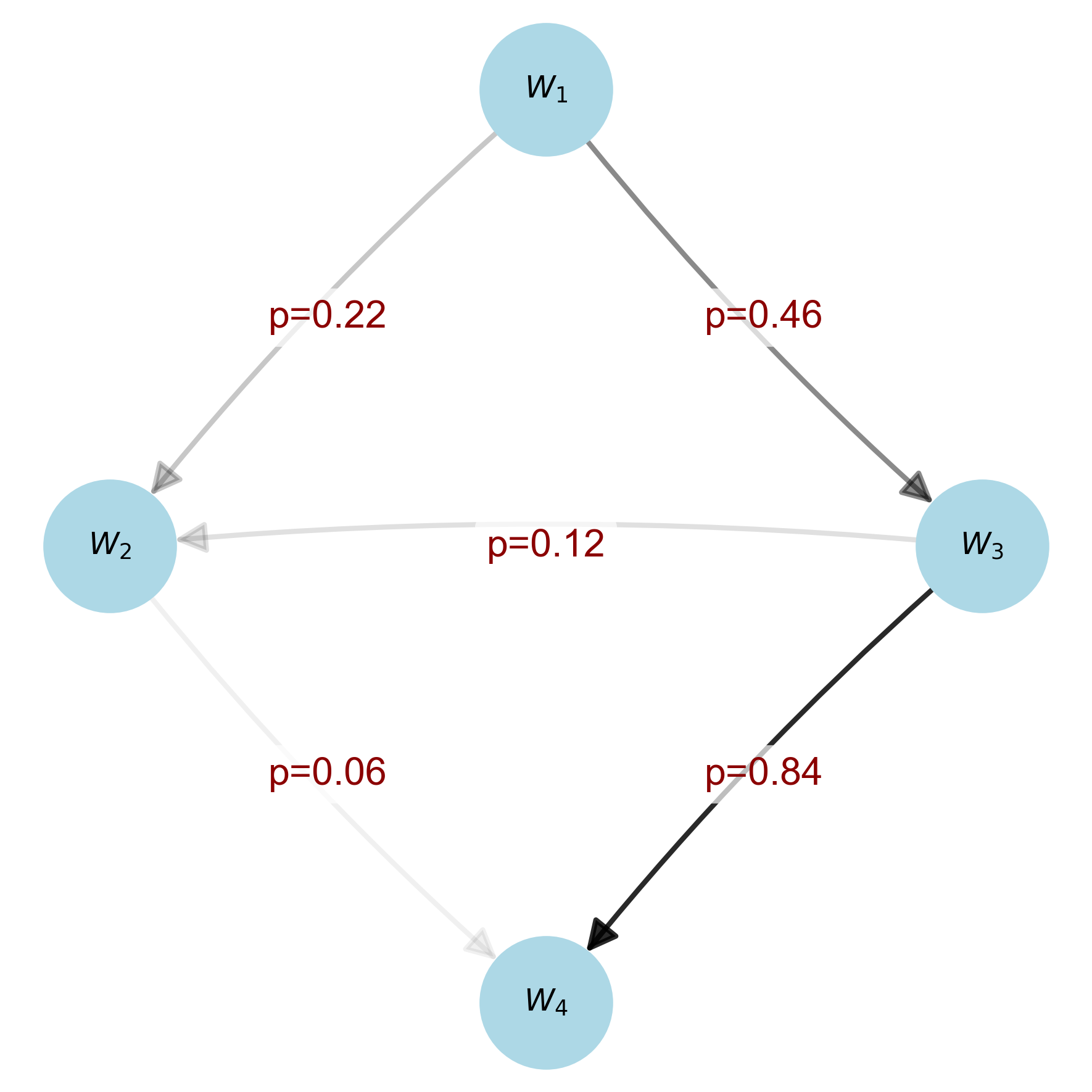}
        \caption{Calibrated causal graph}
       \label{Calibrated causal graph}
    \end{subfigure}
    \caption{Population and estimated causal graphs for the data-generating process \eqref{eq.DGP_system}. ``$p$'' in Figure \ref{Calibrated causal graph} denotes the empirical P-values (ratio of bootstrapped statistics greater than the test statistic) of Algorithm \ref{algorithm.dependent_wild_bootstrap} in testing the corresponding direction.}
\label{figure.causal_graph}
\end{figure}




\begin{table}
    \centering
    \caption{Causal discovery result of the data-generating process \eqref{eq.DGP_system}. }
    \scriptsize
    \begin{tabular}{c c c c c c c c c}
    \hline\hline 
     Node pairs $(X,Y)$    &  \multicolumn{2}{c}{independence test} &    \multicolumn{2}{c}{Test $X\to Y$ }   &   \multicolumn{2}{c}{Test $Y\to X$ } \\
                   &  Statistics & Reject $H_0$ &  Statistics & Reject $H_0$ &  Statistics & Reject $H_0$\\
    $(\mathbf{w}_{t,1}, \mathbf{w}_{t,2} )$ & 87.0   & True & -4.8 &  False & 6.2 & True\\ 
    $(\mathbf{w}_{t,1}, \mathbf{w}_{t,3} )$ & 140.7 & True & 3.1 & False & 9.5 & True\\ 
    $(\mathbf{w}_{t,2}, \mathbf{w}_{t,3} )$ & 198.1 & True & 12.5 & True & 7.3 & False\\
    $(\mathbf{w}_{t,2}, \mathbf{w}_{t,4} )$ & 308.1 & True & 6.9 & False & -24.7 & True\\  
    $(\mathbf{w}_{t,3}, \mathbf{w}_{t,4} )$ & 355.8 & True & -1.8 & False & 26.7 & True\\ 
    \hline\hline 
    \end{tabular}
    \label{table.causation_direction}
\end{table}

\textbf{Testing for independence among stock returns.} This application collects 100 days of daily S\&P 500 stock prices obtained from {\it Kaggle} (see \url{https://www.kaggle.com/datasets/florentbaptist/sp-500}). After collecting the stock prices, we compute log returns and then group stocks according to their Global Industry Classification Standard (GICS) sector classifications. Our objective is to test for independence between log returns across different sectors.

Figure \ref{fig.data_example}  displays  the log returns of two representative stocks, and Figure \ref{fig.test_res} presents the pairwise hypothesis testing results. The tuning parameters are set to
$
(\lambda_1, \lambda_2, h, K_T) = (2, 6, 9, 7.8).
$
The test results indicate dependence between the communication services and the information technology sectors, which is economically plausible given their business overlap. For instance, companies such as Alphabet Inc. (communication services) are closely related to firms such as Apple Inc. (information technology) through shared platforms and cloud-based services, as discussed in \cite{doi:10.5465/amp.2016.0048}. We also observe dependence among the health care, information technology, and financial sectors. This pattern may reflect the growing integration of cloud computing and AI-driven technologies in health care and financial services, as demonstrated in \cite{athey2000impact, 10.1093/rfs/hhaa009}. In empirical studies, detecting dependence across time series provides useful empirical evidence and motivates further investigation into related areas.

\begin{figure}[htbp]
\centering
\begin{subfigure}[b]{0.35\linewidth}
    \centering
    \includegraphics[width=\linewidth]{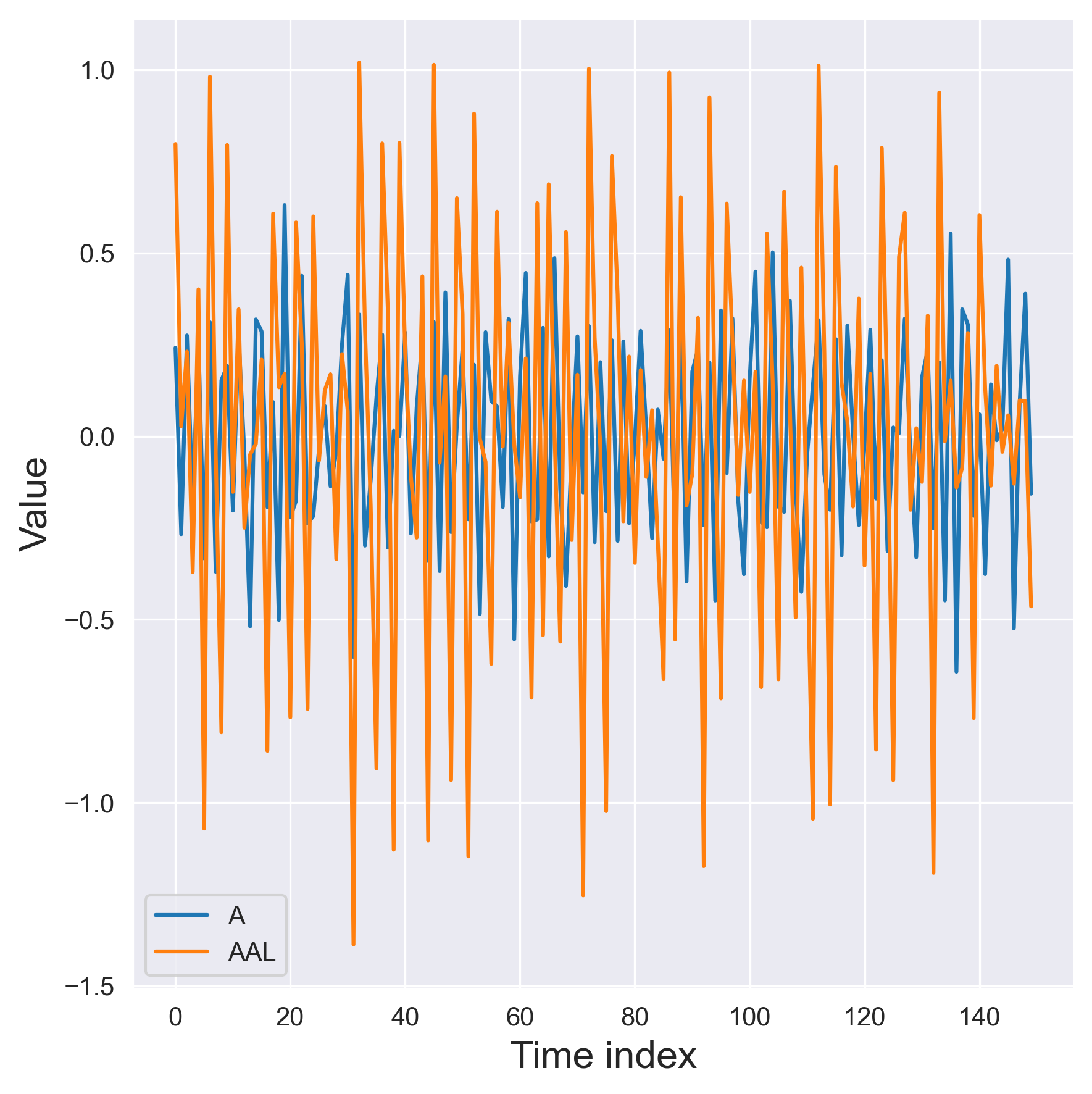}
    \caption{Observation examples}
\label{fig.data_example}
\end{subfigure}
\begin{subfigure}[b]{0.45\linewidth}
    \centering
    \includegraphics[width=\linewidth]{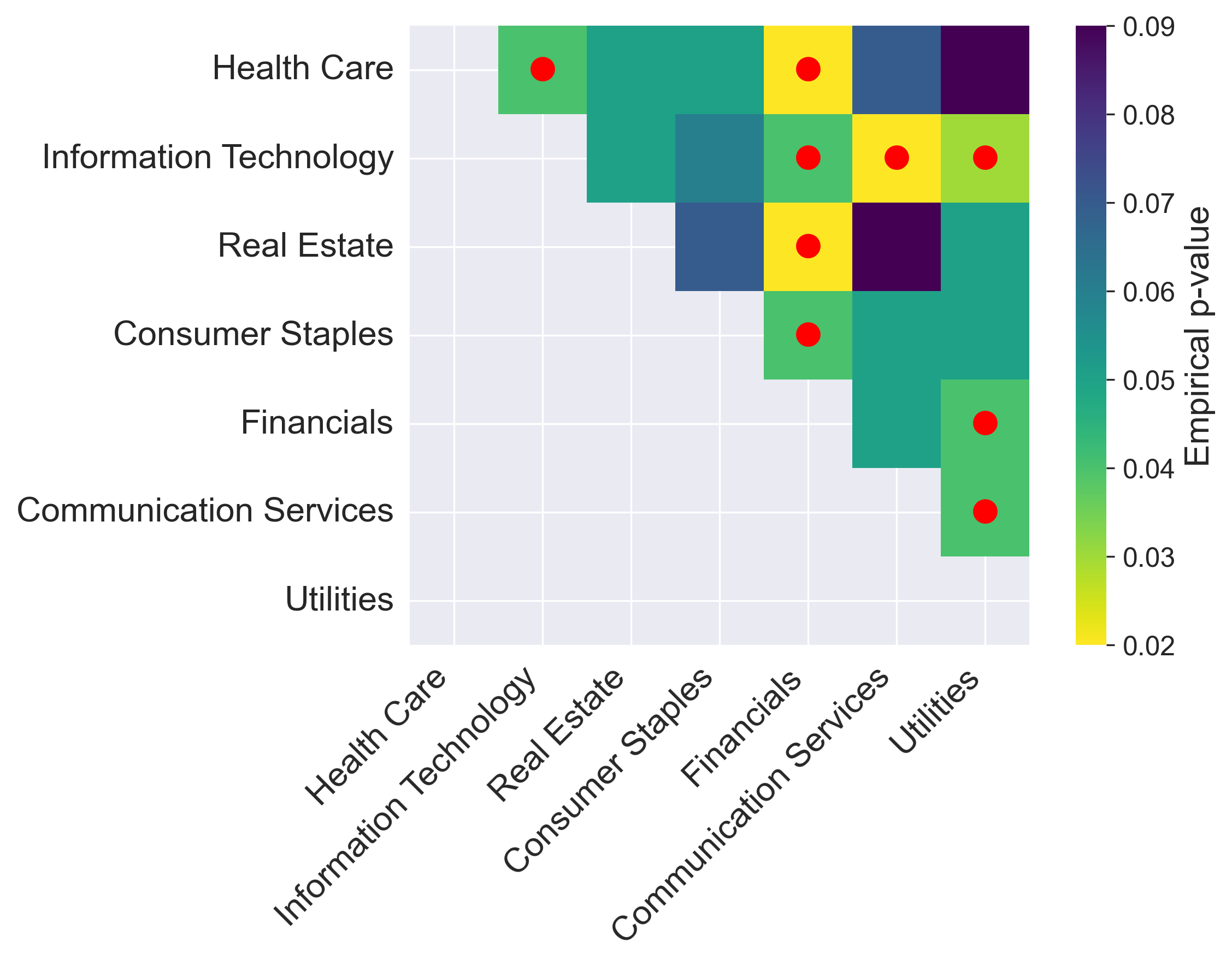}
    \caption{Hypothesis testing results}
    \label{fig.test_res}
\end{subfigure}
\caption{Observation examples and hypothesis testing results on S\&P 500 stock prices. The empirical P-values are computed in the same way as in Figure \ref{figure.causal_graph}.}
\end{figure}


\pagebreak











\appendix

\section{Additional literature review}

\textbf{High-dimensional vector autoregressive (VAR) processes.}  As demonstrated in \cite{MR3416594, MR4270386}, one possible approach is to whiten the time series by fitting a VAR model and conducting hypothesis testing based on the resulting residuals. However, this approach encounters substantial difficulties in high-dimensional settings. Consistent estimation of VAR coefficient matrices typically relies on structural assumptions, such as sparsity or low-rank conditions; see, for example, \cite{MR3450535, MR3357870, MR3620446, MR3925874, MR4561044, zhang2023statistical, MR4863067}. Furthermore, when the time series lacks a stationary autocovariance structure, which is common in modern applications as demonstrated in \cite{MR4270034}, the fitted residuals may fail to approximate the underlying innovations. Consequently, the whitening step can be unreliable, undermining the validity of subsequent independence tests.

\section{Behaviors of the test statistic under $H_0$ and $H_1$}
Remarks \ref{remark.underH0_appendix} and \ref{remark.underH1_appendix} extend Remarks \ref{remark.underH0} and \ref{remark.underH1} in the manuscript by providing detailed calculations of the expectation of the test statistic under both $H_0$ and $H_1$. These calculations formally justify the claims made in those remarks.

\begin{remark}[Under $H_0$]
\label{remark.underH0_appendix}
    Under $H_0,$ $\mathbf{x}_{t}^\top\mathbf{x}_{t_1}$ is independent of $\mathbf{y}_{t - s}^\top\mathbf{y}_{t_1 - s},$ so we have 
\begin{equation}
\begin{aligned}
    \mathbf{E}\left[O_t\right] &= \frac{1}{\mathcal{U}}\sum_{t_1 = (t - \lambda_2)\vee 1}^{t - \lambda_1}\sum_{s = 0}^{h\wedge (t_1 - 1)}\mathbf{E}\left[\mathbf{x}_{t}^\top\mathbf{x}_{t_1}\right]\mathbf{E}\left[\mathbf{y}_{t - s}^\top\mathbf{y}_{t_1 - s}\right]\\
    &+ \frac{1}{\mathcal{U}}\sum_{t_1 = (t - \lambda_2)\vee 1}^{t - \lambda_1}\sum_{s = 1}^{h\wedge (t_1 - 1)}\mathbf{E}\left[\mathbf{x}_{t - s}^\top\mathbf{x}_{t_1 - s}\right]\mathbf{E}\left[\mathbf{y}_{t}^\top\mathbf{y}_{t_1}\right].
\end{aligned}
\label{eq.bias_Ot}
\end{equation}
From Corollary \ref{corollary.covariances} of Section \ref{section.asymptotic} below, under weak dependence (Definition \ref{def.m_alpha_short_range_dependence}), the covariances $\mathbf{E}\left[\mathbf{x}_{t_1}^\top\mathbf{x}_{t_2}\right]$ and $\mathbf{E}\left[\mathbf{y}_{t_1}^\top\mathbf{y}_{t_2}\right]$ respectively satisfies 
\begin{equation}
\begin{aligned}
    \left\vert \mathbf{E}\left[\mathbf{x}_{t_1}^\top\mathbf{x}_{t_2}\right]\right\vert\leq \frac{Cd_1}{\vert t_1 - t_2\vert^\alpha},\quad \left\vert\mathbf{E}\left[\mathbf{y}_{t_1}^\top\mathbf{y}_{t_2}\right]\right\vert\leq \frac{Cd_2}{\vert t_1  -t_2\vert^\alpha}.
\end{aligned}
\label{eq.covariances_x_y}
\end{equation}
When we take summations, the expectation of $O_t$ and $\widehat{R}$ become
\begin{equation}
\begin{aligned}
   &\left\vert
   \mathbf{E}\left[O_t\right]
   \right\vert\leq \frac{Cd_1d_2}{\mathcal{U}}\sum_{t_1 = (t - \lambda_2)\vee 1}^{t - \lambda_1}\sum_{s = 0}^{h\wedge (t_1 - 1)}\frac{1}{(t - t_1)^{2\alpha}}\leq \frac{C_1d_1d_2h}{\mathcal{U}\lambda_1^{2\alpha - 1}},\\
   &\left\vert\mathbf{E}\left[ \widehat{R} \right]\right\vert\leq \sum_{t = 1 + \lambda_1}^T \left\vert\mathbf{E}\left[O_t\right]\right\vert\leq \frac{CTd_1d_2h}{\mathcal{U}\lambda_1^{2\alpha - 1}}.
\end{aligned}
\label{eq.E_R_hat}
\end{equation}
From \eqref{eq.E_R_hat}, although the test statistic exhibits bias arising from temporal dependence under $H_0$, this bias decays at a polynomial rate with respect to the bandwidth $\lambda_1.$  By selecting a sufficiently large $\lambda_1,$ the bias induced by temporal dependence becomes negligible relative to the stochastic error, and, as demonstrated in Theorem \ref{theorem.distributional_under_H0}, the test statistic remains of constant order.
\end{remark}

\begin{remark}[Under $H_1$]
\label{remark.underH1_appendix}
Due to the dependence across two time series, $\mathbf{E}\left[\mathbf{x}_{t}^\top\mathbf{x}_{t_1}\mathbf{y}_{t - s}^\top\mathbf{y}_{t_1 - s}\right]\neq 
\mathbf{E}\left[\mathbf{x}_{t}^\top\mathbf{x}_{t_1}\right]\mathbf{E}\left[\mathbf{y}_{t - s}^\top\mathbf{y}_{t_1 - s}\right]$ in general. For any $t\in\mathbf{Z}$ and $s\geq 0,$ we define the terms $\boldsymbol{\Pi}_{t,s,1} = \mathbf{E}\left[\mathbf{x}_t\mathbf{y}_{t-s}^\top\right]\in\mathbf{R}^{d_1\times d_2}$ and $\boldsymbol{\Pi}_{t,s,2} = \mathbf{E}\left[\mathbf{y}_t\mathbf{x}_{t-s}^\top\right]\in\mathbf{R}^{d_2\times d_1},$ since  
\begin{align*}
    \mathbf{E}\left[\mathbf{x}_{t}^\top\mathbf{x}_{t_1}\mathbf{y}_{t - s}^\top\mathbf{y}_{t_1 - s}\right] = \mathrm{Tr}\left(\boldsymbol{\Pi}_{t_1,s,1}\boldsymbol{\Pi}_{t,s,1}^\top\right) + \mathbf{E}\left[\left(\mathbf{x}_{t_1}\mathbf{y}_{t_1 - s}^\top - \boldsymbol{\Pi}_{t_1,s,1} \right)\left(\mathbf{y}_{t - s} \mathbf{x}_{t}^\top - \boldsymbol{\Pi}_{t,s,1}^\top\right)\right],
\end{align*}
and we have 
\begin{equation}
\begin{aligned}
    \mathcal{U}\mathbf{E}\left[O_t\right] &= \sum_{t_1 = (t - \lambda_2)\vee 1}^{t - \lambda_1}\sum_{s = 0}^{h\wedge (t_1 - 1)}\mathrm{Tr}\left(\boldsymbol{\Pi}_{t_1,s,1}\boldsymbol{\Pi}_{t,s,1}^\top\right) + \sum_{t_1 = (t - \lambda_2)\vee 1}^{t - \lambda_1}\sum_{s = 1}^{h\wedge (t_1 - 1)}\mathrm{Tr}\left(\boldsymbol{\Pi}_{t_1,s,2}\boldsymbol{\Pi}_{t,s,2}^\top\right)\\
    & +\sum_{t_1 = (t - \lambda_2)\vee 1}^{t - \lambda_1}\sum_{s = 0}^{h\wedge (t_1 - 1)}\mathrm{Tr}\left(\mathbf{E}\left[\left(\mathbf{x}_{t_1}\mathbf{y}_{t_1  -  s}^\top - \boldsymbol{\Pi}_{t_1,s,1}\right)\left(\mathbf{y}_{t-s}\mathbf{x}_t^\top - \boldsymbol{\Pi}_{t,s,1}^\top\right)\right]\right)\\
    & + \sum_{t_1 = (t - \lambda_2)\vee 1}^{t - \lambda_1}\sum_{s = 1}^{h\wedge (t_1 - 1)}\mathrm{Tr}\left(\mathbf{E}\left[\left(\mathbf{y}_{t_1}\mathbf{x}_{t_1  -  s}^\top - \boldsymbol{\Pi}_{t_1,s,2}\right)\left(\mathbf{x}_{t-s}\mathbf{y}_t^\top - \boldsymbol{\Pi}_{t,s,2}^\top\right)\right]\right).
\end{aligned}
\label{eq.Tr_Pi}
\end{equation}
Under the alternative where $\boldsymbol{\Pi}_{t_1,s,1}\neq 0$ or $\boldsymbol{\Pi}_{t,s,2}\neq 0$ for some $s = 0,1,\cdots, h,$ the nonzero traces $\mathrm{Tr}\left(\boldsymbol{\Pi}_{t_1,s,1}\boldsymbol{\Pi}_{t,s,1}^\top\right)$ or $\mathrm{Tr}\left(\boldsymbol{\Pi}_{t_1,s,2}\boldsymbol{\Pi}_{t,s,2}^\top\right)$ introduce bias to $O_t$ and consequently into $\widehat{R},$ which distinguishes it from $H_0.$
\end{remark}

\section{Additional numerical experimental details}

\subsection{Values of the parameter matrices in Section \ref{section.numerical_simulation}}
\label{section.values_parameter_matrices}

We select the matrix 
$$
\boldsymbol{\Pi} = \left[
    \begin{matrix}
        \boldsymbol{\Pi}^{(x)}  &  0\\
        0                       &  \boldsymbol{\Pi}^{(y)}
    \end{matrix}
    \right],
$$
where
\begin{align*}
    \quad 
    \boldsymbol{\Pi}^{(x)} = 
    \left[
    \begin{matrix}
        1.0 & 0.5 & 0  & \cdots & 0\\
        -0.3 & 1.0 & 0.5 & \cdots & 0\\
        0 & -0.3 & 1.0  & \cdots & 0\\
        \vdots & \vdots  & \vdots & \cdots & \vdots\\
        0 & 0 & 0 &  \cdots & 1.0
    \end{matrix}
    \right]\in\mathbf{R}^{d_1\times d_1},\quad \boldsymbol{\Pi}^{(y)} = 
    \left[
    \begin{matrix}
        1.0 & 0.3 & 0 & \cdots & 0\\
        -0.5 & 1.0 & 0.3 & \cdots & 0\\
        0 & -0.5 & 1.0  & \cdots & 0\\
        \vdots & \vdots  & \vdots & \cdots & \vdots\\
        0 & 0 & 0 & \cdots & 1.0
    \end{matrix}
    \right]\in\mathbf{R}^{d_2\times d_2}.
\end{align*}
Under the null hypothesis, we consider 
\begin{align*}
    \mathbf{A} =\left[
    \begin{matrix}
        \mathbf{A}^{(x)} & 0\\
        0  & \mathbf{A}^{(y)}
    \end{matrix}
    \right], \quad \mathbf{B} =\left[
    \begin{matrix}
        \mathbf{B}^{(x)} & 0\\
        0  & \mathbf{B}^{(y)}
    \end{matrix}
    \right],
\end{align*}
where 
\begin{align*}
    \mathbf{A}^{(x)} = 
    \left[
    \begin{matrix}
        0 & 0.5 & -0.2  & \cdots & 0\\
        0 & 0 & 0.5 & \cdots & 0\\
        0 & 0 & 0 & \cdots & 0\\
        \vdots & \vdots  & \vdots & \cdots & \vdots\\
        0 & 0 & 0 &  \cdots & 0
    \end{matrix}
    \right]\in\mathbf{R}^{d_1\times d_1},\quad \mathbf{A}^{(y)} = 
    \left[
    \begin{matrix}
        0 & 0 & 0  & \cdots & 0\\
        0.5 & 0 & 0 & \cdots & 0\\
        -0.2 & 0.5 & 0 & \cdots & 0\\
        \vdots & \vdots  & \vdots & \cdots & \vdots\\
        0 & 0 & 0 &  \cdots & 0
    \end{matrix}
    \right]\in\mathbf{R}^{d_2\times d_2},
\end{align*}
and 
\begin{align*}
    \mathbf{B}^{(x)} = 
    \left[
    \begin{matrix}
        0 & 0.2 & -0.5  & \cdots & 0\\
        0 & 0 & 0.2 & \cdots & 0\\
        0 & 0 & 0 & \cdots & 0\\
        \vdots & \vdots  & \vdots & \cdots & \vdots\\
        0 & 0 & 0 &  \cdots & 0
    \end{matrix}
    \right]\in\mathbf{R}^{d_1\times d_1},\quad \mathbf{B}^{(y)} = 
    \left[
    \begin{matrix}
        0 & 0 & 0  & \cdots & 0\\
        0.2 & 0 & 0 & \cdots & 0\\
        -0.5 & 0.2 & 0 & \cdots & 0\\
        \vdots & \vdots  & \vdots & \cdots & \vdots\\
        0 & 0 & 0 &  \cdots & 0
    \end{matrix}
    \right]\in\mathbf{R}^{d_2\times d_2}.
\end{align*}
Under the alternative hypothesis, we consider 
\begin{align*}
    \mathbf{A} =\left[
    \begin{matrix}
        \mathbf{A}^{(x)} & \mathbf{A}^{(x,y)}\\
        \mathbf{A}^{(x,y)\top}  & \mathbf{A}^{(y)}
    \end{matrix}
    \right], \quad \mathbf{B} =\left[
    \begin{matrix}
        \mathbf{B}^{(x)} & \mathbf{B}^{(x,y)}\\
        \mathbf{B}^{(x,y)\top}  & \mathbf{B}^{(y)}
    \end{matrix}
    \right],
\end{align*}
where 
\begin{align*}
    \mathbf{A}^{(x,y)} = \mathbf{B}^{(x,y)} = 
    \left[
    \begin{matrix}
        0 & \delta & 0  & \cdots & 0\\
        0 & 0 & \delta & \cdots & 0\\
        \vdots & \vdots  & \vdots & \cdots & \vdots\\
        0 & 0 & 0 &  \cdots & 0
    \end{matrix}
    \right]\in\mathbf{R}^{d_1\times d_2},
\end{align*}
where $\delta$ reflects the dependence strength.

\subsection{Additional numerical experiments}

Table \ref{size_power_performance_sample_size} further extends Table \ref{size_power_performance} in the manuscript by considering different sample sizes. The results show that the empirical size remains stable even when the sample size is relatively small. However, the sample size has a relatively strong effect on the power performance: the power of the test deteriorates for smaller sample sizes.

\begin{table}[htbp]
	\caption{Size and power performance of the proposed test procedure. The sample sizes are 300 and 700, and the nominal size of the test is 5\%. We repeat the experiment 100 times to evaluate the size and power.   }
    \scriptsize
    \centering
{	\begin{tabular}{c| rrrrr || rrrrr}
\hline\hline
& \multicolumn{10}{c}{Different choices of $\delta$}\\
\hline 
& \multicolumn{5}{c||}{Sample size $T = 300$} & \multicolumn{5}{c}{Sample size $T = 700$}\\
DGP & $h$ & $(\lambda_1, \lambda_2, \mathcal{K}_T)$ & 0(null) & 0.075 & 0.100 & $h$ & $(\lambda_1, \lambda_2, \mathcal{K}_T)$ & 0 (null) & 0.075 & 0.100\\
\hline
\multirow{9}{*}{nAR(1)} & \multirow{3}{*}{5} & $(2, 4, 9.6)$ & 6\% & 7\% & 16\% & \multirow{3}{*}{6} & (2, 4, 4.6) &  7\% & 13\% & 85\% \\
& &  $(2, 6, 7.5)$ & 6\% & 7\% & 19\% && $(2, 6, 7.3)$ & 7\% & 24\% & 90\% \\ 
& &  $(2, 12, 3.4)$ & 6\% & 12\% & 23\% && $(2, 14, 7.3)$ & 5\% & 21\% & 83\% \\ 
\cline{2-11}
& \multirow{3}{*}{10} & $(2, 4, 6.1)$ & 7\% & 7\% & 25\% & \multirow{3}{*}{11} & (2, 4, 4.7) & 7\% & 22\% & 87\%\\
&                     & $(2, 6, 3.4)$ & 5\% & 12\% & 24\% &  & (2, 6, 7.3) & 6\% & 22\% & 94\% \\
&                     & $(2, 6, 3.4)$ & 7\% & 7\% & 29\%  &  & (2, 14, 4.5) & 6\% & 25\% & 93\% \\
\cline{2-11}
& \multirow{3}{*}{33} & $(2, 4,  5.8)$ & 7\% & 13\% & 32\% &  \multirow{3}{*}{42} & $(2, 4, 7.2)$ & 8\% & 36\% & 98\%  \\
&& $(2, 6, 3.8)$ & 6\% & 13\% & 45\% & & $(2, 6, 7.1)$ & 3\% &52\% & 100\%\\
&& $(2, 12, 3.5)$ & 6\% & 10\% & 33\% & & $(2,  14,  28.7)$ & 6\% & 22\% & 94\% \\
\hline\hline
\multirow{9}{*}{nARMA(1, 1)} & \multirow{3}{*}{5}  & $(2, 4, 5.5)$ & 7\% & 11\% & 33\%  & \multirow{3}{*}{6} & $(2, 4, 9.6)$ & 4\% & 46\% & 100\%\\
&& $(2, 6, 6.4)$ & 4\% & 11\% & 49\% & & $(2, 6, 7.6)$ & 3\% & 50\% & 100\%\\
&& $(2, 12, 6.0)$ & 5\% & 12\% & 64\% & & $(2, 14, 9.7)$  & 6\% & 55\% & 100\% \\
\cline{2-11}
& \multirow{3}{*}{10} & $(2, 4, 4.1)$ & 5\% & 12\% & 58\% & \multirow{3}{*}{11} & (2, 4, 9.5) & 5\% & 46\% & 100\% \\
&   & $(2,6, 6.1)$ & 8\% & 12\% & 54\%  & & (2,6, 8.5) & 8\% & 61\% & 100\%\\
&&  $(2, 12, 12.2)$ & 4\% & 11\% & 44\% & & (2, 14, 10.1) & 6\% & 50\% & 100\%\\
\cline{2-11}
&\multirow{3}{*}{33} & (2, 4, 5.9) & 6\% & 11\% & 68\% & \multirow{3}{*}{42} & (2, 4, 7.3) & 7\% & 73\% & 100\%\\
& & (2, 6, 6.2) & 7\% & 12\% & 69\% & & (2, 6, 8.0) & 3\% & 84\% & 100\%\\
&  & (2, 12, 7.6) & 8\% & 15\% & 60\% & &(2, 14, 8.0) & 8\% & 68\% & 100\%\\
\hline\hline
\multirow{9}{*}{nTARMA(1, 1)} & \multirow{3}{*}{5} & (1, 2, 4.4) & 3\% & 10\% & 5\% &  \multirow{3}{*}{6} &  (2, 4, 10.0) & 5\% & 9\% & 35\% \\
                              &                    & (1, 3, 8.2) & 7\% & 5\% & 11\% & &(2, 6, 8.3) & 5\% & 12\% & 46\%\\
                              &                    & (1, 6, 10.1) & 5\% & 5\% & 14\% & & (2, 14, 7.3) & 3\% & 19\% & 58\% \\
                              \cline{2-11}
                              & \multirow{3}{*}{10} & (1, 2, 6.0) & 4\% & 6\% & 9\% & \multirow{3}{*}{11} & (2, 4, 9.7) & 7\% & 10\% & 28\% \\
                              &  & (1, 3, 6.2)  & 4\% & 7\% & 10\% && (2, 6, 11.4) & 5\% & 11\% & 24\% \\
                              &  & (1, 6, 6.5)  & 5\% & 5\% & 10\% && (2, 14, 15.6) & 4\% & 13\% & 28\% \\
                              \cline{2-11}
                              & \multirow{3}{*}{33} & (1, 2, 8.4) & 4\% & 1\% & 6\% &  \multirow{3}{*}{42} & (2, 4, 10.2) & 5\% & 4\% & 7\%\\
                              &  & (1, 3, 8.4) & 4\% & 4\% & 1\%  && (3, 9, 7.5) & 3\% & 5\% & 8\% \\
                              &  & (1, 6, 6.3) & 8\% & 3\% & 9\%  && (2, 14, 8.2) & 8\% & 8\% & 11\% \\
\hline
\hline
	\end{tabular}
}
	\label{size_power_performance_sample_size}
\end{table}

\textbf{Influence of high-dimensionality, dimension imbalance, and data nonlinearity.}  Table \ref{table.sample_imbalance} further assesses the effects of high dimensionality, dimensional imbalance, and the nonlinear mutual dependence structure on the performance of the testing procedure. 
The results show that both size and power remain stable when the dimension is large relative to the sample size. However, significant dimensional imbalance can lead to a loss of power. In addition, when nonlinear dependence is present, a relatively large sample size is required to maintain satisfactory power.

\begin{table}[htbp]
\caption{Size and power performance with respect to  different sample sizes. We choose the sample size $T = 500,$ the scale $\iota = \lfloor 2T^{0.1}\rfloor,$ the lag $\lfloor 6T^{0.3}\rfloor,$ and $\lambda_1,\mathcal{K}_T$ are chosen based on Algorithm \ref{algorithm.selection_lambda_1} and \cite{MR2041534}. 
}
\centering
\scriptsize
{
    \begin{tabular}{c|rrr||rrr||rrr}
    \hline\hline
            &   \multicolumn{8}{c}{Different choices of $\delta$}\\
        \multirow{2}{*}{\diagbox{DGP}{$(d_1, d_2) = $}}  &   \multicolumn{3}{c||}{$(\lfloor 1.2 T^{1.1}\rfloor,\quad \lfloor 1.5 T^{1.1}\rfloor)$} &   \multicolumn{3}{c||}{$(\lfloor 1.2 T^{1.2}\rfloor,\quad \lfloor 1.5 T^{1.2}\rfloor)$} &  \multicolumn{3}{c}{$(\lfloor 1.2 T^{1.3}\rfloor,\quad \lfloor 1.5 T^{1.3}\rfloor)$}\\
          &   0 (null)  &  0.075 & 0.100 & 0 (null)  &  0.075  & 0.100  & 0 (null)  & 0.075 & 0.100 \\
            \hline
    nAR(1)   &  5\% & 47\% & 100\% &  7\% & 76\% & 100\% & 5\% & 100\% & 100\% \\
    nARMA(1,1) & 7\% & 80\% & 100\% &  6\% & 99\% & 100\% & 4\% & 100\% & 100\% \\
    \hline\hline
    \\[1pt]
     &  \multicolumn{3}{c||}{$(\lfloor 0.1 T\rfloor,\quad \lfloor 1.5 T^{1.2}\rfloor)$} &\multicolumn{3}{c||}{$(\lfloor 0.5 T\rfloor,\quad \lfloor 1.5 T^{1.2}\rfloor)$} & \multicolumn{3}{c}{$(\lfloor 0.8 T\rfloor,\quad \lfloor 1.5 T^{1.2}\rfloor)$}\\
     & 0 (null)  & 0.075 & 0.100 &  0 (null)  & 0.075 & 0.100 & 0 (null)  & 0.075 & 0.100\\
     \hline 
     nAR(1) &7\% & 12\% & 28\% & 7\% & 19\% & 63\%  & 6\% & 19\% & 75\% \\
     nARMA(1,1) & 7\% & 10\% & 45\% & 5\% & 34\% & 92\% & 4\% & 49\% & 100\% \\
     \hline\hline
     \\[1pt]
     \multirow{2}{*}{\diagbox{DGP}{ $T$}} & \multicolumn{3}{c||}{$T = 800$}  &  \multicolumn{3}{c||}{$T = 1000$} & \multicolumn{3}{c}{$T = 1500$}\\
                &  0 (null)  &  0.075  & 0.100 & 0 (null)  &  0.075  & 0.100 & 0 (null)  &  0.075  & 0.100\\
                \hline 
     nTARMA(1,1) & 8\%    & 11\% & 63\% & 5\% & 19\% & 86\% & 4\% & 43\% & 99\% \\
     \hline\hline
    \end{tabular}
}
\label{table.sample_imbalance}
\end{table}

\newpage

\section{Proofs and technical details of Section \ref{section.asymptotic}}
\label{section.proof_linear_quadratic}
This section presents the proofs of theoretical results proposed in Section \ref{section.asymptotic}.

\begin{proof}[Proof of Theorem \ref{theorem.linear_and_quadratic}]
Recall that for any $t\in\mathbf{Z}$ and $s\geq 0,$ $\mathcal{F}_{t,s}$ is the $\sigma$-field generated by $e_t,e_{t-1}, \cdots ,e_{t-s}.$
    Define 
    $$
    \mathbf{E}\left[\mathbf{z}_t \mid \mathcal{F}_{t,-1}\right] = \mathbf{E}\left[\mathbf{z}_t\right] = 0,
    $$
    then from Corollary C.9 of \cite{MR2001996}, 
    \begin{align*}
        \mathbf{z}_t = \sum_{j = 0}^\infty\left(\mathbf{E}\left[\mathbf{z}_t\mid\mathcal{F}_{t,j}\right] - \mathbf{E}\left[\mathbf{z}_t\mid\mathcal{F}_{t,j - 1}\right]\right),
    \end{align*}
    which implies that 
    \begin{align*}
        \left\Vert
        \sum_{t = 1}^T \mathbf{a}_t^\top\mathbf{z}_t
        \right\Vert_M\leq \sum_{j = 0}^\infty \left\Vert
        \sum_{t = 1}^T \left(\mathbf{E}\left[\mathbf{a}_t^\top\mathbf{z}_t\mid\mathcal{F}_{t,j}\right] - \mathbf{E}\left[\mathbf{a}_t^\top\mathbf{z}_t\mid\mathcal{F}_{t,j - 1}\right]\right)
        \right\Vert_M
    \end{align*}
    and 
    \begin{align*}
        \left\Vert
        \sum_{t = 1}^T \mathbf{a}_t^\top\left(\mathbf{z}_t - \mathbf{E}\left[\mathbf{z}_t\mid\mathcal{F}_{t,s}\right]\right)
        \right\Vert_M\leq\sum_{j = s + 1}^\infty\left\Vert
        \sum_{t = 1}^T \left(\mathbf{E}\left[\mathbf{a}_t^\top\mathbf{z}_t\mid\mathcal{F}_{t,j}\right] - \mathbf{E}\left[\mathbf{a}_t^\top\mathbf{z}_t\mid\mathcal{F}_{t,j-1}\right]\right)
        \right\Vert_M.
    \end{align*}
    For any $j\geq 0$ and $z = 1,\cdots,T,$ define 
    \begin{align*}
        A_{z,j} = \sum_{t = T - z + 1}^T \left(\mathbf{E}\left[\mathbf{a}_t^\top\mathbf{z}_t\mid\mathcal{F}_{t,j}\right] - \mathbf{E}\left[\mathbf{a}_t^\top\mathbf{z}_t\mid\mathcal{F}_{t,j-1}\right]\right)
    \end{align*}
    and $\mathcal{A}_{z,j}$ the $\sigma$-field generated by $e_T,e_{T-1},\cdots,e_{T-z+1-j}.$ Then $A_{z,j}$ is measurable in $\mathcal{A}_{z,j},$ $\mathcal{A}_{z,j}\subset \mathcal{A}_{z + 1,j},$ and 
    \begin{align*}
        \mathbf{E}\left[A_{z + 1,j} - A_{z,j}\mid\mathcal{A}_{z,j} \right] &=  \mathbf{E}\left[\left(
        \mathbf{E}\left[\mathbf{a}_{T-z}^\top\mathbf{z}_{T-z}\mid\mathcal{F}_{T-z,j}\right] - \mathbf{E}\left[\mathbf{a}_{T-z}^\top\mathbf{z}_{T-z}\mid\mathcal{F}_{T-z,j-1}\right]
        \right)\mid\mathcal{A}_{z,j} \right]\\
        & = \mathbf{E}\left[\mathbf{a}_{T-z}^\top\mathbf{z}_{T-z}\mid\mathcal{F}_{T-z,j - 1}\right] - \mathbf{E}\left[\mathbf{a}_{T-z}^\top\mathbf{z}_{T-z}\mid\mathcal{F}_{T-z,j-1}\right] = 0,
    \end{align*}
    so $A_{z,j}$ forms a martingale. According to Theorem 1.1 of \cite{MR0400380}, 
    \begin{align*}
        \left\Vert \sum_{t = 1}^T \left(\mathbf{E}\left[\mathbf{a}_t^\top\mathbf{z}_t\mid\mathcal{F}_{t,j}\right] - \mathbf{E}\left[\mathbf{a}_t^\top\mathbf{z}_t\mid\mathcal{F}_{t,j-1}\right]\right)\right\Vert_M &\leq C\sqrt{\sum_{t = 1}^T\left\Vert \mathbf{E}\left[\mathbf{a}_t^\top\mathbf{z}_t\mid\mathcal{F}_{t,j}\right] - \mathbf{E}\left[\mathbf{a}_t^\top\mathbf{z}_t\mid\mathcal{F}_{t,j-1}\right]\right\Vert^2_M}\\
        &\leq C\sqrt{\sum_{t = 1}^T\left\Vert \mathbf{a}_t^\top\mathbf{z}_t - \mathbf{a}_t^\top\mathbf{z}_{t,j}\right\Vert^2_M}\leq C\delta_j\sqrt{\sum_{t = 1}^T\left\vert\mathbf{a}_t\right\vert^2_2}.
    \end{align*}
    From this inequality,
    \begin{align*}
        \left\Vert
        \sum_{t = 1}^T \mathbf{a}_t^\top\mathbf{z}_t
        \right\Vert_M\leq C\sqrt{\sum_{t = 1}^T\left\vert\mathbf{a}_t\right\vert^2_2}\sum_{j = 0}^\infty  \delta_j\leq C_1\sqrt{\sum_{t = 1}^T\left\vert\mathbf{a}_t\right\vert^2_2},
    \end{align*}
    and 
    \begin{align*}
         \left\Vert
        \sum_{t = 1}^T \mathbf{a}_t^\top\left(\mathbf{z}_t - \mathbf{E}\left[\mathbf{z}_t\mid\mathcal{F}_{t,s}\right]\right)
        \right\Vert_M\leq C\sqrt{\sum_{t = 1}^T\left\vert\mathbf{a}_t\right\vert^2_2}\sum_{j = s+1}^\infty\delta_j\leq \frac{C_1}{(1+s)^\alpha}\sqrt{\sum_{t = 1}^T\left\vert\mathbf{a}_t\right\vert^2_2},
    \end{align*}
    and we prove \eqref{eq.linear_combination}.

    For any $t\in\mathbf{Z}$ and any $s\in\mathbf{Z},$ define
    \begin{equation}
        \boldsymbol{\zeta}_{t,s} = 
        \begin{cases}
        \mathbf{E}\left[\mathbf{z}_{t}\mid \mathcal{F}_{t,s}\right] &\text{if}\quad s\geq 0,\\
        0&\text{if}\quad s < 0,
        \end{cases}\quad \boldsymbol{\omega}_{t,s} = \mathbf{z}_t - \boldsymbol{\zeta}_{t,s} .
        \label{eq.def_zeta_omega}
    \end{equation}
    In equation \eqref{eq.product_quadratic_forms}, $t_1\geq t_2 + \lambda_1 > t_2 + 1$ for sufficiently large $T,$ so 
    \begin{align*}
        \mathbf{z}_{t_1}^\top\mathbf{A}_{t_1t_2}\mathbf{z}_{t_2} = \boldsymbol{\zeta}_{t_1, t_1 - t_2 - 1}^\top\mathbf{A}_{t_1t_2}\mathbf{z}_{t_2} + \boldsymbol{\omega}_{t_1, t_1 - t_2 - 1}^\top\mathbf{A}_{t_1t_2}\mathbf{z}_{t_2},
    \end{align*}
    and $\boldsymbol{\zeta}_{t_1, t_1 - t_2 - 1}$ is measurable in $\mathcal{F}_{t_1, t_1 - t_2 - 1},$ which is independent of $\mathbf{z}_{t_2}.$ Furthermore, 
    \begin{align*}
        \mathbf{E}\left[\boldsymbol{\zeta}_{t_1, t_1 - t_2 - 1}^\top\mathbf{A}_{t_1t_2}\mathbf{z}_{t_2}\right]
        = \left(\mathbf{E}\left[\boldsymbol{\zeta}_{t_1, t_1 - t_2 - 1}\right]\right)^\top\mathbf{A}_{t_1t_2}\left(\mathbf{E}\left[\mathbf{z}_{t_2}\right]\right) = 0,
    \end{align*}
    so 
    \begin{align*}
        \mathbf{E}\left[\mathbf{z}_{t_1}^\top\mathbf{A}_{t_1t_2}\mathbf{z}_{t_2}\right] = \mathbf{E}\left[\boldsymbol{\omega}_{t_1, t_1 - t_2 - 1}^\top\mathbf{A}_{t_1t_2}\mathbf{z}_{t_2}\right],
    \end{align*}
    and 
    \begin{align*}
        &\left\Vert\sum_{t_1 = 1 + \lambda_1}^T\sum_{t_2 = (t_1 - \lambda_2)\vee 1}^{t_1 - \lambda_1}\left(\mathbf{z}_{t_1}^\top\mathbf{A}_{t_1t_2}\mathbf{z}_{t_2} - \mathbf{E}\left[\mathbf{z}_{t_1}^\top\mathbf{A}_{t_1t_2}\mathbf{z}_{t_2}\right]\right)\right\Vert_{M/2}\leq I + II,\\
    \end{align*}
    where 
    \begin{align*}
        I = \left\Vert\sum_{t_1 = 1 + \lambda_1}^T\sum_{t_2 = (t_1 - \lambda_2)\vee 1}^{t_1 - \lambda_1}\boldsymbol{\zeta}_{t_1, t_1 - t_2 - 1}^\top\mathbf{A}_{t_1t_2}\mathbf{z}_{t_2}\right\Vert_{M/2},
    \end{align*}
    and 
    \begin{align*}
        II = \left\Vert\sum_{t_1 = 1 + \lambda_1}^T\sum_{t_2 = (t_1 - \lambda_2)\vee 1}^{t_1 - \lambda_1}\left(\boldsymbol{\omega}_{t_1, t_1 - t_2 - 1}^\top\mathbf{A}_{t_1t_2}\mathbf{z}_{t_2} - \mathbf{E}\left[\boldsymbol{\omega}_{t_1, t_1 - t_2 - 1}^\top\mathbf{A}_{t_1t_2}\mathbf{z}_{t_2}\right]\right)\right\Vert_{M/2}.
    \end{align*}
    For $\mathbf{z}_{t_2} = \sum_{j = 0}^\infty\left(\mathbf{E}\left[\mathbf{z}_{t_2}\mid\mathcal{F}_{t_2, j}\right] - \mathbf{E}\left[\mathbf{z}_{t_2}\mid\mathcal{F}_{t_2, j - 1}\right]\right),$ 
    \begin{align*}
        I\leq \sum_{j = 0}^\infty\left\Vert\sum_{t_1 = 1 + \lambda_1}^T\sum_{t_2 = (t_1 - \lambda_2)\vee 1}^{t_1 - \lambda_1}\boldsymbol{\zeta}_{t_1, t_1 - t_2 - 1}^\top\mathbf{A}_{t_1t_2}\left(\mathbf{E}\left[\mathbf{z}_{t_2}\mid\mathcal{F}_{t_2, j}\right] - \mathbf{E}\left[\mathbf{z}_{t_2}\mid\mathcal{F}_{t_2, j - 1}\right]\right)\right\Vert_{M/2}.
    \end{align*}
    Notice that  
    \begin{align*}
        &\sum_{t_1 = 1 + \lambda_1}^T\sum_{t_2 = (t_1 - \lambda_2)\vee 1}^{t_1 - \lambda_1}\boldsymbol{\zeta}_{t_1, t_1 - t_2 - 1}^\top\mathbf{A}_{t_1t_2}\left(\mathbf{E}\left[\mathbf{z}_{t_2}\mid\mathcal{F}_{t_2, j}\right] - \mathbf{E}\left[\mathbf{z}_{t_2}\mid\mathcal{F}_{t_2, j - 1}\right]\right)\\
        & = \sum_{t_2 = 1}^{T-\lambda_1}\sum_{t_1 = t_2 + \lambda_1}^{(t_2 + \lambda_2)\wedge T}\boldsymbol{\zeta}_{t_1, t_1 - t_2 - 1}^\top\mathbf{A}_{t_1t_2}\left(\mathbf{E}\left[\mathbf{z}_{t_2}\mid\mathcal{F}_{t_2, j}\right] - \mathbf{E}\left[\mathbf{z}_{t_2}\mid\mathcal{F}_{t_2, j - 1}\right]\right).
    \end{align*}
    For any $j\geq 0$ and $z = 1,\cdots, T-\lambda_1,$ define 
    \begin{align*}
        B_{z,j} = \sum_{t_2 = T - \lambda_1 - z + 1}^{T-\lambda_1}\sum_{t_1 = t_2 + \lambda_1}^{(t_2 + \lambda_2)\wedge T}\boldsymbol{\zeta}_{t_1, t_1 - t_2 - 1}^\top\mathbf{A}_{t_1t_2}\left(\mathbf{E}\left[\mathbf{z}_{t_2}\mid\mathcal{F}_{t_2, j}\right] - \mathbf{E}\left[\mathbf{z}_{t_2}\mid\mathcal{F}_{t_2, j - 1}\right]\right)
    \end{align*}
    and $\mathcal{B}_{z,j}$ the $\sigma$-field generated by $e_T,e_{T-1},\cdots,e_{T-\lambda_1-z+1-j}.$ Then $B_{z,j}$ is measurable in $\mathcal{B}_{z,j},$ $\mathcal{B}_{z,j}\subset \mathcal{B}_{z+1,j},$ and 
    \begin{align*}
        &\mathbf{E}\left[B_{z + 1,j} - B_{z,j}\mid \mathcal{B}_{z,j}\right] \\
        & = \mathbf{E}\left[\sum_{t_1 = T-\lambda_1-z + \lambda_1}^{(T-\lambda_1-z + \lambda_2)\wedge T}\boldsymbol{\zeta}_{t_1, t_1 - T + \lambda_1 + z - 1}^\top\mathbf{A}_{t_1(T-\lambda_1-z)}\right.\\
        &
        \left(\mathbf{E}\left[\mathbf{z}_{T-\lambda_1-z}\mid\mathcal{F}_{T-\lambda_1-z, j}\right]- \mathbf{E}\left[\mathbf{z}_{T-\lambda_1-z}\mid\mathcal{F}_{T-\lambda_1-z, j - 1}\right]\right)\mid \mathcal{B}_{z,j}\Bigg]\\
        & = \sum_{t_1 = T-\lambda_1-z + \lambda_1}^{(T-\lambda_1-z + \lambda_2)\wedge T}\boldsymbol{\zeta}_{t_1, t_1 - T + \lambda_1 + z - 1}^\top\mathbf{A}_{t_1(T-\lambda_1-z)}\left(\mathbf{E}\left[\mathbf{z}_{T-\lambda_1-z}\mid\mathcal{F}_{T-\lambda_1-z, j - 1}\right]\right.\\
        &\left.- \mathbf{E}\left[\mathbf{z}_{T-\lambda_1-z}\mid\mathcal{F}_{T-\lambda_1-z, j - 1}\right]\right)\\
        &= 0.
    \end{align*}
    Therefore, $B_{z,j}$ forms a martingale, and from Theorem 1.1 of \cite{MR0400380}, 
    \begin{align*}
        &\left\Vert
        \sum_{t_1 = 1 + \lambda_1}^T\sum_{t_2 = (t_1 - \lambda_2)\vee 1}^{t_1 - \lambda_1}\boldsymbol{\zeta}_{t_1, t_1 - t_2 - 1}^\top\mathbf{A}_{t_1t_2}\left(\mathbf{E}\left[\mathbf{z}_{t_2}\mid\mathcal{F}_{t_2, j}\right] - \mathbf{E}\left[\mathbf{z}_{t_2}\mid\mathcal{F}_{t_2, j - 1}\right]\right)
        \right\Vert_{M/2}\\
       &\leq C\sqrt{\sum_{t_2 = 1}^{T-\lambda_1}\left\Vert \sum_{t_1 = t_2 + \lambda_1}^{(t_2 + \lambda_2)\wedge T}\boldsymbol{\zeta}_{t_1, t_1 - t_2 - 1}^\top\mathbf{A}_{t_1t_2}\left(\mathbf{E}\left[\mathbf{z}_{t_2}\mid\mathcal{F}_{t_2, j}\right] - \mathbf{E}\left[\mathbf{z}_{t_2}\mid\mathcal{F}_{t_2, j - 1}\right]\right)\right\Vert^2_{M/2}}\\
       & = C\sqrt{\sum_{t_2 = 1}^{T-\lambda_1}\left\Vert \left(\sum_{t_1 = t_2 + \lambda_1}^{(t_2 + \lambda_2)\wedge T}\boldsymbol{\zeta}_{t_1, t_1 - t_2 - 1}^\top\mathbf{A}_{t_1t_2}\right)\left(\mathbf{E}\left[\mathbf{z}_{t_2}\mid\mathcal{F}_{t_2, j}\right] - \mathbf{E}\left[\mathbf{z}_{t_2}\mid\mathcal{F}_{t_2, j - 1}\right]\right)\right\Vert^2_{M/2}}.
    \end{align*}
    For any real-value vector $\mathbf{a}\in\mathbf{R}^d,$ since 
    \begin{align*}
        \sum_{t_1 = t_2 + \lambda_1}^{(t_2 + \lambda_2)\wedge T}\boldsymbol{\zeta}_{t_1, t_1 - t_2 - 1}^\top\mathbf{A}_{t_1t_2}\mathbf{a}
        &= \sum_{t_1 = t_2 + \lambda_1}^{(t_2 + \lambda_2)\wedge T}\left(\mathbf{E}\left[\mathbf{z}_{t_1}\mid\mathcal{F}_{t_1, t_1 - t_2 - 1}\right]\right)^\top\mathbf{A}_{t_1t_2}\mathbf{a}\\
        &=  \sum_{t_1 = t_2 + \lambda_1}^{(t_2 + \lambda_2)\wedge T}\left(\mathbf{E}\left[\mathbf{z}_{t_1}\mid\mathcal{F}_{T, T - t_2 - 1}\right]\right)^\top\mathbf{A}_{t_1t_2}\mathbf{a},
    \end{align*}
    from equation \eqref{eq.linear_combination} in the manuscript, 
    \begin{equation}
    \begin{aligned}
        \left\Vert \sum_{t_1 = t_2 + \lambda_1}^{(t_2 + \lambda_2)\wedge T}\boldsymbol{\zeta}_{t_1, t_1 - t_2 - 1}^\top\mathbf{A}_{t_1t_2}\mathbf{a}\right\Vert_{M/2} &= \left\Vert\mathbf{E}\left[\sum_{t_1 = t_2 + \lambda_1}^{(t_2 + \lambda_2)\wedge T}\mathbf{z}_{t_1}^\top \mathbf{A}_{t_1t_2}\mathbf{a}\mid \mathcal{F}_{T, T - t_2 - 1}\right]\right\Vert_{M/2}\\
        &\leq \left\Vert\sum_{t_1 = t_2 + \lambda_1}^{(t_2 + \lambda_2)\wedge T}\mathbf{a}^\top\mathbf{A}_{t_1t_2}^\top\mathbf{z}_{t_1}\right\Vert_{M/2}\\
        &\leq 
        C\sqrt{\sum_{t_1 = t_2 + \lambda_1}^{(t_2 + \lambda_2)\wedge T}\left\vert \mathbf{A}_{t_1t_2}\mathbf{a}\right\vert_2^2}\leq C\left\vert\mathbf{a}\right\vert_2\sqrt{\sum_{t_1 = t_2 + \lambda_1}^{(t_2 + \lambda_2)\wedge T}\left\vert\mathbf{A}_{t_1t_2}\right\vert^2_2}.
    \end{aligned}
    \label{eq.linear_truncate_moment}
    \end{equation}
    For $\boldsymbol{\zeta}_{t_1, t_1 - t_2 - 1}$ is measurable in the $\sigma$-field generated by $e_{t_1},\cdots,e_{t_2 + 1},$  $\sum_{t_1 = t_2 + \lambda_1}^{(t_2 + \lambda_2)\wedge T}\boldsymbol{\zeta}_{t_1, t_1 - t_2 - 1}^\top\mathbf{A}_{t_1t_2}$ is independent of $\mathbf{E}\left[\mathbf{z}_{t_2}\mid\mathcal{F}_{t_2, j}\right] - \mathbf{E}\left[\mathbf{z}_{t_2}\mid\mathcal{F}_{t_2, j - 1}\right],$ and from 
    Examples 1.22 of \cite{MR2767184}, for any vector $\mathbf{a}\in\mathbf{R}^{d},$ 
    \begin{align*}
        &\mathbf{E}\left[\left\vert\left(\sum_{t_1 = t_2 + \lambda_1}^{(t_2 + \lambda_2)\wedge T}\boldsymbol{\zeta}_{t_1, t_1 - t_2 - 1}^\top\mathbf{A}_{t_1t_2}\right)\left(\mathbf{E}\left[\mathbf{z}_{t_2}\mid\mathcal{F}_{t_2, j}\right] - \mathbf{E}\left[\mathbf{z}_{t_2}\mid\mathcal{F}_{t_2, j - 1}\right]\right)\right\vert^{M/2}\right.\\
        &\mid \mathbf{E}\left[\mathbf{z}_{t_2}\mid\mathcal{F}_{t_2, j}\right] - \mathbf{E}\left[\mathbf{z}_{t_2}\mid\mathcal{F}_{t_2, j - 1}\right] = \mathbf{a}\Bigg]\\
        & = \mathbf{E}\left[\left\vert\sum_{t_1 = t_2 + \lambda_1}^{(t_2 + \lambda_2)\wedge T}\boldsymbol{\zeta}_{t_1, t_1 - t_2 - 1}^\top\mathbf{A}_{t_1t_2}\mathbf{a}\right\vert^{M/2}\right]\leq C\left(\sum_{t_1 = t_2 + \lambda_1}^{(t_2 + \lambda_2)\wedge T}\left\vert\mathbf{A}_{t_1t_2}\right\vert^2_2\right)^{M/4}\left\vert\mathbf{a}\right\vert_2^{M/2},
    \end{align*}
    and 
    \begin{align*}
        &\left\Vert
        \left(\sum_{t_1 = t_2 + \lambda_1}^{(t_2 + \lambda_2)\wedge T}\boldsymbol{\zeta}_{t_1, t_1 - t_2 - 1}^\top\mathbf{A}_{t_1t_2}\right)\left(\mathbf{E}\left[\mathbf{z}_{t_2}\mid\mathcal{F}_{t_2, j}\right] - \mathbf{E}\left[\mathbf{z}_{t_2}\mid\mathcal{F}_{t_2, j - 1}\right]\right)
        \right\Vert_{M/2}\\
        &\leq C\sqrt{\sum_{t_1 = t_2 + \lambda_1}^{(t_2 + \lambda_2)\wedge T}\left\vert\mathbf{A}_{t_1t_2}\right\vert^2_2}\ \left\Vert\ \left\vert \mathbf{E}\left[\mathbf{z}_{t_2}\mid\mathcal{F}_{t_2, j}\right] - \mathbf{E}\left[\mathbf{z}_{t_2}\mid\mathcal{F}_{t_2, j - 1}\right]\right\vert_2\ \right\Vert_{M/2}.
    \end{align*}
    Since 
    \begin{align*}
        &\left\Vert\ \left\vert \mathbf{E}\left[\mathbf{z}_{t_2}\mid\mathcal{F}_{t_2, j}\right] - \mathbf{E}\left[\mathbf{z}_{t_2}\mid\mathcal{F}_{t_2, j - 1}\right]\right\vert_2\ \right\Vert_{M/2}\\
        &= \left\Vert\ \sqrt{\sum_{i = 1}^d \left(\mathbf{E}\left[\mathbf{z}_{t_2}^{(i)}\mid\mathcal{F}_{t_2, j}\right] - \mathbf{E}\left[\mathbf{z}_{t_2}^{(i)}\mid\mathcal{F}_{t_2, j - 1}\right]\right)^2} \right\Vert_{M/2}\\
        & = \sqrt{\left\Vert \sum_{i = 1}^d \left(\mathbf{E}\left[\mathbf{z}_{t_2}^{(i)}\mid\mathcal{F}_{t_2, j}\right] - \mathbf{E}\left[\mathbf{z}_{t_2}^{(i)}\mid\mathcal{F}_{t_2, j - 1}\right]\right)^2\right\Vert_{M/4}}\\
        &\leq \sqrt{\sum_{i = 1}^d\left\Vert \mathbf{E}\left[\mathbf{z}_{t_2}^{(i)}\mid\mathcal{F}_{t_2, j}\right] - \mathbf{E}\left[\mathbf{z}_{t_2}^{(i)}\mid\mathcal{F}_{t_2, j - 1}\right]\right\Vert^2_{M/2}}\leq \sqrt{d}\delta_j,
    \end{align*}
    we have 
    \begin{align*}
        &\left\Vert
        \left(\sum_{t_1 = t_2 + \lambda_1}^{(t_2 + \lambda_2)\wedge T}\boldsymbol{\zeta}_{t_1, t_1 - t_2 - 1}^\top\mathbf{A}_{t_1t_2}\right)\left(\mathbf{E}\left[\mathbf{z}_{t_2}\mid\mathcal{F}_{t_2, j}\right] - \mathbf{E}\left[\mathbf{z}_{t_2}\mid\mathcal{F}_{t_2, j - 1}\right]\right)
        \right\Vert_{M/2}\\
        &\leq C\delta_j\sqrt{d\sum_{t_1 = t_2 + \lambda_1}^{(t_2 + \lambda_2)\wedge T}\left\vert\mathbf{A}_{t_1t_2}\right\vert^2_2},
    \end{align*}
    and 
    \begin{equation}
        \begin{aligned}
            &\left\Vert
        \sum_{t_1 = 1 + \lambda_1}^T\sum_{t_2 = (t_1 - \lambda_2)\vee 1}^{t_1 - \lambda_1}\boldsymbol{\zeta}_{t_1, t_1 - t_2 - 1}^\top\mathbf{A}_{t_1t_2}\left(\mathbf{E}\left[\mathbf{z}_{t_2}\mid\mathcal{F}_{t_2, j}\right] - \mathbf{E}\left[\mathbf{z}_{t_2}\mid\mathcal{F}_{t_2, j - 1}\right]\right)
        \right\Vert_{M/2}\\
        &\leq C\delta_j\sqrt{d\sum_{t_2 = 1}^{T-\lambda_1}\sum_{t_1 = t_2 + \lambda_1}^{(t_2 + \lambda_2)\wedge T}\left\vert\mathbf{A}_{t_1t_2}\right\vert^2_2} = C\delta_j\sqrt{d\sum_{t_1 = 1 + \lambda_1}^{T}\sum_{t_2  = (t_1 - \lambda_2)\vee 1}^{t_1 - \lambda_1}\left\vert\mathbf{A}_{t_1t_2}\right\vert^2_2}.
        \end{aligned}
        \label{eq.Fir_separate}
    \end{equation}
    From \eqref{eq.Fir_separate}, 
    \begin{equation}
        \begin{aligned}
            I &\leq \left(C\sum_{j = 0}^\infty\delta_j\right)\sqrt{d\sum_{t_1 = 1 + \lambda_1}^{T}\sum_{t_2  = (t_1 - \lambda_2)\vee 1}^{t_1 - \lambda_1}\left\vert\mathbf{A}_{t_1t_2}\right\vert^2_2}\\
            &\leq C_1\sqrt{d\sum_{t_1 = 1 + \lambda_1}^{T}\sum_{t_2  = (t_1 - \lambda_2)\vee 1}^{t_1 - \lambda_1}\left\vert\mathbf{A}_{t_1t_2}\right\vert^2_2}.
        \end{aligned}
        \label{eq.fir_part_delta_z}
    \end{equation}
    For the second part, define 
    \begin{align*}
        \mathbf{E}\left[\boldsymbol{\omega}_{t_1, t_1 - t_2 - 1}^\top\mathbf{A}_{t_1t_2}\mathbf{z}_{t_2}\mid\mathcal{F}_{t_1, - 1}\right] = \mathbf{E}\left[\boldsymbol{\omega}_{t_1, t_1 - t_2 - 1}^\top\mathbf{A}_{t_1t_2}\mathbf{z}_{t_2}\right],
    \end{align*}
    then 
    \begin{align*}
    &\boldsymbol{\omega}_{t_1, t_1 - t_2 - 1}^\top\mathbf{A}_{t_1t_2}\mathbf{z}_{t_2} - \mathbf{E}\left[\boldsymbol{\omega}_{t_1, t_1 - t_2 - 1}^\top\mathbf{A}_{t_1t_2}\mathbf{z}_{t_2}\right]\\
    &= \sum_{j = 0}^\infty\left(\mathbf{E}\left[\boldsymbol{\omega}_{t_1, t_1 - t_2 - 1}^\top\mathbf{A}_{t_1t_2}\mathbf{z}_{t_2}\mid\mathcal{F}_{t_1, j}\right] - \mathbf{E}\left[\boldsymbol{\omega}_{t_1, t_1 - t_2 - 1}^\top\mathbf{A}_{t_1t_2}\mathbf{z}_{t_2}\mid\mathcal{F}_{t_1, j - 1}\right]\right),
    \end{align*}
    which makes 
    \begin{align*}
        II\leq \sum_{j = 0}^\infty \left\Vert\sum_{t_1 = 1 + \lambda_1}^T\sum_{t_2 = (t_1 - \lambda_2)\vee 1}^{t_1 - \lambda_1}\left(\mathbf{E}\left[\boldsymbol{\omega}_{t_1, t_1 - t_2 - 1}^\top\mathbf{A}_{t_1t_2}\mathbf{z}_{t_2}\mid\mathcal{F}_{t_1, j}\right] - \mathbf{E}\left[\boldsymbol{\omega}_{t_1, t_1 - t_2 - 1}^\top\mathbf{A}_{t_1t_2}\mathbf{z}_{t_2}\mid\mathcal{F}_{t_1, j - 1}\right]\right)\right\Vert_{M/2}.
    \end{align*}
    For any $j\geq 0$ and $z = 1,2,\cdots, T - \lambda_1,$ define 
    \begin{align*}
        D_{z,j} = \sum_{t_1 = T-z+1}^T\sum_{t_2 = (t_1 - \lambda_2)\vee 1}^{t_1 - \lambda_1}\left(\mathbf{E}\left[\boldsymbol{\omega}_{t_1, t_1 - t_2 - 1}^\top\mathbf{A}_{t_1t_2}\mathbf{z}_{t_2}\mid\mathcal{F}_{t_1, j}\right] - \mathbf{E}\left[\boldsymbol{\omega}_{t_1, t_1 - t_2 - 1}^\top\mathbf{A}_{t_1t_2}\mathbf{z}_{t_2}\mid\mathcal{F}_{t_1, j - 1}\right]\right)
    \end{align*}
    and $\mathcal{D}_{z,j}$ the $\sigma$-field generated by $e_T,\cdots,e_{T-z+1-j}.$ Then $D_{z,j} $ is measurable in $\mathcal{D}_{z,j},$ $\mathcal{D}_{z,j}\subset \mathcal{D}_{z + 1,j},$ and 
    \begin{align*}
        &\mathbf{E}\left[D_{z + 1,j} - D_{z,j} \mid \mathcal{D}_{z,j}\right]\\
        &= \sum_{t_2 = (T - z - \lambda_2)\vee 1}^{T - z - \lambda_1}\mathbf{E}\left[\left(\mathbf{E}\left[\boldsymbol{\omega}_{T-z, T-z - t_2 - 1}^\top\mathbf{A}_{(T-z)t_2}\mathbf{z}_{t_2}\mid\mathcal{F}_{T-z, j}\right] \right.\right.\\
        &\left.\left.- \mathbf{E}\left[\boldsymbol{\omega}_{T-z, T-z - t_2 - 1}^\top\mathbf{A}_{(T-z)t_2}\mathbf{z}_{t_2}\mid\mathcal{F}_{T-z, j - 1}\right]\right)\mid \mathcal{D}_{z,j}\right]\\
        & = \sum_{t_2 = (T - z - \lambda_2)\vee 1}^{T - z - \lambda_1}\left(\mathbf{E}\left[\boldsymbol{\omega}_{T-z, T-z - t_2 - 1}^\top\mathbf{A}_{(T-z)t_2}\mathbf{z}_{t_2}\mid\mathcal{F}_{T-z, j - 1}\right]\right.\\
        &\left.- \mathbf{E}\left[\boldsymbol{\omega}_{T-z, T-z - t_2 - 1}^\top\mathbf{A}_{(T-z)t_2}\mathbf{z}_{t_2}\mid\mathcal{F}_{T-z, j - 1}\right]\right)\\
        &= 0,
    \end{align*}
    so $D_{z,j}$ forms a martingale. From Theorem 1.1 of \cite{MR0400380},
    \begin{align*}
        &\left\Vert\sum_{t_1 = 1 + \lambda_1}^T\sum_{t_2 = (t_1 - \lambda_2)\vee 1}^{t_1 - \lambda_1}\left(\mathbf{E}\left[\boldsymbol{\omega}_{t_1, t_1 - t_2 - 1}^\top\mathbf{A}_{t_1t_2}\mathbf{z}_{t_2}\mid\mathcal{F}_{t_1, j}\right] - \mathbf{E}\left[\boldsymbol{\omega}_{t_1, t_1 - t_2 - 1}^\top\mathbf{A}_{t_1t_2}\mathbf{z}_{t_2}\mid\mathcal{F}_{t_1, j - 1}\right]\right)\right\Vert_{M/2}\\
        &\leq C\sqrt{\sum_{t_1 = 1 + \lambda_1}^T\left\Vert \sum_{t_2 = (t_1 - \lambda_2)\vee 1}^{t_1 - \lambda_1}\left(\mathbf{E}\left[\boldsymbol{\omega}_{t_1, t_1 - t_2 - 1}^\top\mathbf{A}_{t_1t_2}\mathbf{z}_{t_2}\mid\mathcal{F}_{t_1, j}\right] - \mathbf{E}\left[\boldsymbol{\omega}_{t_1, t_1 - t_2 - 1}^\top\mathbf{A}_{t_1t_2}\mathbf{z}_{t_2}\mid\mathcal{F}_{t_1, j - 1}\right]\right)\right\Vert^2_{M/2}}.
    \end{align*}
    Notice that 
    \begin{align*}
        &\left\Vert \sum_{t_2 = (t_1 - \lambda_2)\vee 1}^{t_1 - \lambda_1}\left(\mathbf{E}\left[\boldsymbol{\omega}_{t_1, t_1 - t_2 - 1}^\top\mathbf{A}_{t_1t_2}\mathbf{z}_{t_2}\mid\mathcal{F}_{t_1, j}\right] - \mathbf{E}\left[\boldsymbol{\omega}_{t_1, t_1 - t_2 - 1}^\top\mathbf{A}_{t_1t_2}\mathbf{z}_{t_2}\mid\mathcal{F}_{t_1, j - 1}\right]\right)\right\Vert_{M/2}\\
        &\leq \sum_{t_2 = (t_1 - \lambda_2)\vee 1}^{t_1 - \lambda_1}\left\Vert \mathbf{E}\left[\boldsymbol{\omega}_{t_1, t_1 - t_2 - 1}^\top\mathbf{A}_{t_1t_2}\mathbf{z}_{t_2}\mid\mathcal{F}_{t_1, j}\right] - \mathbf{E}\left[\boldsymbol{\omega}_{t_1, t_1 - t_2 - 1}^\top\mathbf{A}_{t_1t_2}\mathbf{z}_{t_2}\mid\mathcal{F}_{t_1, j - 1}\right]\right\Vert_{M/2}
    \end{align*}
    On the one hand, we have 
    \begin{align*}
        &\left\Vert\mathbf{E}\left[\boldsymbol{\omega}_{t_1, t_1 - t_2 - 1}^\top\mathbf{A}_{t_1t_2}\mathbf{z}_{t_2}\mid\mathcal{F}_{t_1, j}\right] - \mathbf{E}\left[\boldsymbol{\omega}_{t_1, t_1 - t_2 - 1}^\top\mathbf{A}_{t_1t_2}\mathbf{z}_{t_2}\mid\mathcal{F}_{t_1, j - 1}\right]\right\Vert_{M/2}\\
        &\leq 2\left\Vert \boldsymbol{\omega}_{t_1, t_1 - t_2 - 1}^\top\mathbf{A}_{t_1t_2}\mathbf{z}_{t_2}\right\Vert_{M/2}.
    \end{align*}
    Since
    \begin{align*}
        \left\vert \boldsymbol{\omega}_{t_1, t_1 - t_2 - 1}^\top\mathbf{A}_{t_1t_2}\mathbf{z}_{t_2}\right\vert\leq \left\vert \boldsymbol{\omega}_{t_1, t_1 - t_2 - 1}\right\vert_2\left\vert\mathbf{A}_{t_1t_2}\right\vert_2\left\vert \mathbf{z}_{t_2}\right\vert_2,
    \end{align*}
    from Cauchy--Schwarz inequality,
    \begin{align*}
        \left\Vert
        \boldsymbol{\omega}_{t_1, t_1 - t_2 - 1}^\top\mathbf{A}_{t_1t_2}\mathbf{z}_{t_2}
        \right\Vert_{M/2} &\leq \left\Vert \left\vert \boldsymbol{\omega}_{t_1, t_1 - t_2 - 1}\right\vert_2\left\vert\mathbf{A}_{t_1t_2}\right\vert_2\left\vert \mathbf{z}_{t_2}\right\vert_2\right\Vert_{M/2}\\
        &\leq \left\vert\mathbf{A}_{t_1t_2}\right\vert_2\left\Vert\ \left\vert \boldsymbol{\omega}_{t_1, t_1 - t_2 - 1}\right\vert_2\ \right\Vert_M\left\Vert\ \left\vert \mathbf{z}_{t_2}\right\vert_2\ \right\Vert_M.
    \end{align*}
    From \eqref{eq.linear_combination} 
    \begin{align*}
        \left\Vert\ \left\vert \boldsymbol{\omega}_{t_1, t_1 - t_2 - 1}\right\vert_2\ \right\Vert_M
        \leq \sqrt{\sum_{i = 1}^d\left\Vert \mathbf{z}_{t_1}^{(i)} - \mathbf{E}\left[\mathbf{z}_{t_1}^{(i)}\mid\mathcal{F}_{t_1, t_1 - t_2 - 1}\right]\right\Vert^2_{M}}\leq\frac{C\sqrt{d}}{(t_1 -  t_2)^\alpha},    
    \end{align*}
    we also have 
    \begin{align*}
        \left\Vert\ \left\vert \mathbf{z}_{t_2}\right\vert_2\ \right\Vert_M\leq \sqrt{\sum_{i = 1}^d\left\Vert \mathbf{z}_{t_2}^{(i)}\right\Vert^2_{M}}\leq C\sqrt{d},
    \end{align*}
    which makes 
    \begin{equation}
    \begin{aligned}
        \left\Vert
        \boldsymbol{\omega}_{t_1, t_1 - t_2 - 1}^\top\mathbf{A}_{t_1t_2}\mathbf{z}_{t_2}
        \right\Vert_{M/2}\leq \frac{Cd a^\dagger}{(t_1 -  t_2)^\alpha},
    \end{aligned}
    \label{eq.omega_one_term}
    \end{equation}
    where $a^\dagger = \max_{t_1 = 1+\lambda_1,\cdots, T;\ t_2 = (t_1 - \lambda_2)\vee 1,\cdots, t_1  -\lambda_1}\vert \mathbf{A}_{t_1t_2}\vert_2.$

    On the other hand, if $t_1 - t_2 < j,$ we have  
    \begin{align*}
         \boldsymbol{\omega}_{t_1, t_1 - t_2 - 1} = \mathbf{z}_{t_1} - \boldsymbol{\zeta}_{t_1, j - 1} + \boldsymbol{\zeta}_{t_1, j - 1} - \boldsymbol{\zeta}_{t_1, t_1 - t_2 -  1} =  \boldsymbol{\omega}_{t_1, j - 1}  + \left(\boldsymbol{\zeta}_{t_1, j - 1} - \boldsymbol{\zeta}_{t_1, t_1 - t_2 -  1}\right),
    \end{align*}
    and $\boldsymbol{\zeta}_{t_1, j - 1} - \boldsymbol{\zeta}_{t_1, t_1 - t_2 -  1}$ is measurable in $\mathcal{F}_{t_1, j - 1}.$ Therefore, 
    \begin{align*}
    &\left\Vert \mathbf{E}\left[\boldsymbol{\omega}_{t_1, t_1 - t_2 - 1}^\top\mathbf{A}_{t_1t_2}\mathbf{z}_{t_2}\mid\mathcal{F}_{t_1, j}\right] - \mathbf{E}\left[\boldsymbol{\omega}_{t_1, t_1 - t_2 - 1}^\top\mathbf{A}_{t_1t_2}\mathbf{z}_{t_2}\mid\mathcal{F}_{t_1, j - 1}\right]\right\Vert_{M/2}\\
    &\leq \left\Vert \mathbf{E}\left[\boldsymbol{\omega}_{t_1, j - 1}^\top\mathbf{A}_{t_1t_2}\mathbf{z}_{t_2}\mid\mathcal{F}_{t_1, j}\right] - \mathbf{E}\left[\boldsymbol{\omega}_{t_1, j - 1}^\top\mathbf{A}_{t_1t_2}\mathbf{z}_{t_2}\mid\mathcal{F}_{t_1, j - 1}\right]\right\Vert_{M/2}\\
    & + \left\Vert \left(\boldsymbol{\zeta}_{t_1, j - 1} - \boldsymbol{\zeta}_{t_1, t_1 - t_2 -  1}\right)^\top\mathbf{A}_{t_1t_2}\left(\mathbf{E}\left[\mathbf{z}_{t_2}\mid\mathcal{F}_{t_2, t_2 + j - t_1}\right] - \mathbf{E}\left[\mathbf{z}_{t_2}\mid\mathcal{F}_{t_2, t_2 + j - 1 - t_1}\right]\right)\right\Vert_{M/2}\\
    &\leq 2\left\Vert \boldsymbol{\omega}_{t_1, j - 1}^\top\mathbf{A}_{t_1t_2}\mathbf{z}_{t_2}\right\Vert_{M/2}\\
    & + \left\vert\mathbf{A}_{t_1t_2}\right\vert_2\left\Vert\ \left\vert\boldsymbol{\zeta}_{t_1, j - 1} - \boldsymbol{\zeta}_{t_1, t_1 - t_2 -  1}\right\vert_2\ \right\Vert_M\left\Vert\ \left\vert \mathbf{E}\left[\mathbf{z}_{t_2}\mid\mathcal{F}_{t_2, t_2 + j - t_1}\right] - \mathbf{E}\left[\mathbf{z}_{t_2}\mid\mathcal{F}_{t_2, t_2 + j - 1 - t_1}\right]\right\vert_2\ \right\Vert_M.
    \end{align*}
    We notice that 
    \begin{align*}
        \left\Vert \boldsymbol{\omega}_{t_1, j - 1}^\top\mathbf{A}_{t_1t_2}\mathbf{z}_{t_2}\right\Vert_{M/2}\leq \frac{Cda^\dagger}{j^\alpha}.
    \end{align*}
    On the other hand,
    \begin{equation}
    \begin{aligned}
        \left\Vert\ \left\vert\boldsymbol{\zeta}_{t_1, j - 1} - \boldsymbol{\zeta}_{t_1, t_1 - t_2 -  1}\right\vert_2\ \right\Vert_M
        &= \left\Vert \ \left\vert\boldsymbol{\omega}_{t_1, t_1 - t_2 - 1} - \boldsymbol{\omega}_{t_1, j -  1}\right\vert_2\ \right\Vert_M\\
        &\leq \left\Vert\ \left\vert \boldsymbol{\omega}_{t_1, t_1 - t_2 - 1}\right\vert_2\ \right\Vert_M + \left\Vert\ \left\vert \boldsymbol{\omega}_{t_1, j -  1}\right\vert_2\ \right\Vert_M\\
        &\leq \frac{C\sqrt{d}}{(t_1 - t_2)^\alpha} + \frac{C\sqrt{d}}{j^\alpha}\leq \frac{2C\sqrt{d}}{(t_1 - t_2)^\alpha},
    \end{aligned}
    \label{eq.zeta_transform}
    \end{equation}
    and 
    \begin{align*}
        &\left\Vert\ \left\vert \mathbf{E}\left[\mathbf{z}_{t_2}\mid\mathcal{F}_{t_2, t_2 + j - t_1}\right] - \mathbf{E}\left[\mathbf{z}_{t_2}\mid\mathcal{F}_{t_2, t_2 + j - 1 - t_1}\right]\right\vert_2\ \right\Vert_M\\
        &\leq \sqrt{\sum_{i = 1}^d\left\Vert \mathbf{E}\left[\mathbf{z}_{t_2}^{(i)}\mid\mathcal{F}_{t_2, t_2 + j - t_1}\right] - \mathbf{E}\left[\mathbf{z}_{t_2}^{(i)}\mid\mathcal{F}_{t_2, t_2 + j - 1 - t_1}\right]\right\Vert^2_M}\leq \sqrt{d}\delta_{t_2 + j - t_1},
    \end{align*}
    so 
    \begin{align*}
        \left\Vert \mathbf{E}\left[\boldsymbol{\omega}_{t_1, t_1 - t_2 - 1}^\top\mathbf{A}_{t_1t_2}\mathbf{z}_{t_2}\mid\mathcal{F}_{t_1, j}\right] - \mathbf{E}\left[\boldsymbol{\omega}_{t_1, t_1 - t_2 - 1}^\top\mathbf{A}_{t_1t_2}\mathbf{z}_{t_2}\mid\mathcal{F}_{t_1, j - 1}\right]\right\Vert_{M/2}
        \leq \frac{Cda^\dagger}{j^\alpha} + \frac{Cda^\dagger\delta_{t_2+j-t_1}}{(t_1 - t_2)^\alpha}.
    \end{align*}
    If $j\leq \lambda_2,$ then 
    \begin{align*}
        &\sum_{t_2 = (t_1 - \lambda_2)\vee 1}^{t_1 - \lambda_1}\left\Vert \mathbf{E}\left[\boldsymbol{\omega}_{t_1, t_1 - t_2 - 1}^\top\mathbf{A}_{t_1t_2}\mathbf{z}_{t_2}\mid\mathcal{F}_{t_1, j}\right] - \mathbf{E}\left[\boldsymbol{\omega}_{t_1, t_1 - t_2 - 1}^\top\mathbf{A}_{t_1t_2}\mathbf{z}_{t_2}\mid\mathcal{F}_{t_1, j - 1}\right]\right\Vert_{M/2}\\
        &\leq  \sum_{t_2 = (t_1 - \lambda_2)\vee 1}^{t_1 - \lambda_1} \frac{Cda^\dagger}{(t_1 - t_2)^\alpha}\leq \frac{C_1da^\dagger}{\lambda_1^{\alpha  - 1}},
    \end{align*}
    and 
    \begin{align*}
        &\left\Vert\sum_{t_1 = 1 + \lambda_1}^T\sum_{t_2 = (t_1 - \lambda_2)\vee 1}^{t_1 - \lambda_1}\left(\mathbf{E}\left[\boldsymbol{\omega}_{t_1, t_1 - t_2 - 1}^\top\mathbf{A}_{t_1t_2}\mathbf{z}_{t_2}\mid\mathcal{F}_{t_1, j}\right] - \mathbf{E}\left[\boldsymbol{\omega}_{t_1, t_1 - t_2 - 1}^\top\mathbf{A}_{t_1t_2}\mathbf{z}_{t_2}\mid\mathcal{F}_{t_1, j - 1}\right]\right)\right\Vert_{M/2}\\
        &\leq \frac{Cda^\dagger\sqrt{T}}{\lambda_1^{\alpha  - 1}}.
    \end{align*}    
    If $j > \lambda_2,$ then for any $t_2 = (t_1 - \lambda_2)\vee 1,\cdots, t_1 - \lambda_1,$ $t_1 - t_2 < j,$ so 
    \begin{align*}
        &\sum_{t_2 = (t_1 - \lambda_2)\vee 1}^{t_1 - \lambda_1}\left\Vert \mathbf{E}\left[\boldsymbol{\omega}_{t_1, t_1 - t_2 - 1}^\top\mathbf{A}_{t_1t_2}\mathbf{z}_{t_2}\mid\mathcal{F}_{t_1, j}\right] - \mathbf{E}\left[\boldsymbol{\omega}_{t_1, t_1 - t_2 - 1}^\top\mathbf{A}_{t_1t_2}\mathbf{z}_{t_2}\mid\mathcal{F}_{t_1, j - 1}\right]\right\Vert_{M/2}\\
        &\leq \sum_{t_2 = (t_1 - \lambda_2)\vee 1}^{t_1 - \lambda_1}\frac{Cda^\dagger}{j^\alpha} + \sum_{t_2 = (t_1 - \lambda_2)\vee 1}^{t_1 - \lambda_1}\frac{Cda^\dagger\delta_{t_2+j-t_1}}{(t_1 - t_2)^\alpha}\\
        &\leq \frac{C_1da^\dagger(\lambda_2 - \lambda_1)}{j^\alpha} + \frac{Cda^\dagger}{\lambda_1^\alpha}\sum_{t_2 = (t_1 - \lambda_2)\vee 1}^{t_1 - \lambda_1}\delta_{t_2+j-t_1}
        \leq \frac{C_1da^\dagger(\lambda_2 - \lambda_1)}{j^\alpha} + \frac{C_1da^\dagger}{\lambda_1^\alpha (j - \lambda_2)^\alpha},
    \end{align*}
    and 
    \begin{equation}
    \begin{aligned}
        &\left\Vert\sum_{t_1 = 1 + \lambda_1}^T\sum_{t_2 = (t_1 - \lambda_2)\vee 1}^{t_1 - \lambda_1}\left(\mathbf{E}\left[\boldsymbol{\omega}_{t_1, t_1 - t_2 - 1}^\top\mathbf{A}_{t_1t_2}\mathbf{z}_{t_2}\mid\mathcal{F}_{t_1, j}\right] - \mathbf{E}\left[\boldsymbol{\omega}_{t_1, t_1 - t_2 - 1}^\top\mathbf{A}_{t_1t_2}\mathbf{z}_{t_2}\mid\mathcal{F}_{t_1, j - 1}\right]\right)\right\Vert_{M/2}\\
        &\leq \frac{Cda^\dagger(\lambda_2 - \lambda_1)\sqrt{T}}{j^\alpha} + \frac{Cda^\dagger\sqrt{T}}{\lambda_1^\alpha (j - \lambda_2)^\alpha}.
    \end{aligned}
    \label{eq.moment_j_greater_lambda_2}
    \end{equation}
    From these observations, we have 
    \begin{equation}
        \begin{aligned}
            II &\leq \sum_{j = 0}^{\lambda_2} \frac{Cda^\dagger\sqrt{T}}{\lambda_1^{\alpha  - 1}} + \sum_{j = \lambda_2 + 1}^\infty\frac{Cda^\dagger(\lambda_2 - \lambda_1)\sqrt{T}}{j^\alpha} + \sum_{j = \lambda_2 + 1}^\infty\frac{Cda^\dagger\sqrt{T}}{\lambda_1^\alpha (j - \lambda_2)^\alpha}\\
            &\leq \frac{C_1da^\dagger\sqrt{T}\lambda_2}{\lambda_1^{\alpha - 1}} + \frac{Cda^\dagger\sqrt{T}}{\lambda_2^{\alpha - 2}} + \frac{C_1da^\dagger\sqrt{T}}{\lambda_1^\alpha}\leq \frac{C_2da^\dagger\sqrt{T}\lambda_2}{\lambda_1^{\alpha - 1}}.
            \label{eq.sec_part_delta_z}
        \end{aligned}
    \end{equation}
    For $d = O(T)$ and $\kappa_1 > \frac{5}{2(\alpha - 2)},$ we have 
    \begin{align*}
        \frac{da^\dagger\sqrt{T}\lambda_2}{\lambda_1^{\alpha - 1}}\leq \frac{CT a^\dagger T\sqrt{T}}{T^{5/2}} = Ca^\dagger =  O\left(\sqrt{d\sum_{t_1 = 1 + \lambda_1}^T\sum_{t_2 = (t_1 - \lambda_2)\vee 1}^{t_1 - \lambda_1}\left\vert \mathbf{A}_{t_1t_2}\right\vert^2_2}\right).
    \end{align*}
From \eqref{eq.fir_part_delta_z} and \eqref{eq.sec_part_delta_z}, we prove \eqref{eq.product_quadratic_forms}.

For $S > 2\lambda_2$ and $t_1 - t_2\leq \lambda_2 < S,$ define $\boldsymbol{\zeta}_{t_1,t_1 - t_2 - 1}$ and $\boldsymbol{\omega}_{t_1, t_1 - t_2 - 1}$ as in \eqref{eq.def_zeta_omega} with $s = t_1 - t_2 - 1.$ Then $\boldsymbol{\zeta}_{t_1,t_1 - t_2 - 1}$ is measurable in $\mathcal{F}_{t_1,S},$ and 

\begin{align*}
    \mathbf{E}\left[\mathbf{z}_{t_1}^\top\mathbf{A}_{t_1t_2}\mathbf{z}_{t_2}\mid\mathcal{F}_{t_1,S}\right] = \boldsymbol{\zeta}_{t_1,t_1 - t_2 - 1}^\top\mathbf{A}_{t_1t_2}\mathbf{E}\left[\mathbf{z}_{t_2}\mid\mathcal{F}_{t_2, t_2 + S - t_1}\right] +
    \mathbf{E}\left[\boldsymbol{\omega}_{t_1, t_1 - t_2 - 1}^\top\mathbf{A}_{t_1t_2}\mathbf{z}_{t_2}\mid\mathcal{F}_{t_1,S}\right],
\end{align*}
and 
\begin{align*}
    &\left\Vert\sum_{t_1 = 1 + \lambda_1}^T\sum_{t_2 = (t_1 - \lambda_2)\vee 1}^{t_1 - \lambda_1}\left(\mathbf{z}_{t_1}^\top\mathbf{A}_{t_1t_2}\mathbf{z}_{t_2} - \mathbf{E}\left[\mathbf{z}_{t_1}^\top\mathbf{A}_{t_1t_2}\mathbf{z}_{t_2}\mid\mathcal{F}_{t_1, S}\right]\right)\right\Vert_{M/2}\\
    &\leq \left\Vert\sum_{t_1 = 1 + \lambda_1}^T\sum_{t_2 = (t_1 - \lambda_2)\vee 1}^{t_1 - \lambda_1}
    \boldsymbol{\zeta}_{t_1,t_1 - t_2 - 1}^\top\mathbf{A}_{t_1t_2}\left(\mathbf{z}_{t_2} - \mathbf{E}\left[\mathbf{z}_{t_2}\mid\mathcal{F}_{t_2, t_2 + S - t_1}\right]\right)
    \right\Vert_{M/2}\\
    & + \left\Vert\sum_{t_1 = 1 + \lambda_1}^T\sum_{t_2 = (t_1 - \lambda_2)\vee 1}^{t_1 - \lambda_1}
    \left(\boldsymbol{\omega}_{t_1, t_1 - t_2 - 1}^\top\mathbf{A}_{t_1t_2}\mathbf{z}_{t_2} - \mathbf{E}\left[\boldsymbol{\omega}_{t_1, t_1 - t_2 - 1}^\top\mathbf{A}_{t_1t_2}\mathbf{z}_{t_2}\mid\mathcal{F}_{t_1,S}\right]\right)
    \right\Vert_{M/2}.
\end{align*}
Since
\begin{align*}
    &\sum_{t_1 = 1 + \lambda_1}^T\sum_{t_2 = (t_1 - \lambda_2)\vee 1}^{t_1 - \lambda_1}
    \boldsymbol{\zeta}_{t_1,t_1 - t_2 - 1}^\top\mathbf{A}_{t_1t_2}\left(\mathbf{z}_{t_2} - \mathbf{E}\left[\mathbf{z}_{t_2}\mid\mathcal{F}_{t_2, t_2 + S - t_1}\right]\right)\\
    & = \sum_{t_2 = 1}^{T - \lambda_1}\sum_{t_1 = t_2 + \lambda_1}^{(t_2 + \lambda_2)\wedge T}\boldsymbol{\zeta}_{t_1,t_1 - t_2 - 1}^\top\mathbf{A}_{t_1t_2}\left(\mathbf{z}_{t_2} - \mathbf{E}\left[\mathbf{z}_{t_2}\mid\mathcal{F}_{t_2, t_2 + S - t_1}\right]\right),
\end{align*}
and 
\begin{align*}
    &\left\Vert
    \sum_{t_2 = 1}^{T - \lambda_1}\sum_{t_1 = t_2 + \lambda_1}^{(t_2 + \lambda_2)\wedge T}\boldsymbol{\zeta}_{t_1,t_1 - t_2 - 1}^\top\mathbf{A}_{t_1t_2}\left(\mathbf{z}_{t_2} - \mathbf{E}\left[\mathbf{z}_{t_2}\mid\mathcal{F}_{t_2, t_2 + S - t_1}\right]\right)
    \right\Vert_{M/2}\\
    & = \left\Vert
    \sum_{t_2 = 1}^{T - \lambda_1}\sum_{t_1 = t_2 + \lambda_1}^{(t_2 + \lambda_2)\wedge T}\sum_{j = t_2 + S  - t_1 + 1}^\infty\boldsymbol{\zeta}_{t_1,t_1 - t_2 - 1}^\top\mathbf{A}_{t_1t_2}\left(\mathbf{E}\left[\mathbf{z}_{t_2}\mid\mathcal{F}_{t_2, j}\right] - \mathbf{E}\left[\mathbf{z}_{t_2}\mid\mathcal{F}_{t_2, j - 1}\right]\right)
    \right\Vert_{M/2}\\
    &\leq \sum_{j = S + 1 - \lambda_2}^\infty\left\Vert \sum_{t_2 = 1}^{T - \lambda_1}\sum_{t_1 = (t_2 + \lambda_1)\vee (t_2 + S + 1 - j)}^{(t_2 + \lambda_2)\wedge T}\boldsymbol{\zeta}_{t_1,t_1 - t_2 - 1}^\top\mathbf{A}_{t_1t_2}\left(\mathbf{E}\left[\mathbf{z}_{t_2}\mid\mathcal{F}_{t_2, j}\right] - \mathbf{E}\left[\mathbf{z}_{t_2}\mid\mathcal{F}_{t_2, j - 1}\right]\right)\right\Vert_{M/2}.
\end{align*}
For any $z = 1,\cdots, T - \lambda_1$ and any $j\geq S + 1 - \lambda_2,$ define 
\begin{align*}
   H_{z,j} = \sum_{t_2 = T-\lambda_1 + 1 - z}^{T - \lambda_1}\sum_{t_1 = (t_2 + \lambda_1)\vee (t_2 + S + 1 - j)}^{(t_2 + \lambda_2)\wedge T}\boldsymbol{\zeta}_{t_1,t_1 - t_2 - 1}^\top\mathbf{A}_{t_1t_2}\left(\mathbf{E}\left[\mathbf{z}_{t_2}\mid\mathcal{F}_{t_2, j}\right] - \mathbf{E}\left[\mathbf{z}_{t_2}\mid\mathcal{F}_{t_2, j - 1}\right]\right)
\end{align*}
and $\mathcal{H}_{z,j} $ the $\sigma$-field generated by $e_T,e_{T-1},\cdots, e_{T-\lambda_1 + 1 - z - j}.$ Then $\mathcal{H}_{z,j}\subset \mathcal{H}_{z + 1,j},$ $ H_{z,j} $ is measurable in $\mathcal{H}_{z,j},$ and 
\begin{align*}
    \mathbf{E}\left[H_{z + 1,j}  - H_{z,j} \mid \mathcal{H}_{z,j} \right] &= \sum_{t_1 = (T-\lambda_1 - z + \lambda_1)\vee (T-\lambda_1 - z + S + 1 - j)}^{(T-\lambda_1 - z + \lambda_2)\wedge T}\mathbf{E}\left[\boldsymbol{\zeta}_{t_1,t_1 - T + \lambda_1 + z - 1}^\top\mathbf{A}_{t_1(T-\lambda_1 - z)}\right.\\
    &\left.\left(\mathbf{E}\left[\mathbf{z}_{T-\lambda_1 - z}\mid\mathcal{F}_{T-\lambda_1 - z, j}\right] - \mathbf{E}\left[\mathbf{z}_{T-\lambda_1 - z}\mid\mathcal{F}_{T-\lambda_1 - z, j - 1}\right]\right)\mid \mathcal{H}_{z,j} \right]\\
    & = \sum_{t_1 = (T-\lambda_1 - z + \lambda_1)\vee (T-\lambda_1 - z + S + 1 - j)}^{(T-\lambda_1 - z + \lambda_2)\wedge T}
    \boldsymbol{\zeta}_{t_1,t_1 - T + \lambda_1 + z - 1}^\top\mathbf{A}_{t_1(T-\lambda_1 - z)}\\
    &\left(\mathbf{E}\left[\mathbf{z}_{T-\lambda_1 - z}\mid\mathcal{F}_{T-\lambda_1 - z, j - 1}\right] - \mathbf{E}\left[\mathbf{z}_{T-\lambda_1 - z}\mid\mathcal{F}_{T-\lambda_1 - z, j - 1}\right]\right)\\
    & = 0,
\end{align*}
so $H_{z,j}$ forms a martingale. From Theorem 1.1 of \cite{MR0400380}, 
\begin{align*}
   &\left\Vert \sum_{t_2 = 1}^{T - \lambda_1}\sum_{t_1 = (t_2 + \lambda_1)\vee (t_2 + S + 1 - j)}^{(t_2 + \lambda_2)\wedge T}\boldsymbol{\zeta}_{t_1,t_1 - t_2 - 1}^\top\mathbf{A}_{t_1t_2}\left(\mathbf{E}\left[\mathbf{z}_{t_2}\mid\mathcal{F}_{t_2, j}\right] - \mathbf{E}\left[\mathbf{z}_{t_2}\mid\mathcal{F}_{t_2, j - 1}\right]\right)\right\Vert_{M/2}\\
   &\leq C\sqrt{\sum_{t_2 = 1}^{T - \lambda_1}\left\Vert \sum_{t_1 = (t_2 + \lambda_1)\vee (t_2 + S + 1 - j)}^{(t_2 + \lambda_2)\wedge T}\boldsymbol{\zeta}_{t_1,t_1 - t_2 - 1}^\top\mathbf{A}_{t_1t_2}\left(\mathbf{E}\left[\mathbf{z}_{t_2}\mid\mathcal{F}_{t_2, j}\right] - \mathbf{E}\left[\mathbf{z}_{t_2}\mid\mathcal{F}_{t_2, j - 1}\right]\right)\right\Vert^2_{M/2}}.
\end{align*}
From \eqref{eq.linear_truncate_moment} and Examples 1.22 of \cite{MR2767184}, for any vector $\mathbf{a}\in\mathbf{R}^d$
\begin{align*}
    &\mathbf{E}\left[
    \left\vert \sum_{t_1 = (t_2 + \lambda_1)\vee (t_2 + S + 1 - j)}^{(t_2 + \lambda_2)\wedge T}\boldsymbol{\zeta}_{t_1,t_1 - t_2 - 1}^\top\mathbf{A}_{t_1t_2}\left(\mathbf{E}\left[\mathbf{z}_{t_2}\mid\mathcal{F}_{t_2, j}\right] - \mathbf{E}\left[\mathbf{z}_{t_2}\mid\mathcal{F}_{t_2, j - 1}\right]\right)\right\vert^{M / 2}\right.\\
    &\mid \mathbf{E}\left[\mathbf{z}_{t_2}\mid\mathcal{F}_{t_2, j}\right] - \mathbf{E}\left[\mathbf{z}_{t_2}\mid\mathcal{F}_{t_2, j - 1}\right] = \mathbf{a}
    \Bigg]\\
    & = \mathbf{E}\left[
    \left\vert \sum_{t_1 = (t_2 + \lambda_1)\vee (t_2 + S + 1 - j)}^{(t_2 + \lambda_2)\wedge T}\boldsymbol{\zeta}_{t_1,t_1 - t_2 - 1}^\top\mathbf{A}_{t_1t_2}\mathbf{a}\right\vert^{M / 2}
    \right]\\
    &\leq C\left\vert\mathbf{a}\right\vert_2^{M/2}\left( \sum_{t_1 = (t_2 + \lambda_1)\vee (t_2 + S + 1 - j)}^{(t_2 + \lambda_2)\wedge T}\left\vert \mathbf{A}_{t_1t_2}\right\vert_2^2\right)^{M/4},
\end{align*}
which makes 
\begin{align*}
    &\left\Vert \sum_{t_1 = (t_2 + \lambda_1)\vee (t_2 + S + 1 - j)}^{(t_2 + \lambda_2)\wedge T}\boldsymbol{\zeta}_{t_1,t_1 - t_2 - 1}^\top\mathbf{A}_{t_1t_2}\left(\mathbf{E}\left[\mathbf{z}_{t_2}\mid\mathcal{F}_{t_2, j}\right] - \mathbf{E}\left[\mathbf{z}_{t_2}\mid\mathcal{F}_{t_2, j - 1}\right]\right)\right\Vert_{M/2}\\
    &\leq C\sqrt{\sum_{t_1 = (t_2 + \lambda_1)\vee (t_2 + S + 1 - j)}^{(t_2 + \lambda_2)\wedge T}\left\vert \mathbf{A}_{t_1t_2}\right\vert_2^2}\ \left\Vert\ \left\vert\mathbf{E}\left[\mathbf{z}_{t_2}\mid\mathcal{F}_{t_2, j}\right] - \mathbf{E}\left[\mathbf{z}_{t_2}\mid\mathcal{F}_{t_2, j - 1}\right]\right\vert_2\ \right\Vert_{M/2}.
\end{align*}
Since 
\begin{align*}
    &\left\Vert\ \left\vert\mathbf{E}\left[\mathbf{z}_{t_2}\mid\mathcal{F}_{t_2, j}\right] - \mathbf{E}\left[\mathbf{z}_{t_2}\mid\mathcal{F}_{t_2, j - 1}\right]\right\vert_2\ \right\Vert_{M/2}\\
    &\leq \sqrt{\sum_{i = 1}^d\left\Vert\mathbf{E}\left[\mathbf{z}_{t_2}^{(i)}\mid\mathcal{F}_{t_2, j}\right] - \mathbf{E}\left[\mathbf{z}_{t_2}^{(i)}\mid\mathcal{F}_{t_2, j - 1}\right]\right\Vert_{M/2}^2}\leq \sqrt{d}\delta_j,
\end{align*}
we have 
\begin{align*}
    &\left\Vert \sum_{t_1 = (t_2 + \lambda_1)\vee (t_2 + S + 1 - j)}^{(t_2 + \lambda_2)\wedge T}\boldsymbol{\zeta}_{t_1,t_1 - t_2 - 1}^\top\mathbf{A}_{t_1t_2}\left(\mathbf{E}\left[\mathbf{z}_{t_2}\mid\mathcal{F}_{t_2, j}\right] - \mathbf{E}\left[\mathbf{z}_{t_2}\mid\mathcal{F}_{t_2, j - 1}\right]\right)\right\Vert_{M/2}\\
    &\leq C\delta_j\sqrt{d\sum_{t_1 = (t_2 + \lambda_1)\vee (t_2 + S + 1 - j)}^{(t_2 + \lambda_2)\wedge T}\left\vert \mathbf{A}_{t_1t_2}\right\vert_2^2}
    \leq C\delta_j\sqrt{d\sum_{t_1 = (t_2 + \lambda_1)}^{(t_2 + \lambda_2)\wedge T}\left\vert \mathbf{A}_{t_1t_2}\right\vert_2^2},
\end{align*}
which makes 
\begin{align*}
    &\left\Vert \sum_{t_2 = 1}^{T - \lambda_1}\sum_{t_1 = (t_2 + \lambda_1)\vee (t_2 + S + 1 - j)}^{(t_2 + \lambda_2)\wedge T}\boldsymbol{\zeta}_{t_1,t_1 - t_2 - 1}^\top\mathbf{A}_{t_1t_2}\left(\mathbf{E}\left[\mathbf{z}_{t_2}\mid\mathcal{F}_{t_2, j}\right] - \mathbf{E}\left[\mathbf{z}_{t_2}\mid\mathcal{F}_{t_2, j - 1}\right]\right)\right\Vert_{M/2}\\
    &\leq C\delta_j\sqrt{d\sum_{t_1 = 1 + \lambda_1}^{T}\sum_{t_2 = (t_1  - \lambda_2)\vee 1}^{t_1 - \lambda_1}\left\vert \mathbf{A}_{t_1t_2}\right\vert_2^2}.
\end{align*}
Since $d = O(T)$ and $S>\lambda_2,$ we have 
\begin{equation}
    \begin{aligned}
        &\left\Vert
    \sum_{t_2 = 1}^{T - \lambda_1}\sum_{t_1 = t_2 + \lambda_1}^{(t_2 + \lambda_2)\wedge T}\boldsymbol{\zeta}_{t_1,t_1 - t_2 - 1}^\top\mathbf{A}_{t_1t_2}\left(\mathbf{z}_{t_2} - \mathbf{E}\left[\mathbf{z}_{t_2}\mid\mathcal{F}_{t_2, t_2 + S - t_1}\right]\right)
    \right\Vert_{M/2}\\
    & \leq C\sum_{j = S+1 - \lambda_2}^\infty\delta_j\sqrt{d\sum_{t_1 = 1 + \lambda_1}^{T}\sum_{t_2 = (t_1  - \lambda_2)\vee 1}^{t_1 - \lambda_1}\left\vert \mathbf{A}_{t_1t_2}\right\vert_2^2}\\
    &\leq \frac{C_1}{(S+1-\lambda_2)^\alpha}\sqrt{d\sum_{t_1 = 1 + \lambda_1}^{T}\sum_{t_2 = (t_1  - \lambda_2)\vee 1}^{t_1 - \lambda_1}\left\vert \mathbf{A}_{t_1t_2}\right\vert_2^2}\leq \frac{C_2a^\dagger T\sqrt{\lambda_2}}{(S-\lambda_2)^\alpha}\leq \frac{C_3a^\dagger T}{(S-\lambda_2)^{\alpha - 1}}
    \end{aligned}
    \label{eq.first_moment_truncate_part}
\end{equation}
On the other hand, notice that 
\begin{align*}
    &\left\Vert\sum_{t_1 = 1 + \lambda_1}^T\sum_{t_2 = (t_1 - \lambda_2)\vee 1}^{t_1 - \lambda_1}
    \left(\boldsymbol{\omega}_{t_1, t_1 - t_2 - 1}^\top\mathbf{A}_{t_1t_2}\mathbf{z}_{t_2} - \mathbf{E}\left[\boldsymbol{\omega}_{t_1, t_1 - t_2 - 1}^\top\mathbf{A}_{t_1t_2}\mathbf{z}_{t_2}\mid\mathcal{F}_{t_1,S}\right]\right)
    \right\Vert_{M/2}\\
    &\leq \sum_{j = S + 1}^\infty\left\Vert\sum_{t_1 = 1 + \lambda_1}^T\sum_{t_2 = (t_1 - \lambda_2)\vee 1}^{t_1 - \lambda_1}
    \left(\mathbf{E}\left[\boldsymbol{\omega}_{t_1, t_1 - t_2 - 1}^\top\mathbf{A}_{t_1t_2}\mathbf{z}_{t_2}\mid\mathcal{F}_{t_1, j}\right] - \mathbf{E}\left[\boldsymbol{\omega}_{t_1, t_1 - t_2 - 1}^\top\mathbf{A}_{t_1t_2}\mathbf{z}_{t_2}\mid\mathcal{F}_{t_1, j - 1}\right]\right)
    \right\Vert_{M/2}.
\end{align*}
For $j\geq S+1 > \lambda_2,$ from \eqref{eq.moment_j_greater_lambda_2},
\begin{align*}
     &\left\Vert\sum_{t_1 = 1 + \lambda_1}^T\sum_{t_2 = (t_1 - \lambda_2)\vee 1}^{t_1 - \lambda_1}\left(\mathbf{E}\left[\boldsymbol{\omega}_{t_1, t_1 - t_2 - 1}^\top\mathbf{A}_{t_1t_2}\mathbf{z}_{t_2}\mid\mathcal{F}_{t_1, j}\right] - \mathbf{E}\left[\boldsymbol{\omega}_{t_1, t_1 - t_2 - 1}^\top\mathbf{A}_{t_1t_2}\mathbf{z}_{t_2}\mid\mathcal{F}_{t_1, j - 1}\right]\right)\right\Vert_{M/2}\\
        &\leq \frac{Cda^\dagger(\lambda_2 - \lambda_1)\sqrt{T}}{j^\alpha} + \frac{Cda^\dagger\sqrt{T}}{\lambda_1^\alpha (j - \lambda_2)^\alpha},
\end{align*}
also notice that $S > \lambda_2,$
\begin{equation}
    \begin{aligned}
        &\left\Vert\sum_{t_1 = 1 + \lambda_1}^T\sum_{t_2 = (t_1 - \lambda_2)\vee 1}^{t_1 - \lambda_1}
    \left(\boldsymbol{\omega}_{t_1, t_1 - t_2 - 1}^\top\mathbf{A}_{t_1t_2}\mathbf{z}_{t_2} - \mathbf{E}\left[\boldsymbol{\omega}_{t_1, t_1 - t_2 - 1}^\top\mathbf{A}_{t_1t_2}\mathbf{z}_{t_2}\mid\mathcal{F}_{t_1,S}\right]\right)
    \right\Vert_{M/2}\\
    &\leq \sum_{j = S + 1}^\infty\left(\frac{Cda^\dagger(\lambda_2 - \lambda_1)\sqrt{T}}{j^\alpha} + \frac{Cda^\dagger\sqrt{T}}{\lambda_1^\alpha (j - \lambda_2)^\alpha}\right)\\
    &\leq \frac{C_1da^\dagger(\lambda_2 - \lambda_1)\sqrt{T}}{(S+1)^{\alpha - 1}} + \frac{C_1da^\dagger\sqrt{T}}{\lambda_1^\alpha (S - \lambda_2)^{\alpha - 1}}\leq \frac{C_2T^{3/2}a^\dagger}{(S+1)^{\alpha - 2}}.
    \end{aligned}
    \label{eq.second_moment_truncate_part}
\end{equation}
From \eqref{eq.first_moment_truncate_part} and \eqref{eq.second_moment_truncate_part}, we prove \eqref{eq.truncate_whole}.
\end{proof}

\begin{remark}
\label{remark.independent_linear_quadratic}
    To make a further illustration of the bounds in equations \eqref{eq.linear_combination} and \eqref{eq.product_quadratic_forms}, suppose $\mathbf{z}_t, t\in\mathbf{Z}$ are mutually independent with independent entries, i.e., 
    $\mathbf{z}_t^{(i)},  t\in\mathbf{Z}, i = 1,\cdots,d$ are mutually independent. In this case, from Theorem 2 of \cite{MR0133849}, 
\begin{align*}
    \left\Vert
        \sum_{t = 1}^T \mathbf{a}_t^\top\mathbf{z}_t
        \right\Vert_M = \left\Vert
        \sum_{t = 1}^T\sum_{i = 1}^d \mathbf{a}_t^{(i)}\mathbf{z}_t^{(i)}
        \right\Vert_M\leq C\sqrt{\sum_{t = 1}^T\sum_{i = 1}^d\left\vert \mathbf{a}_t^{(i)}\right\vert^2} = C\sqrt{\sum_{t = 1}^T\left\vert \mathbf{a}_t\right\vert_2^2},
\end{align*}
which is exactly the moment bound in equation \eqref{eq.linear_combination}. Furthermore, define the random vector $\boldsymbol{z}\in\mathbf{R}^{Td}$ and the matrix $\boldsymbol{B}\in\mathbf{R}^{Td\times Td}$ as follows:
\begin{align*}
    \boldsymbol{z} = \left[
    \begin{matrix}
        \mathbf{z}_{1}\\
        \mathbf{z}_2\\
        \vdots\\
        \mathbf{z}_T
    \end{matrix}
    \right],\quad 
    \boldsymbol{B} = \left[
    \begin{matrix}
        \boldsymbol{B}_{11} & \boldsymbol{B}_{12} & \cdots & \boldsymbol{B}_{1T}\\
        \boldsymbol{B}_{21} & \boldsymbol{B}_{22} & \cdots & \boldsymbol{B}_{2T}\\
        \vdots & \vdots & \cdots & \vdots \\
        \boldsymbol{B}_{T1} & \boldsymbol{B}_{T2} & \cdots & \boldsymbol{B}_{TT}
    \end{matrix}
    \right],
\end{align*}
where $\boldsymbol{B}_{t_1t_2}\in\mathbf{R}^{d\times d}$ is defined as follows:
\begin{align*}
   \boldsymbol{B}_{t_1t_2} = 
    \begin{cases}
        \mathbf{A}_{t_1t_2} & \text{if}\quad t_1 \geq 1 + \lambda_1 \quad \text{and}\quad    (t_1 - \lambda_2) \vee 1\leq t_2 \leq t_1 - \lambda_1,\\
        0 & \text{Otherwise}.
    \end{cases}
\end{align*}
Then 
\begin{align*}
    &\left\Vert
    \sum_{t_1 = 1 + \lambda_1}^T\sum_{t_2 = (t_1 - \lambda_2)\vee 1}^{t_1 - \lambda_1}\left(\mathbf{z}_{t_1}\mathbf{A}_{t_1t_2}\mathbf{z}_{t_2} - \mathbf{E}\left[\mathbf{z}_{t_1}\mathbf{A}_{t_1t_2}\mathbf{z}_{t_2}\right]\right)
    \right\Vert_{M/2}\\
    &= \left\Vert\sum_{t_1 = 1}^T\sum_{t_2 = 1}^T\left(\mathbf{z}_{t_1}\boldsymbol{B}_{t_1t_2}\mathbf{z}_{t_2} - \mathbf{E}\left[\mathbf{z}_{t_1}\boldsymbol{B}_{t_1t_2}\mathbf{z}_{t_2}\right]\right)\right\Vert_{M/2}\\
    & = \left\Vert
    \boldsymbol{z}^\top \boldsymbol{B}\boldsymbol{z} - \mathbf{E}\left[\boldsymbol{z}^\top \boldsymbol{B}\boldsymbol{z}\right]
    \right\Vert_{M/2}.
\end{align*}
From Theorem 2 of \cite{MR0133849}, we have 
\begin{align*}
    \left\Vert
    \boldsymbol{z}^\top \boldsymbol{B}\boldsymbol{z} - \mathbf{E}\left[\boldsymbol{z}^\top \boldsymbol{B}\boldsymbol{z}\right]
    \right\Vert_{M/2}\leq C\left\vert \boldsymbol{B}\right\vert_{F} = C\sqrt{\sum_{t_1 = 1}^T\sum_{t_2 = 1}^T \left\vert \boldsymbol{B}_{t_1t_2}\right\vert_F^2}
    = C\sqrt{\sum_{t_1 = 1 + \lambda_1}^T\sum_{t_2 = (t_1 - \lambda_2)\vee 1}^{t_1 - \lambda_1}\left\vert \mathbf{A}_{t_1t_2}\right\vert_F^2}.
\end{align*}
We derive the moment bound in \eqref{eq.product_quadratic_forms} by further employing the inequality $\left\vert \mathbf{A}_{t_1t_2}\right\vert_F\leq \sqrt{d}\left\vert \mathbf{A}_{t_1t_2}\right\vert_2.$ As a result, we adopt a relatively loose moment bound to compensate for the effect of temporal dependence in the proof. However, the inequality $\left\vert \mathbf{A}_{t_1t_2}\right\vert_F\leq \sqrt{d}\left\vert \mathbf{A}_{t_1t_2}\right\vert_2$ does not change the order of the moment bound: If all singular values of $\mathbf{A}_{t_1t_2}$ 
are of the same order, then $\left\vert \mathbf{A}_{t_1t_2}\right\vert_F\asymp\sqrt{d}\left\vert \mathbf{A}_{t_1t_2}\right\vert_2.$ Since our primary interest lies in the order of the moment bound, replacing $\left\vert \mathbf{A}_{t_1t_2}\right\vert_F$ with $\sqrt{d}\left\vert \mathbf{A}_{t_1t_2}\right\vert_2$ does not lead to too much information loss. 
\end{remark}

We then present a corollary of Theorem \ref{theorem.linear_and_quadratic} that serves an important role in the proof of Theorem \ref{theorem.validity_of_bootstrap_algorithm}.

\begin{corollary}
    Suppose the conditions of Theorem \ref{theorem.linear_and_quadratic} hold true and suppose $h\in\mathbf{Z}$ satisfies $h = O(T^{\kappa_h})$ with $0<\kappa_h<1.$ Then there exists a constant $C>0$ such that for any $t_1 = 1+\lambda_1,\cdots, T,$ $ 1\vee (t_1 - \lambda_2)\leq t_2\leq t_1-\lambda_1,$ and 
    any real symmetric matrices $\mathbf{B}_u, u = 0,1,\cdots, h,$ we have
    \begin{equation}
        \left\Vert
        \sum_{u = 0}^{h\wedge (t_2 - 1)}\left(\mathbf{z}_{t_1- u}^\top\mathbf{B}_u\mathbf{z}_{t_2- u} - \mathbf{E}\left[\mathbf{z}_{t_1- u}^\top\mathbf{B}_u\mathbf{z}_{t_2- u} \right]\right)
        \right\Vert_{M/2}\leq C\sqrt{d\sum_{u = 0}^{h\wedge (t_2 - 1)}\left\vert\mathbf{B}_u\right\vert^2_2}.
        \label{eq.quadratic_lag}
    \end{equation}
    Furthermore, for any $S > 2\lambda_2 + h,$
    \begin{equation}
        \left\Vert
        \sum_{u = 0}^{h\wedge (t_2 - 1)}\left(\mathbf{z}_{t_1- u}^\top\mathbf{B}_u\mathbf{z}_{t_2- u} - \mathbf{E}\left[\mathbf{z}_{t_1- u}^\top\mathbf{B}_u\mathbf{z}_{t_2- u}\mid\mathcal{F}_{t_1,S}\right]\right)
        \right\Vert_{M/2}\leq \frac{C\sqrt{dh}b^\dagger}{(S-h-\lambda_2)^\alpha},
        \label{eq.organize_lag_S}
    \end{equation}
    where the notations $\lambda_1,\lambda_2$ coincide with Theorem \ref{theorem.linear_and_quadratic}, and $b^\dagger = \max_{u = 0, \cdots, h}\vert\mathbf{B}_u\vert_2.$ 
    \label{corollary.combination_lag}
\end{corollary}

\begin{proof}[Proof of Corollary \ref{corollary.combination_lag}]
For given $t_1$ and $t_2,$ we define the matrices $\mathbf{A}_{t_3t_4}$ as follows:
\begin{align*}
    \mathbf{A}_{t_3t_4} = 
    \begin{cases}
    \mathbf{B}_u & \text{if}\quad t_3 = t_1 - u\quad\text{and}\quad t_4 = t_2 - u,\quad\text{where}\quad u = 0,1,\cdots, h\wedge (t_2  -1),\\
    0 & \text{Otherwise.} 
    \end{cases}
\end{align*}
In such case, 
\begin{align*}
    &\sum_{t_3 = 1 + \lambda_1}^T\sum_{t_4 = (t_3 - \lambda_2)\vee 1}^{t_3 - \lambda_1}\left(
    \mathbf{z}_{t_3}^\top\mathbf{A}_{t_3t_4}\mathbf{z}_{t_4} - \mathbf{E}\left[\mathbf{z}_{t_3}^\top\mathbf{A}_{t_3t_4}\mathbf{z}_{t_4} \right]
    \right)\\
    &= \sum_{u = 0}^{h\wedge (t_2 - 1)}\left(\mathbf{z}_{t_1- u}^\top\mathbf{B}_u\mathbf{z}_{t_2- u} - \mathbf{E}\left[\mathbf{z}_{t_1- u}^\top\mathbf{B}_u\mathbf{z}_{t_2- u}\right]\right),
\end{align*}
 so from equation \eqref{eq.product_quadratic_forms},
 \begin{align*}
     \left\Vert
        \sum_{u = 0}^{h\wedge (t_2 - 1)}\left(\mathbf{z}_{t_1- u}^\top\mathbf{B}_u\mathbf{z}_{t_2- u} - \mathbf{E}\left[\mathbf{z}_{t_1- u}^\top\mathbf{B}_u\mathbf{z}_{t_2- u} \right]\right)
        \right\Vert_{M/2} &\leq C\sqrt{d\sum_{t_3 = 1 + \lambda_1}^T\sum_{t_4 = (t_3 - \lambda_2)\vee 1}^{t_3 - \lambda_1}\left\vert \mathbf{A}_{t_3t_4}\right\vert^2_2}\\
        & = C\sqrt{d\sum_{u = 0}^{h\wedge (t_2 - 1)}\left\vert \mathbf{B}_u\right\vert^2_2},
 \end{align*}
 which proves \eqref{eq.quadratic_lag}.

Consider the decomposition  $\mathbf{z}_{t_1- u} = \boldsymbol{\zeta}_{t_1 -u, t_1 - t_2 - 1} + \boldsymbol{\omega }_{t_1 -u, t_1 - t_2 - 1},$ where $\boldsymbol{\zeta}_{t_1 -u, t_1 - t_2 - 1} $ and $\boldsymbol{\omega }_{t_1 -u, t_1 - t_2 - 1}$ coincide with \eqref{eq.def_zeta_omega} with $s = t_1 - t_2 - 1.$ Then $\boldsymbol{\zeta}_{t_1 -u, t_1 - t_2 - 1}$ is measurable in $\mathcal{F}_{t_1,S}, $ and 
\begin{align*}
    &\left\Vert
        \sum_{u = 0}^{h\wedge (t_2 - 1)}\left(\mathbf{z}_{t_1- u}^\top\mathbf{B}_u\mathbf{z}_{t_2- u} - \mathbf{E}\left[\mathbf{z}_{t_1- u}^\top\mathbf{B}_u\mathbf{z}_{t_2- u}\mid\mathcal{F}_{t_1,S}\right]\right)
        \right\Vert_{M/2}\\
    &\leq     \left\Vert
        \sum_{u = 0}^{h\wedge (t_2 - 1)}\left(\boldsymbol{\zeta}_{t_1 -u, t_1 - t_2 - 1}^\top\mathbf{B}_u\left(\mathbf{z}_{t_2- u} - \mathbf{E}\left[\mathbf{z}_{t_2- u}\mid\mathcal{F}_{t_2 - u, t_2 +S- u-t_1}\right]\right)\right)
        \right\Vert_{M/2}\\
        & + \left\Vert
        \sum_{u = 0}^{h\wedge (t_2 - 1)}\left(\boldsymbol{\omega }_{t_1 -u, t_1 - t_2 - 1}^\top\mathbf{B}_u\mathbf{z}_{t_2- u} - \mathbf{E}\left[\boldsymbol{\omega }_{t_1 -u, t_1 - t_2 - 1}^\top\mathbf{B}_u\mathbf{z}_{t_2- u}\mid\mathcal{F}_{t_1,S}\right]\right)
        \right\Vert_{M/2}.
\end{align*}
For 
\begin{align*}
    &\left\Vert
        \sum_{u = 0}^{h\wedge (t_2 - 1)}\left(\boldsymbol{\zeta}_{t_1 -u, t_1 - t_2 - 1}^\top\mathbf{B}_u\left(\mathbf{z}_{t_2- u} - \mathbf{E}\left[\mathbf{z}_{t_2- u}\mid\mathcal{F}_{t_2 - u, t_2 +S- u-t_1}\right]\right)\right)
    \right\Vert_{M/2}\\
    & = \left\Vert
        \sum_{u = 0}^{h\wedge (t_2 - 1)}\sum_{j = t_2 + S - u - t_1 + 1}^\infty\left(\boldsymbol{\zeta}_{t_1 -u, t_1 - t_2 - 1}^\top\mathbf{B}_u\left(\mathbf{E}\left[\mathbf{z}_{t_2- u}\mid\mathcal{F}_{t_2 - u,j}\right] - \mathbf{E}\left[\mathbf{z}_{t_2- u}\mid\mathcal{F}_{t_2 - u, j-1}\right]\right)\right)
    \right\Vert_{M/2}\\
    &\leq \sum_{j = t_2 + S - h - t_1 + 1}^\infty\left\Vert
        \sum_{u = 0\vee(t_2 + S - t_1 + 1 - j)}^{h\wedge (t_2 - 1)}\left(\boldsymbol{\zeta}_{t_1 -u, t_1 - t_2 - 1}^\top\mathbf{B}_u\left(\mathbf{E}\left[\mathbf{z}_{t_2- u}\mid\mathcal{F}_{t_2 - u,j}\right]\right.\right.\right.\\
    &\left.\left.- \mathbf{E}\left[\mathbf{z}_{t_2- u}\mid\mathcal{F}_{t_2 - u, j-1}\right]\right)\right)
    \Bigg\Vert_{M/2}.
\end{align*}
For any $j = t_2 + S - h - t_1 + 1, \cdots,$ and any $z = 1,\cdots, h\wedge (t_2 - 1) - 0\vee(t_2 + S - t_1 + 1 - j) + 1,$ define 
\begin{align*}
    J_{z,j} = \sum_{u = 0\vee(t_2 + S - t_1 + 1 - j)}^{ 0\vee(t_2 + S - t_1 + 1 - j) + z - 1}\left(\boldsymbol{\zeta}_{t_1 -u, t_1 - t_2 - 1}^\top\mathbf{B}_u\left(\mathbf{E}\left[\mathbf{z}_{t_2- u}\mid\mathcal{F}_{t_2 - u,j}\right] - \mathbf{E}\left[\mathbf{z}_{t_2- u}\mid\mathcal{F}_{t_2 - u, j-1}\right]\right)\right)
\end{align*}
and $\mathcal{J}_{z,j}$ the $\sigma$-field generated by $e_T,e_{T-1},\cdots, e_{t_2 - 0\vee(t_2 + S - t_1 + 1 - j) - z + 1 - j}.$ Then $J_{z,j} $ is measurable in $\mathcal{J}_{z,j},$ $\mathcal{J}_{z,j}\subset \mathcal{J}_{z +1,j},$ and 
\begin{align*}
    &\mathbf{E}\left[J_{z + 1,j} - J_{z,j}\mid \mathcal{J}_{z,j}\right]\\
    &= \mathbf{E}\left[\left(\boldsymbol{\zeta}_{t_1 - 0\vee(t_2 + S - t_1 + 1 - j) - z, t_1 - t_2 - 1}^\top\mathbf{B}_{0\vee(t_2 + S - t_1 + 1 - j) + z}\right.\right.\\
    &\left(\mathbf{E}\left[\mathbf{z}_{t_2- 0\vee(t_2 + S - t_1 + 1 - j) - z}\mid\mathcal{F}_{t_2 - 0\vee(t_2 + S - t_1 + 1 - j) - z,j}\right]\right.\\
    &\left.\left.\left.
    - \mathbf{E}\left[\mathbf{z}_{t_2- 0\vee(t_2 + S - t_1 + 1 - j) - z}\mid\mathcal{F}_{t_2 - 0\vee(t_2 + S - t_1 + 1 - j) - z, j-1}\right]\right)\right)\mid \mathcal{J}_{z,j}\right]\\
    & = \boldsymbol{\zeta}_{t_1 - 0\vee(t_2 + S - t_1 + 1 - j) - z, t_1 - t_2 - 1}^\top\mathbf{B}_{0\vee(t_2 + S - t_1 + 1 - j) + z}\left(\mathbf{E}\left[\mathbf{z}_{t_2- 0\vee(t_2 + S - t_1 + 1 - j) - z}\mid\mathcal{F}_{t_2 - 0\vee(t_2 + S - t_1 + 1 - j) - z,j - 1}\right]\right.\\
    &\left. - \mathbf{E}\left[\mathbf{z}_{t_2- 0\vee(t_2 + S - t_1 + 1 - j) - z}\mid\mathcal{F}_{t_2 - 0\vee(t_2 + S - t_1 + 1 - j) - z,j - 1}\right]\right) = 0.
\end{align*}
Therefore, $J_{z,j}$ forms a martingale, and from   Theorem 1.1 of \cite{MR0400380}, 
\begin{align*}
    &\left\Vert
    \sum_{u = 0\vee(t_2 + S - t_1 + 1 - j)}^{h\wedge (t_2 - 1)}\left(\boldsymbol{\zeta}_{t_1 -u, t_1 - t_2 - 1}^\top\mathbf{B}_u\left(\mathbf{E}\left[\mathbf{z}_{t_2- u}\mid\mathcal{F}_{t_2 - u,j}\right] - \mathbf{E}\left[\mathbf{z}_{t_2- u}\mid\mathcal{F}_{t_2 - u, j-1}\right]\right)\right)
    \right\Vert_{M/2}\\
    &\leq C\sqrt{\sum_{u = 0\vee(t_2 + S - t_1 + 1 - j)}^{h\wedge (t_2 - 1)}\left\Vert \boldsymbol{\zeta}_{t_1 -u, t_1 - t_2 - 1}^\top\mathbf{B}_u\left(\mathbf{E}\left[\mathbf{z}_{t_2- u}\mid\mathcal{F}_{t_2 - u,j}\right] - \mathbf{E}\left[\mathbf{z}_{t_2- u}\mid\mathcal{F}_{t_2 - u, j-1}\right]\right)\right\Vert^2_{M/2}}.
\end{align*}
For any vector $\mathbf{a}\in\mathbf{R}^d,$
\begin{align*}
    \left\Vert\boldsymbol{\zeta}_{t_1 -u, t_1 - t_2 - 1}^\top\mathbf{B}_u\mathbf{a}\right\Vert_{M/2} 
    = \left\Vert\mathbf{E}\left[\mathbf{z}_{t_1-  u}^\top\mathbf{B}_u\mathbf{a}\mid\mathcal{F}_{t_1 - u, t_1 - t_2 - 1}\right]\right\Vert_{M/2}
    \leq \left\Vert\mathbf{z}_{t_1-  u}^\top\mathbf{B}_u\mathbf{a}\right\Vert_{M/2}\leq C\left\vert\mathbf{B}_u\right\vert_2\vert\mathbf{a}\vert_2.
\end{align*}
Since $\boldsymbol{\zeta}_{t_1 -u, t_1 - t_2 - 1}^\top\mathbf{B}_u$ is independent of $\mathbf{E}\left[\mathbf{z}_{t_2- u}\mid\mathcal{F}_{t_2 - u,j}\right] - \mathbf{E}\left[\mathbf{z}_{t_2- u}\mid\mathcal{F}_{t_2 - u, j-1}\right],$ from Examples 1.22 of \cite{MR2767184}, for any $\mathbf{a}\in\mathbf{R}^d,$
\begin{align*}
    &\mathbf{E}\left[\left\vert  \boldsymbol{\zeta}_{t_1 -u, t_1 - t_2 - 1}^\top\mathbf{B}_u\left(\mathbf{E}\left[\mathbf{z}_{t_2- u}\mid\mathcal{F}_{t_2 - u,j}\right] - \mathbf{E}\left[\mathbf{z}_{t_2- u}\mid\mathcal{F}_{t_2 - u, j-1}\right]\right)\right\vert^{M/2}\right.\\
    &\mid \mathbf{E}\left[\mathbf{z}_{t_2- u}\mid\mathcal{F}_{t_2 - u,j}\right] - \mathbf{E}\left[\mathbf{z}_{t_2- u}\mid\mathcal{F}_{t_2 - u, j-1}\right] = \mathbf{a}\Big]\\
    & = \mathbf{E}\left[\left\vert  \boldsymbol{\zeta}_{t_1 -u, t_1 - t_2 - 1}^\top\mathbf{B}_u\mathbf{a}\right\vert^{M/2}\right]\leq C\left\vert\mathbf{B}_u\right\vert_2^{M/2}\vert\mathbf{a}\vert_2^{M/2},
\end{align*}
which implies 
\begin{align*}
    &\left\Vert
    \boldsymbol{\zeta}_{t_1 -u, t_1 - t_2 - 1}^\top\mathbf{B}_u\left(\mathbf{E}\left[\mathbf{z}_{t_2- u}\mid\mathcal{F}_{t_2 - u,j}\right] - \mathbf{E}\left[\mathbf{z}_{t_2- u}\mid\mathcal{F}_{t_2 - u, j-1}\right]\right)
    \right\Vert_{M/2}\\
    &\leq C\vert \mathbf{B}_u\vert_2\left\Vert\ \left\vert \mathbf{E}\left[\mathbf{z}_{t_2- u}\mid\mathcal{F}_{t_2 - u,j}\right] - \mathbf{E}\left[\mathbf{z}_{t_2- u}\mid\mathcal{F}_{t_2 - u, j-1}\right]\right\vert_2\ \right\Vert_{M/2}\leq C_1\vert \mathbf{B}_u\vert_2\delta_j\sqrt{d},
\end{align*}
and 
\begin{align*}
        &\left\Vert
    \sum_{u = 0\vee(t_2 + S - t_1 + 1 - j)}^{h\wedge (t_2 - 1)}\left(\boldsymbol{\zeta}_{t_1 -u, t_1 - t_2 - 1}^\top\mathbf{B}_u\left(\mathbf{E}\left[\mathbf{z}_{t_2- u}\mid\mathcal{F}_{t_2 - u,j}\right] - \mathbf{E}\left[\mathbf{z}_{t_2- u}\mid\mathcal{F}_{t_2 - u, j-1}\right]\right)\right)
    \right\Vert_{M/2}\\
    &\leq C\delta_j\sqrt{d\sum_{u = 0\vee(t_2 + S - t_1 + 1 - j)}^{h\wedge (t_2 - 1)}\vert\mathbf{B}_u\vert^2_2}\leq C\delta_j\sqrt{d\sum_{u = 0}^{h}\vert\mathbf{B}_u\vert^2_2},
\end{align*}
which implies that 
\begin{equation}
    \begin{aligned}
        &\left\Vert
        \sum_{u = 0}^{h\wedge (t_2 - 1)}\left(\boldsymbol{\zeta}_{t_1 -u, t_1 - t_2 - 1}^\top\mathbf{B}_u\left(\mathbf{z}_{t_2- u} - \mathbf{E}\left[\mathbf{z}_{t_2- u}\mid\mathcal{F}_{t_2 - u, t_2 +S- u-t_1}\right]\right)\right)
    \right\Vert_{M/2}\\
    &\leq C\sum_{j = t_2 + S - h - t_1 + 1}^\infty \delta_j\sqrt{d\sum_{u = 0}^{h}\vert\mathbf{B}_u\vert^2_2}\\
    &\leq \frac{C_1}{(t_2 + S - h - t_1 + 1)^\alpha}\sqrt{d\sum_{u = 0}^{h}\vert\mathbf{B}_u\vert^2_2}\leq \frac{C_1\sqrt{dh}b^\dagger}{(S-h-\lambda_2)^\alpha}.
    \end{aligned}
    \label{eq.idea_first_part}
\end{equation}
On the other hand, since 
\begin{align*}
    \mathbf{E}\left[\boldsymbol{\omega }_{t_1 -u, t_1 - t_2 - 1}^\top\mathbf{B}_u\mathbf{z}_{t_2- u}\mid\mathcal{F}_{t_1,S}\right]
    = \mathbf{E}\left[\boldsymbol{\omega }_{t_1 -u, t_1 - t_2 - 1}^\top\mathbf{B}_u\mathbf{z}_{t_2- u}\mid\mathcal{F}_{t_1 - u,S - u}\right],
\end{align*}
we have 
\begin{align*}
    &\left\Vert
        \sum_{u = 0}^{h\wedge (t_2 - 1)}\left(\boldsymbol{\omega }_{t_1 -u, t_1 - t_2 - 1}^\top\mathbf{B}_u\mathbf{z}_{t_2- u} - \mathbf{E}\left[\boldsymbol{\omega }_{t_1 -u, t_1 - t_2 - 1}^\top\mathbf{B}_u\mathbf{z}_{t_2- u}\mid\mathcal{F}_{t_1,S}\right]\right)
        \right\Vert_{M/2}\\
    & = \left\Vert
        \sum_{u = 0}^{h\wedge (t_2 - 1)}\sum_{j = S-u + 1}^\infty\left(\mathbf{E}\left[\boldsymbol{\omega }_{t_1 -u, t_1 - t_2 - 1}^\top\mathbf{B}_u\mathbf{z}_{t_2- u}\mid\mathcal{F}_{t_1 - u, j}\right] - \mathbf{E}\left[\boldsymbol{\omega }_{t_1 -u, t_1 - t_2 - 1}^\top\mathbf{B}_u\mathbf{z}_{t_2- u}\mid\mathcal{F}_{t_1 - u, j - 1}\right]\right)
        \right\Vert_{M/2}\\
    &\leq \sum_{j = S- h + 1}^\infty\left\Vert
        \sum_{u = 0\vee (S - j + 1)}^{h\wedge (t_2 - 1)}\left(\mathbf{E}\left[\boldsymbol{\omega }_{t_1 -u, t_1 - t_2 - 1}^\top\mathbf{B}_u\mathbf{z}_{t_2- u}\mid\mathcal{F}_{t_1 - u, j}\right]\right.\right.\\
        &\left.- \mathbf{E}\left[\boldsymbol{\omega }_{t_1 -u, t_1 - t_2 - 1}^\top\mathbf{B}_u\mathbf{z}_{t_2- u}\mid\mathcal{F}_{t_1 - u, j - 1}\right]\right)
        \Bigg\Vert_{M/2}
\end{align*}
For any $j = S - h + 1,\cdots,$ and $z = 1,\cdots, h\wedge (t_2 - 1) -  0\vee (S - j + 1) + 1,$ define 
\begin{align*}
    K_{z,j} = \sum_{u = 0\vee (S - j + 1)}^{0\vee (S - j + 1) + z - 1}\left(\mathbf{E}\left[\boldsymbol{\omega }_{t_1 -u, t_1 - t_2 - 1}^\top\mathbf{B}_u\mathbf{z}_{t_2- u}\mid\mathcal{F}_{t_1 - u, j}\right] - \mathbf{E}\left[\boldsymbol{\omega }_{t_1 -u, t_1 - t_2 - 1}^\top\mathbf{B}_u\mathbf{z}_{t_2- u}\mid\mathcal{F}_{t_1 - u, j - 1}\right]\right)
\end{align*}
and $\mathcal{K}_{z,j}$ the $\sigma$-field generated by $e_T,\cdots, e_{t_1 - j - 0\vee (S - j + 1) - z + 1}.$ Then $K_{z,j}$ is measurable in  $\mathcal{K}_{z,j},$  $\mathcal{K}_{z,j}\subset \mathcal{K}_{z+1,j},$ and  
\begin{align*}
    &\mathbf{E}\left[K_{z + 1,j} - K_{z,j}\mid \mathcal{K}_{z,j}\right]\\
    & = \mathbf{E}\left[\left(\mathbf{E}\left[\boldsymbol{\omega }_{t_1 -0\vee (S - j + 1) - z, t_1 - t_2 - 1}^\top\mathbf{B}_{0\vee (S - j + 1) + z}\mathbf{z}_{t_2- 0\vee (S - j + 1) - z}\mid\mathcal{F}_{t_1 - 0\vee (S - j + 1) - z, j}\right]\right.\right.\\
    &\left.\left.- \mathbf{E}\left[\boldsymbol{\omega }_{t_1 - 0\vee (S - j + 1) - z, t_1 - t_2 - 1}^\top\mathbf{B}_{0\vee (S - j + 1) + z}\mathbf{z}_{t_2- 0\vee (S - j + 1) - z}\mid\mathcal{F}_{t_1 - 0\vee (S - j + 1) - z, j - 1}\right]\right)\mid \mathcal{K}_{z,j}\right]\\
    & = \mathbf{E}\left[\boldsymbol{\omega }_{t_1 -0\vee (S - j + 1) - z, t_1 - t_2 - 1}^\top\mathbf{B}_{0\vee (S - j + 1) + z}\mathbf{z}_{t_2- 0\vee (S - j + 1) - z}\mid\mathcal{F}_{t_1 - 0\vee (S - j + 1) - z, j - 1}\right]\\
    & - \mathbf{E}\left[\boldsymbol{\omega }_{t_1 -0\vee (S - j + 1) - z, t_1 - t_2 - 1}^\top\mathbf{B}_{0\vee (S - j + 1) + z}\mathbf{z}_{t_2- 0\vee (S - j + 1) - z}\mid\mathcal{F}_{t_1 - 0\vee (S - j + 1) - z, j - 1}\right] = 0,
\end{align*}
so $K_{z,j}$ forms a martingale. From Theorem 1.1 of \cite{MR0400380},
\begin{align*}
    &\left\Vert
        \sum_{u = 0\vee (S - j + 1)}^{h\wedge (t_2 - 1)}\left(\mathbf{E}\left[\boldsymbol{\omega }_{t_1 -u, t_1 - t_2 - 1}^\top\mathbf{B}_u\mathbf{z}_{t_2- u}\mid\mathcal{F}_{t_1 - u, j}\right] - \mathbf{E}\left[\boldsymbol{\omega }_{t_1 -u, t_1 - t_2 - 1}^\top\mathbf{B}_u\mathbf{z}_{t_2- u}\mid\mathcal{F}_{t_1 - u, j - 1}\right]\right)
        \right\Vert_{M/2}\\
    &\leq C\sqrt{\sum_{u = 0\vee (S - j + 1)}^{h\wedge (t_2 - 1)}\left\Vert \mathbf{E}\left[\boldsymbol{\omega }_{t_1 -u, t_1 - t_2 - 1}^\top\mathbf{B}_u\mathbf{z}_{t_2- u}\mid\mathcal{F}_{t_1 - u, j}\right] - \mathbf{E}\left[\boldsymbol{\omega }_{t_1 -u, t_1 - t_2 - 1}^\top\mathbf{B}_u\mathbf{z}_{t_2- u}\mid\mathcal{F}_{t_1 - u, j - 1}\right]\right\Vert^2_{M/2}}    
\end{align*}

Since $j\geq S -h + 1 > t_1 - t_2,$ we can establish the following  decomposition
\begin{align*}
    \boldsymbol{\omega }_{t_1 -u, t_1 - t_2 - 1} = \boldsymbol{\omega}_{t_1  - u, j - 1} + \boldsymbol{\zeta}_{t_1 - u, j - 1} - \boldsymbol{\zeta}_{t_1 - u, t_1 - t_2 - 1},
\end{align*}
and $\boldsymbol{\zeta}_{t_1 - u, j - 1} - \boldsymbol{\zeta}_{t_1 - u, t_1 - t_2 - 1}$ is measurable in $\mathcal{F}_{t_1 - u, j - 1}.$ In such case,
\begin{align*}
     &\mathbf{E}\left[\boldsymbol{\omega }_{t_1 -u, t_1 - t_2 - 1}^\top\mathbf{B}_u\mathbf{z}_{t_2- u}\mid\mathcal{F}_{t_1 - u, j}\right] - \mathbf{E}\left[\boldsymbol{\omega }_{t_1 -u, t_1 - t_2 - 1}^\top\mathbf{B}_u\mathbf{z}_{t_2- u}\mid\mathcal{F}_{t_1 - u, j - 1}\right]\\
     & = \mathbf{E}\left[\boldsymbol{\omega }_{t_1 -u, j - 1}^\top\mathbf{B}_u\mathbf{z}_{t_2- u}\mid\mathcal{F}_{t_1 - u, j}\right] - \mathbf{E}\left[\boldsymbol{\omega }_{t_1 -u, j - 1}^\top\mathbf{B}_u\mathbf{z}_{t_2- u}\mid\mathcal{F}_{t_1 - u, j - 1}\right]\\
     & + \left(\boldsymbol{\zeta}_{t_1 - u, j - 1} - \boldsymbol{\zeta}_{t_1 - u, t_1 - t_2 - 1}\right)^\top\mathbf{B}_u\left(\mathbf{E}\left[\mathbf{z}_{t_2 - u}\mid\mathcal{F}_{t_2 - u, t_2 + j - t_1}\right] - \mathbf{E}\left[\mathbf{z}_{t_2 - u}\mid\mathcal{F}_{t_2 - u, t_2 + j - t_1 - 1}\right]\right),
\end{align*}
and from \eqref{eq.zeta_transform},
\begin{align*}
    &\left\Vert
    \left(\boldsymbol{\zeta}_{t_1 - u, j - 1} - \boldsymbol{\zeta}_{t_1 - u, t_1 - t_2 - 1}\right)^\top\mathbf{B}_u\left(\mathbf{E}\left[\mathbf{z}_{t_2 - u}\mid\mathcal{F}_{t_2 - u, t_2 + j - t_1}\right] - \mathbf{E}\left[\mathbf{z}_{t_2 - u}\mid\mathcal{F}_{t_2 - u, t_2 + j - t_1 - 1}\right]\right)
    \right\Vert_{M/2}\\
    &\leq \vert\mathbf{B}_u\vert_2\left\Vert\ \left\vert \boldsymbol{\zeta}_{t_1 - u, j - 1} - \boldsymbol{\zeta}_{t_1 - u, t_1 - t_2 - 1}\right\vert_2 \right\Vert_{M}\left\Vert\ \left\vert (\mathbf{E}\left[\mathbf{z}_{t_2 - u}\mid\mathcal{F}_{t_2 - u, t_2 + j - t_1}\right] - \mathbf{E}\left[\mathbf{z}_{t_2 - u}\mid\mathcal{F}_{t_2 - u, t_2 + j - t_1 - 1}\right]\right\vert_2\ \right\Vert_{M}\\
    &\leq \frac{Cb^\dagger d}{(t_1 - t_2)^\alpha}\delta_{t_2 + j - t_1}.
\end{align*}
Therefore, we have 
\begin{align*}
    &\left\Vert
        \sum_{u = 0\vee (S - j + 1)}^{h\wedge (t_2 - 1)}\left(\mathbf{E}\left[\boldsymbol{\omega }_{t_1 -u, t_1 - t_2 - 1}^\top\mathbf{B}_u\mathbf{z}_{t_2- u}\mid\mathcal{F}_{t_1 - u, j}\right] - \mathbf{E}\left[\boldsymbol{\omega }_{t_1 -u, t_1 - t_2 - 1}^\top\mathbf{B}_u\mathbf{z}_{t_2- u}\mid\mathcal{F}_{t_1 - u, j - 1}\right]\right)
        \right\Vert_{M/2}\\
    &\leq C\sqrt{\sum_{u = 0\vee (S - j + 1)}^{h\wedge (t_2 - 1)}\left(\frac{Cb^\dagger d}{(t_1 - t_2)^\alpha}\delta_{t_2 + j - t_1}\right)^2}\leq \frac{Cb^\dagger d\sqrt{h}}{(t_1 - t_2)^\alpha}\delta_{t_2 + j - t_1},
\end{align*}
and 
\begin{equation}
\begin{aligned}
     &\left\Vert
        \sum_{u = 0}^{h\wedge (t_2 - 1)}\left(\boldsymbol{\omega }_{t_1 -u, t_1 - t_2 - 1}^\top\mathbf{B}_u\mathbf{z}_{t_2- u} - \mathbf{E}\left[\boldsymbol{\omega }_{t_1 -u, t_1 - t_2 - 1}^\top\mathbf{B}_u\mathbf{z}_{t_2- u}\mid\mathcal{F}_{t_1,S}\right]\right)
        \right\Vert_{M/2}\\
    &\leq C\sum_{j = S - h + 1}^\infty\frac{b^\dagger d\sqrt{h}}{(t_1 - t_2)^\alpha}\delta_{t_2 + j - t_1}\\
    &\leq \frac{C_1b^\dagger d\sqrt{h}}{(t_1 - t_2)^\alpha}\frac{1}{(S-h+1+t_2 - t_1)^\alpha}\leq\frac{C_1b^\dagger d\sqrt{h}}{(t_1 - t_2)^\alpha}\frac{1}{(S-h - \lambda_2)^\alpha}.
\end{aligned}
\label{eq.second_term_idea}
\end{equation}
From \eqref{eq.idea_first_part} and \eqref{eq.second_term_idea}, 
\begin{align*}
    &\left\Vert
        \sum_{u = 0}^{h\wedge (t_2 - 1)}\left(\mathbf{z}_{t_1- u}^\top\mathbf{B}_u\mathbf{z}_{t_2- u} - \mathbf{E}\left[\mathbf{z}_{t_1- u}^\top\mathbf{B}_u\mathbf{z}_{t_2- u}\mid\mathcal{F}_{t_1,S}\right]\right)
        \right\Vert_{M/2}\\
    &\leq \frac{C\sqrt{dh}b^\dagger}{(S-h-\lambda_2)^\alpha} +\frac{Cb^\dagger d\sqrt{h}}{(t_1 - t_2)^\alpha}\frac{1}{(S-h - \lambda_2)^\alpha}\\
    &\leq \frac{C\sqrt{dh}b^\dagger}{(S-h-\lambda_2)^\alpha} +\frac{Cb^\dagger d\sqrt{h}}{\lambda_1^\alpha}\frac{1}{(S-h - \lambda_2)^\alpha}\leq \frac{2C\sqrt{dh}b^\dagger}{(S-h-\lambda_2)^\alpha},
\end{align*}
which proves \eqref{eq.organize_lag_S}.
\end{proof}

\begin{corollary}
    Suppose conditions of Corollary \ref{corollary.combination_lag} hold true.  Then for any $S > 2\lambda_2 + h,$ we further derive
    \begin{align*}
        &\left\Vert\sum_{t_1 = 1 + \lambda_1}^T\sum_{t_2 = (t_1 - \lambda_2)\vee 1}^{t_1 - \lambda_1}\left(\mathbf{E}\left[\mathbf{z}_{t_1}^\top\mathbf{A}_{t_1t_2}\mathbf{z}_{t_2}\mid\mathcal{F}_{t_1, S}\right] - \mathbf{E}\left[\mathbf{z}_{t_1}^\top\mathbf{A}_{t_1t_2}\mathbf{z}_{t_2}\right]\right)\right\Vert_{M/2}\\
        &\leq \left\Vert\sum_{t_1 = 1 + \lambda_1}^T\sum_{t_2 = (t_1 - \lambda_2)\vee 1}^{t_1 - \lambda_1}\left(\mathbf{z}_{t_1}^\top\mathbf{A}_{t_1t_2}\mathbf{z}_{t_2} - \mathbf{E}\left[\mathbf{z}_{t_1}^\top\mathbf{A}_{t_1t_2}\mathbf{z}_{t_2}\right]\right)\right\Vert_{M/2}\\
        & + \left\Vert
        \sum_{t_1 = 1 + \lambda_1}^T\sum_{t_2 = (t_1 - \lambda_2)\vee 1}^{t_1 - \lambda_1}\left(\mathbf{z}_{t_1}^\top\mathbf{A}_{t_1t_2}\mathbf{z}_{t_2} - \mathbf{E}\left[\mathbf{z}_{t_1}^\top\mathbf{A}_{t_1t_2}\mathbf{z}_{t_2}\mid\mathcal{F}_{t_1,S}\right]\right)
        \right\Vert_{M/2}\\
        &\leq C\sqrt{d\sum_{t_1 = 1 + \lambda_1}^T\sum_{t_2 = (t_1 - \lambda_2)\vee 1}^{t_1 - \lambda_1}\left\vert\mathbf{A}_{t_1t_2}\right\vert^2_2} + \frac{CT\sqrt{T}a^\dagger}{(S-\lambda_2)^{\alpha - 2}}\\
        &\leq C_1\sqrt{d\sum_{t_1 = 1 + \lambda_1}^T\sum_{t_2 = (t_1 - \lambda_2)\vee 1}^{t_1 - \lambda_1}\left\vert\mathbf{A}_{t_1t_2}\right\vert^2_2},
    \end{align*}
    and 
    \begin{align*}
        &\left\Vert
        \sum_{u = 0}^{h\wedge (t_2 - 1)}\left(\mathbf{E}\left[\mathbf{z}_{t_1- u}^\top\mathbf{B}_u\mathbf{z}_{t_2- u}\mid\mathcal{F}_{t_1,S}\right] - \mathbf{E}\left[\mathbf{z}_{t_1- u}^\top\mathbf{B}_u\mathbf{z}_{t_2- u}\right]\right)
        \right\Vert_{M/2}\\
        &\leq \left\Vert
        \sum_{u = 0}^{h\wedge (t_2 - 1)}\left(\mathbf{z}_{t_1- u}^\top\mathbf{B}_u\mathbf{z}_{t_2- u} - \mathbf{E}\left[\mathbf{z}_{t_1- u}^\top\mathbf{B}_u\mathbf{z}_{t_2- u}\right]\right)
        \right\Vert_{M/2}\\
        & + \left\Vert
        \sum_{u = 0}^{h\wedge (t_2 - 1)}\left(\mathbf{z}_{t_1- u}^\top\mathbf{B}_u\mathbf{z}_{t_2- u} - \mathbf{E}\left[\mathbf{z}_{t_1- u}^\top\mathbf{B}_u\mathbf{z}_{t_2- u}\mid\mathcal{F}_{t_1,S}\right]\right)
        \right\Vert_{M/2}\\
        &\leq C\sqrt{d\sum_{u = 0}^{h\wedge (t_2 - 1)}\left\vert\mathbf{B}_u\right\vert^2_2} + \frac{C\sqrt{dh}b^\dagger}{(S-h-\lambda_2)^\alpha}\\
        &\leq C_1\sqrt{d\sum_{u = 0}^{h\wedge (t_2 - 1)}\left\vert\mathbf{B}_u\right\vert^2_2}.
    \end{align*}
    \label{corollary.moment_expectation}
\end{corollary}

\subsection{Technical details of distributional results under $H_0$}
\label{section.technical_details_under_H0}
Some technical assumptions are needed for the distributional results of the test statistic, which are summarized as follows:

\begin{assumption}
\label{assumption.independent_short_range_structure}
    Suppose independent random variables $e_t = (e_t^{(1)}, e_t^{(2)})\in\mathbf{H}^{(1)}\times \mathbf{H}^{(2)}$ for measurable spaces $\mathbf{H}^{(1)}, \mathbf{H}^{(2)}.$ Furthermore, assume that $e_{t}^{(1)}$ is independent of $e_{t}^{(2)}$ for any $t\in\mathbf{Z},$ and   
    \begin{align*}
    \mathbf{z}_t = \left(
    \begin{matrix}
        \mathbf{x}_t\\
        \mathbf{y}_t
    \end{matrix}
    \right) = 
    \left(
    \begin{matrix}
        g_{t,T}^{(1)}\left(\cdots, e_{t - 2}^{(1)}, e_{t - 1}^{(1)}, e_t^{(1)}\right)\\
        g_{t,T}^{(2)}\left(\cdots, e_{t - 2}^{(2)}, e_{t - 1}^{(2)}, e_t^{(2)}\right)
    \end{matrix}
    \right),
    \end{align*}
    where $g_{t,T}^{(1)}\left(\cdot\right):\mathbf{H}^{(1)\mathbf{Z}}\to\mathbf{R}^{d_1}$ and $g_{t,T}^{(2)}\left(\cdot\right):\mathbf{H}^{(2)\mathbf{Z}}\to\mathbf{R}^{d_2}$ are respectively measurable functions in $\mathbf{H}^{(1)\mathbf{Z}}$ and $\mathbf{H}^{(2)\mathbf{Z}}.$ Assume that $\mathbf{z}_t, t\in\mathbf{Z}$ satisfy Definition \ref{def.m_alpha_short_range_dependence} with $M, \alpha > 8,$ and $d_1,d_2\asymp T.$
\end{assumption}
Assumption \ref{assumption.independent_short_range_structure} further restricts the functional form of $g_{t,T}(\cdot)$ in \eqref{eq.def_z_t}. Specifically, it requires that $\mathbf{x}_t$, $t \in \mathbf{Z}$, depend only on ${e_t^{(1)}, t \in \mathbf{Z}}$, while $\mathbf{y}_t$, $t \in \mathbf{Z}$, depend only on ${e_t^{(2)}, t \in \mathbf{Z}}$. Since $e_t^{(1)}$ and $e_t^{(2)}$ are independent for all $t$, Assumption \ref{assumption.independent_short_range_structure} guarantees that the sequence $\{\mathbf{x}_t, t \in \mathbf{Z}\}$ is independent of the sequence $\{\mathbf{y}_t, t \in \mathbf{Z}\}$.

\begin{assumption}
\label{assumption.lambda_choices}
    Suppose the parameters satisfy $\lambda_1\asymp T^{\kappa_1},$  $\lambda_2 \asymp T^{\kappa_2},$ and $h\asymp T^{\kappa_h},$ where $ 1 >\kappa_2> \kappa_1 > \frac{5}{2(\alpha - 2)}$ and $0 < \kappa_h < 1.$
\end{assumption}

\begin{assumption}
\label{assumption.variance}
    Define the canonical statistic 
    \begin{align*}
        \widehat{R}^\dagger = \sum_{t = 1 + \lambda_1}^T O_t^\dagger,
    \quad O_t^\dagger &= \frac{1}{\mathcal{U}}\sum_{t_1 = (t - \lambda_2)\vee 1}^{t - \lambda_1}\sum_{s = 0}^{h\wedge (t_1 - 1)}\left(\mathbf{x}_{t}^\top\mathbf{x}_{t_1} - \mathbf{E}\left[\mathbf{x}_{t}^\top\mathbf{x}_{t_1}\right]\right)\left(\mathbf{y}_{t - s}^\top\mathbf{y}_{t_1 - s} - \mathbf{E}\left[\mathbf{y}_{t - s}^\top\mathbf{y}_{t_1 - s}\right]\right)\\
     & + \frac{1}{\mathcal{U}}\sum_{t_1 = (t - \lambda_2)\vee 1}^{t - \lambda_1}\sum_{s = 1}^{h\wedge(t_1 - 1)}\left(\mathbf{x}_{t - s}^\top\mathbf{x}_{t_1 - s} - \mathbf{E}\left[\mathbf{x}_{t - s}^\top\mathbf{x}_{t_1 - s}\right]\right)\left(\mathbf{y}_{t}^\top\mathbf{y}_{t_1} - \mathbf{E}\left[\mathbf{y}_{t}^\top\mathbf{y}_{t_1}\right]\right).
    \end{align*}
    Assume that there exists a constant $c>0$ such that $\mathrm{Var}\left(\widehat{R}^\dagger\right) > c$ for  sufficiently large $T.$
\end{assumption}

We call $\widehat{R}^\dagger$ ``canonical'' because it cannot be calculated in practice, as the expectations $\mathbf{E}\left[\mathbf{x}_{t}^\top\mathbf{x}_{t_1}\right]$ and $\mathbf{E}\left[\mathbf{y}_{t}^\top\mathbf{y}_{t_1}\right]$ are unknown and hard to estimate, especially in the  setting of nonstationary time series, whose autocovariances may change with the time index $t.$ Suppose Assumptions \ref{assumption.independent_short_range_structure} and \ref{assumption.lambda_choices} hold true. For each $t,t_1,$ define 
\begin{equation}
\gamma_{tt_1} = \sum_{s = 0}^{h\wedge (t_1 - 1)}\left(\mathbf{y}_{t - s}^\top\mathbf{y}_{t_1 - s} - \mathbf{E}\left[\mathbf{y}_{t - s}^\top\mathbf{y}_{t_1 - s}\right]\right),
\label{eq.def_gamma_tt1}
\end{equation}
and define the matrices $\mathbf{I}_{d_1}\in\mathbf{R}^{d\times d},$ where $\mathbf{I}_{d_1}^{(i,i)} = 1$ for $i = 1,\cdots, d_1,$ and $\mathbf{I}_{d_1}^{(i,j)}  = 0$ otherwise; and $\mathbf{I}_{d_2}\in\mathbf{R}^{d\times d},$ where $\mathbf{I}_{d_1}^{(i,i)} = 1$ for $i = d_1 + 1,\cdots, d,$ and  $\mathbf{I}_{d_1}^{(i,j)}  = 0$ otherwise. With these notations, we have $\mathbf{x}_t^\top\mathbf{x}_{t_1} = \mathbf{z}^\top_t\mathbf{I}_{d_1}\mathbf{z}_t$ and $\mathbf{y}_t^\top\mathbf{y}_{t_1} = \mathbf{z}^\top_t\mathbf{I}_{d_2}\mathbf{z}_t.$ Furthermore, from Corollary \ref{corollary.combination_lag},
\begin{equation}
\begin{aligned}
    \left\Vert\gamma_{tt_1}\right\Vert_{M/2} &= \left\Vert
    \sum_{s = 0}^{h\wedge (t_1 - 1)}\left(\mathbf{z}_{t - s}^\top\mathbf{I}_{d_2}\mathbf{z}_{t_1 - s} - \mathbf{E}\left[\mathbf{z}_{t - s}^\top\mathbf{I}_{d_2}\mathbf{z}_{t_1 - s}\right]\right)
    \right\Vert_{M/2}\\
    &\leq C\sqrt{d\sum_{s = 0}^{h\wedge (t_1 - 1)}\left\vert \mathbf{I}_{d_2}\right\vert^2_2}\leq C\sqrt{dh}
\end{aligned}
\label{eq.gamma_norm}
\end{equation}

Since the random variables $\left\{\gamma_{tt_1}: t = 1+\lambda_1,\cdots,T, t_1 = (t-\lambda_2)\vee 1,\cdots, t-\lambda_1\right\}$ are independent of $\mathbf{x}_t, t\in\mathbf{Z},$ from Examples 1.22 of \cite{MR2767184} and Theorem \ref{theorem.linear_and_quadratic}, for any real numbers $c_{tt_1},$ 
\begin{align*}
    &\mathbf{E}\left[\left\vert \sum_{t = 1 + \lambda_1}^T\sum_{t_1 = (t - \lambda_2)\vee 1}^{t - \lambda_1}\gamma_{tt_1}\left(\mathbf{x}_{t}^\top\mathbf{x}_{t_1} - \mathbf{E}\left[\mathbf{x}_{t}^\top\mathbf{x}_{t_1}\right]\right)\right\vert^{M/2}\right.\\
    & \mid\gamma_{tt_1} = c_{tt_1}, t = 1+\lambda_1,\cdots,T, t_1 = (t-\lambda_2)\vee 1,\cdots, t-\lambda_1\Bigg]\\
    &=  \mathbf{E}\left[\left\vert \sum_{t = 1 + \lambda_1}^T\sum_{t_1 = (t - \lambda_2)\vee 1}^{t - \lambda_1}c_{tt_1}\left(\mathbf{z}_{t}^\top\mathbf{I}_{d_1}\mathbf{z}_{t_1} - \mathbf{E}\left[\mathbf{z}_{t}^\top\mathbf{I}_{d_1}\mathbf{z}_{t_1}\right]\right)\right\vert^{M/2}\right]
    \leq C\left(\sqrt{d\sum_{t = 1 + \lambda_1}^T\sum_{t_1 = (t - \lambda_2)\vee 1}^{t - \lambda_1}c^2_{tt_1}}\right)^{M/2}.
\end{align*}
Therefore, from Corollary \ref{corollary.combination_lag} and \eqref{eq.gamma_norm},
\begin{equation}
\begin{aligned}
    &\left\Vert
    \sum_{t = 1 + \lambda_1}^T\sum_{t_1 = (t - \lambda_2)\vee 1}^{t - \lambda_1}\gamma_{tt_1}\left(\mathbf{x}_{t}^\top\mathbf{x}_{t_1} - \mathbf{E}\left[\mathbf{x}_{t}^\top\mathbf{x}_{t_1}\right]\right)
    \right\Vert_{M/2}\\
    &\leq C\sqrt{d}\left\Vert\sqrt{\sum_{t = 1 + \lambda_1}^T\sum_{t_1 = (t - \lambda_2)\vee 1}^{t - \lambda_1}\gamma^2_{tt_1}}\right\Vert_{M/2}\\
    &\leq C\sqrt{d\sum_{t = 1 + \lambda_1}^T\sum_{t_1 = (t - \lambda_2)\vee 1}^{t - \lambda_1}\left\Vert\gamma_{tt_1}\right\Vert^2_{M/2}}\leq C_1d\sqrt{T(\lambda_2 - \lambda_1)h}.
\end{aligned}
\label{eq.moment_bound}
\end{equation}
Similar to \eqref{eq.def_gamma_tt1} to \eqref{eq.moment_bound}, we have 
\begin{align*}
    &\left\Vert\sum_{t_1 = (t - \lambda_2)\vee 1}^{t - \lambda_1}\sum_{s = 1}^{h\wedge(t_1 - 1)}\left(\mathbf{x}_{t - s}^\top\mathbf{x}_{t_1 - s} - \mathbf{E}\left[\mathbf{x}_{t - s}^\top\mathbf{x}_{t_1 - s}\right]\right)\left(\mathbf{y}_{t}^\top\mathbf{y}_{t_1} - \mathbf{E}\left[\mathbf{y}_{t}^\top\mathbf{y}_{t_1}\right]\right)\right\Vert_{M/2}\\
    &\leq Cd\sqrt{T(\lambda_2 - \lambda_1)h},
\end{align*}
and 
\begin{equation}
\begin{aligned}
    \mathrm{Var}\left(\widehat{R}^\dagger\right)\leq \left\Vert \widehat{R}^\dagger\right\Vert_{M/2}^2\leq C.
\end{aligned}
\label{eq.R_scale}
\end{equation}
Therefore, Assumption \ref{assumption.variance} actually ensures that $\widehat{R}^\dagger$ has a constant order. Furthermore, as demonstrated in the proof of Theorem \ref{theorem.distributional_under_H0}, the difference between $\widehat{R}$ and $\widehat{R}^\dagger$ is negligible asymptotically, so under $H_0,$  $\widehat{R}$ should also have a constant order.

As in Section \ref{section.proof_linear_quadratic}, we denote $\mathcal{F}_{t,S}$ as the $\sigma$-field generated by $e_{t-S}, e_{t-S+1}, \ldots, e_t$, where $e_t = (e_t^{(1)}, e_t^{(2)})$.

\begin{proof}[Proof of Theorem \ref{theorem.distributional_under_H0}]
 Define the functions 
    \begin{align*}
        g_0(u)  = \left(1 - \min\left(1, \max(u,0)\right)^4\right)^4\quad\text{and}\quad g_{\psi,t}(u) = g_0\left(\psi(u - t)\right).
    \end{align*}
    According to (S2) - (S5) of \cite{MR3992401},  we have 
    \begin{align*}
        &g_* = \sup_{u\in\mathbf{R}}\left(
        \left\vert g^\prime_0(u)\right\vert + \left\vert g^{\prime\prime}_0(u)\right\vert + \left\vert g^{\prime\prime\prime}_0(u)\right\vert
        \right) < \infty,\quad \mathbf{1}_{u\leq t}\leq g_{\psi,t}(u)\leq\mathbf{1}_{u\leq t + \psi^{-1}},\\
        &\text{and}\quad \sup_{x,t\in\mathbf{R}}\left\vert g_{\psi,t}^\prime(u)\right\vert\leq g_*\psi,\quad \sup_{x,t\in\mathbf{R}}\left\vert g_{\psi,t}^{\prime\prime}(u)\right\vert\leq g_*\psi^2,\quad \sup_{x,t\in\mathbf{R}}\left\vert g_{\psi,t}^{\prime\prime\prime}(u)\right\vert\leq g_*\psi^3.
    \end{align*}
    Since $\mathrm{Var}(\epsilon) > c > 0,$ for any given $\psi  > 0,$
    \begin{align*}
        &\mathbf{pr}\left(\widehat{R}\leq x\right) - \mathbf{pr}\left(\epsilon\leq x\right)\\
        &= \mathbf{pr}\left(\widehat{R}\leq x\right) - \mathbf{pr}\left(\epsilon\leq x + \frac{1}{\psi}\right)  + \mathbf{pr}\left(\epsilon\leq x + \frac{1}{\psi}\right) - \mathbf{pr}\left(\epsilon\leq x\right)\\
        &\leq \mathbf{E}\left[g_{\psi,x}\left(\widehat{R}\right)\right] - \mathbf{E}\left[g_{\psi,x}\left(\epsilon\right)\right] +\int_{t\in[x,x + 1/\psi]}\frac{1}{\sqrt{2\pi\mathrm{Var}(\epsilon)}}\exp\left(-\frac{t^2}{2\mathrm{Var}(\epsilon)}\right)\mathrm{d}t\\
        &\leq \sup_{x\in\mathbf{R}}\left\vert\mathbf{E}\left[g_{\psi,x}\left(\widehat{R}\right)\right] - \mathbf{E}\left[g_{\psi,x}\left(\epsilon\right)\right]\right\vert + \frac{C}{\psi},
    \end{align*}
    and 
    \begin{align*}
        &\mathbf{pr}\left(\widehat{R}\leq x\right) - \mathbf{pr}\left(\epsilon\leq x\right)\\
        &= \mathbf{pr}\left(\widehat{R}\leq x\right) - \mathbf{pr}\left(\epsilon\leq x - \frac{1}{\psi}\right)  -\left(\mathbf{pr}\left(\epsilon\leq x\right) - \mathbf{pr}\left(\epsilon\leq x - \frac{1}{\psi}\right)\right)\\
        &\geq \mathbf{E}\left[g_{\psi,x - 1/\psi}\left(\widehat{R}\right)\right] -  \mathbf{E}\left[g_{\psi,x - 1/\psi}\left(\epsilon\right)\right] - \int_{t\in[x - 1/\psi, x]}\frac{1}{\sqrt{2\pi\mathrm{Var}(\epsilon)}}\exp\left(-\frac{t^2}{2\mathrm{Var}(\epsilon)}\right)\mathrm{d}t\\
        &\geq -\sup_{x\in\mathbf{R}}\left\vert
        \mathbf{E}\left[g_{\psi,x}\left(\widehat{R}\right)\right] -  \mathbf{E}\left[g_{\psi,x}\left(\epsilon\right)\right]
        \right\vert - \frac{C}{\psi},
    \end{align*}
    so we have 
    \begin{equation}
        \begin{aligned}
            \sup_{x\in\mathbf{R}}\left\vert \mathbf{pr}\left(\widehat{R}\leq x\right) - \mathbf{pr}\left(\epsilon\leq x\right)\right\vert
            \leq \sup_{x\in\mathbf{R}}\left\vert\mathbf{E}\left[g_{\psi,x}\left(\widehat{R}\right)\right] - \mathbf{E}\left[g_{\psi,x}\left(\epsilon\right)\right]\right\vert + \frac{C}{\psi}.
        \end{aligned}
        \label{eq.Prob_R_to_epsilon}
    \end{equation}  
    Since
    \begin{align*}
        &\frac{1}{\mathcal{U}}\left\Vert\sum_{t = 1 +\lambda_1}^T\sum_{t_1 = (t - \lambda_2)\vee 1}^{t - \lambda_1}\sum_{s = 0}^{h\wedge (t_1 - 1)}\mathbf{x}_{t}^\top\mathbf{x}_{t_1}\mathbf{y}_{t - s}^\top\mathbf{y}_{t_1 - s}\right.\\
        &\left.- \sum_{t = 1 +\lambda_1}^T\sum_{t_1 = (t - \lambda_2)\vee 1}^{t - \lambda_1}\sum_{s = 0}^{h\wedge (t_1 - 1)}\left(\mathbf{x}_{t}^\top\mathbf{x}_{t_1} - \mathbf{E}\left[\mathbf{x}_{t}^\top\mathbf{x}_{t_1}\right]\right)\left(\mathbf{y}_{t - s}^\top\mathbf{y}_{t_1 - s} - \mathbf{E}\left[\mathbf{y}_{t - s}^\top\mathbf{y}_{t_1 - s}\right]\right)\right\Vert_{M/2}\\
        &\leq I + II + III,\\
    \end{align*}
    where 
    \begin{align*}
        & I =  \frac{1}{\mathcal{U}}\left\Vert
        \sum_{t = 1 +\lambda_1}^T\sum_{t_1 = (t - \lambda_2)\vee 1}^{t - \lambda_1}\sum_{s = 0}^{h\wedge (t_1 - 1)}\mathbf{E}\left[\mathbf{x}_{t}^\top\mathbf{x}_{t_1}\right]\left(\mathbf{y}_{t - s}^\top\mathbf{y}_{t_1 - s} - \mathbf{E}\left[\mathbf{y}_{t - s}^\top\mathbf{y}_{t_1 - s}\right]\right)
        \right\Vert_{M/2},\\
        & II  = \frac{1}{\mathcal{U}}\left\Vert
        \sum_{t = 1 +\lambda_1}^T\sum_{t_1 = (t - \lambda_2)\vee 1}^{t - \lambda_1}\sum_{s = 0}^{h\wedge (t_1 - 1)}
        \mathbf{E}\left[\mathbf{y}_{t - s}^\top\mathbf{y}_{t_1 - s}\right]
        \left(\mathbf{x}_{t}^\top\mathbf{x}_{t_1} - \mathbf{E}\left[\mathbf{x}_{t}^\top\mathbf{x}_{t_1}\right]\right)
        \right\Vert_{M/2},\\
        & III = \frac{1}{\mathcal{U}}\sum_{t = 1 +\lambda_1}^T\sum_{t_1 = (t - \lambda_2)\vee 1}^{t - \lambda_1}\sum_{s = 0}^{h\wedge (t_1 - 1)}\left\vert \mathbf{E}\left[\mathbf{x}_{t}^\top\mathbf{x}_{t_1}\right]\right\vert\left\vert \mathbf{E}\left[\mathbf{y}_{t - s}^\top\mathbf{y}_{t_1 - s}\right]\right\vert.
    \end{align*}
    From \eqref{eq.covariances_x_y} and \eqref{eq.gamma_norm},
    \begin{align*}
        I 
        &\leq \frac{1}{\mathcal{U}}\sum_{t = 1 +\lambda_1}^T\sum_{t_1 = (t - \lambda_2)\vee 1}^{t - \lambda_1}
        \left\vert \mathbf{E}\left[\mathbf{x}_{t}^\top\mathbf{x}_{t_1}\right]\right\vert
        \left\Vert \sum_{s = 0}^{h\wedge (t_1 - 1)}\left(\mathbf{y}_{t - s}^\top\mathbf{y}_{t_1 - s} - \mathbf{E}\left[\mathbf{y}_{t - s}^\top\mathbf{y}_{t_1 - s}\right]\right)\right\Vert_{M/2}\\
        &\leq \frac{C}{\sqrt{dT(\lambda_2 - \lambda_1)}}\sum_{t = 1 +\lambda_1}^T\sum_{t_1 = (t - \lambda_2)\vee 1}^{t - \lambda_1}\frac{d_1}{(t - t_1)^\alpha}\leq \frac{C_1T}{\sqrt{\lambda_2 - \lambda_1}\lambda_1^{\alpha - 1}}\leq \frac{C_2}{T^{3/2}}.
    \end{align*}
    From Theorem \ref{theorem.linear_and_quadratic}, 
    \begin{align*}
       II  &= \frac{1}{\mathcal{U}}\left\Vert
        \sum_{t = 1 +\lambda_1}^T\sum_{t_1 = (t - \lambda_2)\vee 1}^{t - \lambda_1}\left(\sum_{s = 0}^{h\wedge (t_1 - 1)}
        \mathbf{E}\left[\mathbf{y}_{t - s}^\top\mathbf{y}_{t_1 - s}\right]\right)
        \left(\mathbf{z}_{t}^\top\mathbf{I}_{d_1}\mathbf{z}_{t_1} - \mathbf{E}\left[\mathbf{z}_{t}^\top\mathbf{I}_{d_1}\mathbf{z}_{t_1}\right]\right)
        \right\Vert_{M/2}\\
        &\leq \frac{C}{\mathcal{U}}\sqrt{d\sum_{t = 1 +\lambda_1}^T\sum_{t_1 = (t - \lambda_2)\vee 1}^{t - \lambda_1}\left(\sum_{s = 0}^{h\wedge (t_1 - 1)}
        \mathbf{E}\left[\mathbf{y}_{t - s}^\top\mathbf{y}_{t_1 - s}\right]\right)^2}.
    \end{align*}
    From \eqref{eq.covariances_x_y},
    \begin{align*}
        \left\vert\sum_{s = 0}^{h\wedge (t_1 - 1)}
        \mathbf{E}\left[\mathbf{y}_{t - s}^\top\mathbf{y}_{t_1 - s}\right]\right\vert\leq \sum_{s = 0}^{h\wedge (t_1 - 1)}\frac{Cd_2}{(t  -t_1)^\alpha}\leq \frac{C_1dh}{(t  -t_1)^\alpha},
    \end{align*}
    so 
    \begin{align*}
        II\leq \frac{Cdh}{\mathcal{U}}\sqrt{d\sum_{t = 1 +\lambda_1}^T\sum_{t_1 = (t - \lambda_2)\vee 1}^{t - \lambda_1}\frac{1}{(t  -t_1)^{2\alpha}}}\leq \frac{C_1\sqrt{Th}}{\lambda_1^{\alpha - 1/2}\sqrt{\lambda_2 - \lambda_1}}\leq \frac{C_2}{T^{3/2}}.
    \end{align*}
    From \eqref{eq.covariances_x_y}, 
    \begin{align*}
        III  &\leq  \frac{1}{\mathcal{U}}\sum_{t = 1 +\lambda_1}^T\sum_{t_1 = (t - \lambda_2)\vee 1}^{t - \lambda_1}\sum_{s = 0}^{h\wedge (t_1 - 1)} \frac{Cd^2}{(t - t_1)^{2\alpha}}\\
        &\leq \frac{C_1d^2h}{\mathcal{U}}\sum_{t = 1 +\lambda_1}^T\sum_{t_1 = (t - \lambda_2)\vee 1}^{t - \lambda_1}\frac{1}{(t - t_1)^{2\alpha}}\leq \frac{C_2d^2hT}{\mathcal{U}\lambda_1^{2\alpha - 1}}\leq \frac{C_3}{T^3}.
    \end{align*}
    Therefore, we have 
    \begin{align*}
        &\frac{1}{\mathcal{U}}\left\Vert\sum_{t = 1 +\lambda_1}^T\sum_{t_1 = (t - \lambda_2)\vee 1}^{t - \lambda_1}\sum_{s = 0}^{h\wedge (t_1 - 1)}\mathbf{x}_{t}^\top\mathbf{x}_{t_1}\mathbf{y}_{t - s}^\top\mathbf{y}_{t_1 - s}\right.\\
        &\left.- \sum_{t = 1 +\lambda_1}^T\sum_{t_1 = (t - \lambda_2)\vee 1}^{t - \lambda_1}\sum_{s = 0}^{h\wedge (t_1 - 1)}\left(\mathbf{x}_{t}^\top\mathbf{x}_{t_1} - \mathbf{E}\left[\mathbf{x}_{t}^\top\mathbf{x}_{t_1}\right]\right)\left(\mathbf{y}_{t - s}^\top\mathbf{y}_{t_1 - s} - \mathbf{E}\left[\mathbf{y}_{t - s}^\top\mathbf{y}_{t_1 - s}\right]\right)\right\Vert_{M/2}\leq \frac{C}{T^{3/2}}.
    \end{align*}
    Similarly,
    \begin{align*}
        &\frac{1}{\mathcal{U}}\left\Vert
        \sum_{t = 1+\lambda_1}^T\sum_{t_1 = (t - \lambda_2)\vee 1}^{t - \lambda_1}\sum_{s = 1}^{h\wedge(t_1 - 1)}\mathbf{x}_{t - s}^\top\mathbf{x}_{t_1 - s}\mathbf{y}_{t}^\top\mathbf{y}_{t_1}\right.\\
        &-\left.
        \sum_{t = 1+\lambda_1}^T\sum_{t_1 = (t - \lambda_2)\vee 1}^{t - \lambda_1}\sum_{s = 1}^{h\wedge(t_1 - 1)}\left(\mathbf{x}_{t - s}^\top\mathbf{x}_{t_1 - s} - \mathbf{E}\left[\mathbf{x}_{t - s}^\top\mathbf{x}_{t_1 - s}\right]\right)\left(\mathbf{y}_{t}^\top\mathbf{y}_{t_1} - \mathbf{E}\left[\mathbf{y}_{t}^\top\mathbf{y}_{t_1}\right]\right)
        \right\Vert_{M/2}\leq \frac{C}{T^{3/2}},
    \end{align*}
    and 
    \begin{align*}
        \left\Vert
         \widehat{R}^\dagger - \widehat{R}
        \right\Vert_{M/2}\leq \frac{C}{T^{3/2}}.
    \end{align*}
    For any given $\psi > 0,$ we have 
    \begin{equation}
    \sup_{x\in\mathbf{R}}\left\vert\mathbf{E}\left[g_{\psi,x}\left(\widehat{R}\right)\right] - \mathbf{E}\left[g_{\psi,x}\left(\widehat{R}^\dagger\right)\right]\right\vert\leq C\psi\left\Vert \widehat{R} - \widehat{R}^\dagger\right\Vert_{M/2}\leq \frac{C_1\psi}{T^{3/2}}.
    \label{eq.R_to_R_dagger}
    \end{equation}
    Choose $S > 2\lambda_2 + h,$ and define 
    \begin{equation}
    \begin{aligned}
    \widehat{R}^\dagger_S &= \sum_{t = 1 + \lambda_1}^T O^\dagger_{t, S},\\
     O^\dagger_{t, S} &= \frac{1}{\mathcal{U}}\sum_{t_1 = (t - \lambda_2)\vee 1}^{t - \lambda_1}\sum_{s = 0}^{h\wedge (t_1 - 1)}\left(\mathbf{E}\left[\mathbf{x}_{t}^\top\mathbf{x}_{t_1}\mid\mathcal{F}_{t,S}\right] - \mathbf{E}\left[\mathbf{x}_{t}^\top\mathbf{x}_{t_1}\right]\right)\left(\mathbf{E}\left[\mathbf{y}_{t - s}^\top\mathbf{y}_{t_1 - s}\mid\mathcal{F}_{t,S}\right] - \mathbf{E}\left[\mathbf{y}_{t - s}^\top\mathbf{y}_{t_1 - s}\right]\right)\\
     & + \frac{1}{\mathcal{U}}\sum_{t_1 = (t - \lambda_2)\vee 1}^{t - \lambda_1}\sum_{s = 1}^{h\wedge(t_1 - 1)}\left(\mathbf{E}\left[\mathbf{x}_{t - s}^\top\mathbf{x}_{t_1 - s}\mid\mathcal{F}_{t,S}\right] - \mathbf{E}\left[\mathbf{x}_{t - s}^\top\mathbf{x}_{t_1 - s}\right]\right)\left(\mathbf{E}\left[\mathbf{y}_{t}^\top\mathbf{y}_{t_1}\mid\mathcal{F}_{t,S}\right] - \mathbf{E}\left[\mathbf{y}_{t}^\top\mathbf{y}_{t_1}\right]\right).
    \end{aligned}
    \label{eq.def_O_t_S}
    \end{equation}
    Define
    \begin{equation}
        \gamma_{tt_1,S} = \sum_{s = 0}^{h\wedge (t_1 - 1)}\left(\mathbf{E}\left[\mathbf{y}_{t - s}^\top\mathbf{y}_{t_1 - s}\mid\mathcal{F}_{t,S}\right] - \mathbf{E}\left[\mathbf{y}_{t - s}^\top\mathbf{y}_{t_1 - s}\right]\right),
        \label{eq.def_gamma_tt1S}
    \end{equation}
    then from Corollary \ref{corollary.combination_lag}, 
    \begin{align*}
        \left\Vert
        \gamma_{tt_1} -  \gamma_{tt_1,S} 
        \right\Vert_{M/2} =  \left\Vert
        \sum_{s = 0}^{h\wedge (t_1 - 1)}\left(\mathbf{z}_{t - s}^\top\mathbf{I}_{d_2}\mathbf{z}_{t_1 - s}- \mathbf{E}\left[\mathbf{z}_{t - s}^\top\mathbf{I}_{d_2}\mathbf{z}_{t_1 - s}\mid\mathcal{F}_{t,S}\right]\right)
        \right\Vert_{M/2}\leq \frac{C\sqrt{dh}}{(S-h-\lambda_2)^\alpha},
    \end{align*}
    and from \eqref{eq.gamma_norm},
    \begin{equation}
    \begin{aligned}
        \left\Vert
       \gamma_{tt_1,S} 
        \right\Vert_{M/2}\leq \left\Vert
        \gamma_{tt_1} 
        \right\Vert_{M/2} + \left\Vert
        \gamma_{tt_1} -  \gamma_{tt_1,S} 
        \right\Vert_{M/2}\leq C\sqrt{dh}.
    \end{aligned}
    \label{eq.gamma_S_norm}
    \end{equation}
    We also have 
    \begin{align*}
    &\frac{1}{\mathcal{U}}\sum_{t = 1 + \lambda_1}^T\sum_{t_1 = (t - \lambda_2)\vee 1}^{t - \lambda_1}\sum_{s = 0}^{h\wedge (t_1 - 1)}\left(\mathbf{E}\left[\mathbf{x}_{t}^\top\mathbf{x}_{t_1}\mid\mathcal{F}_{t,S}\right] - \mathbf{E}\left[\mathbf{x}_{t}^\top\mathbf{x}_{t_1}\right]\right)\left(\mathbf{E}\left[\mathbf{y}_{t - s}^\top\mathbf{y}_{t_1 - s}\mid\mathcal{F}_{t,S}\right] - \mathbf{E}\left[\mathbf{y}_{t - s}^\top\mathbf{y}_{t_1 - s}\right]\right)\\
    & = \frac{1}{\mathcal{U}}\sum_{t = 1 + \lambda_1}^T\sum_{t_1 = (t - \lambda_2)\vee 1}^{t - \lambda_1}\gamma_{tt_1,S}\left(\mathbf{E}\left[\mathbf{x}_{t}^\top\mathbf{x}_{t_1}\mid\mathcal{F}_{t,S}\right] - \mathbf{E}\left[\mathbf{x}_{t}^\top\mathbf{x}_{t_1}\right]\right),
    \end{align*}
    and 
    \begin{align*}
        &\frac{1}{\mathcal{U}}\left\Vert
        \sum_{t = 1 + \lambda_1}^T\sum_{t_1 = (t - \lambda_2)\vee 1}^{t - \lambda_1}\sum_{s = 0}^{h\wedge (t_1 - 1)}\left(\mathbf{x}_{t}^\top\mathbf{x}_{t_1} - \mathbf{E}\left[\mathbf{x}_{t}^\top\mathbf{x}_{t_1}\right]\right)\left(\mathbf{y}_{t - s}^\top\mathbf{y}_{t_1 - s} - \mathbf{E}\left[\mathbf{y}_{t - s}^\top\mathbf{y}_{t_1 - s}\right]\right)\right.\\
        &\left. -\sum_{t = 1 + \lambda_1}^T\sum_{t_1 = (t - \lambda_2)\vee 1}^{t - \lambda_1}\sum_{s = 0}^{h\wedge (t_1 - 1)}\left(\mathbf{E}\left[\mathbf{x}_{t}^\top\mathbf{x}_{t_1}\mid\mathcal{F}_{t,S}\right] - \mathbf{E}\left[\mathbf{x}_{t}^\top\mathbf{x}_{t_1}\right]\right)\left(\mathbf{E}\left[\mathbf{y}_{t - s}^\top\mathbf{y}_{t_1 - s}\mid\mathcal{F}_{t,S}\right] - \mathbf{E}\left[\mathbf{y}_{t - s}^\top\mathbf{y}_{t_1 - s}\right]\right)
        \right\Vert_{M/2}\\
        & =  \frac{1}{\mathcal{U}}\left\Vert
        \sum_{t = 1 + \lambda_1}^T\sum_{t_1 = (t - \lambda_2)\vee 1}^{t - \lambda_1}\gamma_{tt_1}\left(\mathbf{x}_{t}^\top\mathbf{x}_{t_1} - \mathbf{E}\left[\mathbf{x}_{t}^\top\mathbf{x}_{t_1}\right]\right)\right.\\
        &\left.-
        \sum_{t = 1 + \lambda_1}^T\sum_{t_1 = (t - \lambda_2)\vee 1}^{t - \lambda_1}\gamma_{tt_1,S}\left(\mathbf{E}\left[\mathbf{x}_{t}^\top\mathbf{x}_{t_1}\mid\mathcal{F}_{t,S}\right] - \mathbf{E}\left[\mathbf{x}_{t}^\top\mathbf{x}_{t_1}\right]\right)
        \right\Vert_{M/2}\\
        &\leq \frac{1}{\mathcal{U}}\left\Vert
        \sum_{t = 1 + \lambda_1}^T\sum_{t_1 = (t - \lambda_2)\vee 1}^{t - \lambda_1}\left(\gamma_{tt_1} - \gamma_{tt_1,S}\right)\left(\mathbf{x}_{t}^\top\mathbf{x}_{t_1} - \mathbf{E}\left[\mathbf{x}_{t}^\top\mathbf{x}_{t_1}\right]\right)
        \right\Vert_{M/2}\\
        & + \frac{1}{\mathcal{U}}\left\Vert
        \sum_{t = 1 + \lambda_1}^T\sum_{t_1 = (t - \lambda_2)\vee 1}^{t - \lambda_1}\gamma_{tt_1,S}\left(\mathbf{x}_{t}^\top\mathbf{x}_{t_1} - \mathbf{E}\left[\mathbf{x}_{t}^\top\mathbf{x}_{t_1}\mid\mathcal{F}_{t,S}\right]\right)
        \right\Vert_{M/2}
    \end{align*}
    From Assumption \ref{assumption.independent_short_range_structure}, $\gamma_{tt_1} - \gamma_{tt_1,S}$ is a function of $\cdots, e_{t - 2}^{(2)}, e_{t - 1}^{(2)}, e_t^{(2)},$ which makes it independent of $\left\{\mathbf{x}_t, t\in\mathbf{Z}\right\}$ under $H_0.$ Therefore, 
    for any real numbers $c_{tt_1}, t = 1+\lambda_1,\cdots, T, t_1 = (t-\lambda_2)\vee 1,\cdots, t-\lambda_1,$  from Examples 1.22 of \cite{MR2767184}, 
    \begin{align*}
        &\mathbf{E}\left[
        \left\vert \sum_{t = 1 + \lambda_1}^T\sum_{t_1 = (t - \lambda_2)\vee 1}^{t - \lambda_1}\left(\gamma_{tt_1} - \gamma_{tt_1,S}\right)\left(\mathbf{x}_{t}^\top\mathbf{x}_{t_1} - \mathbf{E}\left[\mathbf{x}_{t}^\top\mathbf{x}_{t_1}\right]\right)\right\vert^{M/2}\mid \gamma_{tt_1} - \gamma_{tt_1,S}  = c_{tt_1},\right.\\
        &t = 1+\lambda_1,\cdots, T, t_1 = (t-\lambda_2)\vee 1,\cdots, t-\lambda_1
        \Bigg]\\
        & = \mathbf{E}\left[
        \left\vert \sum_{t = 1 + \lambda_1}^T\sum_{t_1 = (t - \lambda_2)\vee 1}^{t - \lambda_1}c_{tt_1}\left(\mathbf{z}_{t}^\top\mathbf{I}_{d_1}\mathbf{z}_{t_1} - \mathbf{E}\left[\mathbf{z}_{t}^\top\mathbf{I}_{d_1}\mathbf{z}_{t_1}\right]\right)\right\vert^{M/2}
        \right]\leq C\left(\sqrt{d\sum_{t = 1 + \lambda_1}^T\sum_{t_1 = (t - \lambda_2)\vee 1}^{t - \lambda_1}c_{tt_1}^2}\right)^{M/2},
    \end{align*}
    which makes 
    \begin{equation}
    \begin{aligned}
        &\frac{1}{\mathcal{U}}\left\Vert
        \sum_{t = 1 + \lambda_1}^T\sum_{t_1 = (t - \lambda_2)\vee 1}^{t - \lambda_1}\left(\gamma_{tt_1} - \gamma_{tt_1,S}\right)\left(\mathbf{x}_{t}^\top\mathbf{x}_{t_1} - \mathbf{E}\left[\mathbf{x}_{t}^\top\mathbf{x}_{t_1}\right]\right)
        \right\Vert_{M/2}\\
        &\leq \frac{C\sqrt{d}}{\mathcal{U}}\left\Vert\sqrt{\sum_{t = 1 + \lambda_1}^T\sum_{t_1 = (t - \lambda_2)\vee 1}^{t - \lambda_1}\left(\gamma_{tt_1} - \gamma_{tt_1,S}\right)^2}\right\Vert_{M/2}\\
        &\leq \frac{C\sqrt{d}}{\mathcal{U}}\sqrt{\sum_{t = 1 + \lambda_1}^T\sum_{t_1 = (t - \lambda_2)\vee 1}^{t - \lambda_1}\left\Vert\gamma_{tt_1} - \gamma_{tt_1,S}\right\Vert^2_{M/2}}\leq \frac{C_1}{(S-h-\lambda_2)^\alpha}.
    \end{aligned}
    \label{eq.truncate_fir_half}
    \end{equation}
    Similarly, notice that $\gamma_{tt_1,S}$ is a function of $\cdots, e_{t - 2}^{(2)}, e_{t - 1}^{(2)}, e_t^{(2)}$ and $\mathbf{x}_{t}^\top\mathbf{x}_{t_1} - \mathbf{E}\left[\mathbf{x}_{t}^\top\mathbf{x}_{t_1}\mid\mathcal{F}_{t,S}\right]$ is a function of $\cdots, e_{t - 2}^{(1)}, e_{t - 1}^{(1)}, e_t^{(1)}$ for any $t, t_1,$ the sequence 
    $$
    \left\{\gamma_{tt_1,S}: t = 1+\lambda_1,\cdots, T,\quad t_1 = (t-\lambda_2)\vee 1,\cdots, t-\lambda_1\right\}
    $$ is independent of the sequence  $$
    \left\{\mathbf{x}_{t}^\top\mathbf{x}_{t_1} - \mathbf{E}\left[\mathbf{x}_{t}^\top\mathbf{x}_{t_1}\mid\mathcal{F}_{t,S}\right]: t = 1+\lambda_1,\cdots, T,\quad t_1 = (t-\lambda_2)\vee 1,\cdots, t-\lambda_1\right\},
    $$ 
    and Theorem \ref{theorem.linear_and_quadratic} implies 
    \begin{align*}
        &\mathbf{E}\left[\left(\sum_{t = 1 + \lambda_1}^T\sum_{t_1 = (t - \lambda_2)\vee 1}^{t - \lambda_1}\gamma_{tt_1,S}\left(\mathbf{x}_{t}^\top\mathbf{x}_{t_1} - \mathbf{E}\left[\mathbf{x}_{t}^\top\mathbf{x}_{t_1}\mid\mathcal{F}_{t,S}\right]\right)\right)^{M/2}\right.\\
        &\mid \gamma_{tt_1,S} = c_{tt_1}, t = 1+\lambda_1,\cdots, T, t_1 = (t-\lambda_2)\vee 1,\cdots, t-\lambda_1\Bigg]\\
        & = \mathbf{E}\left[\left(\sum_{t = 1 + \lambda_1}^T\sum_{t_1 = (t - \lambda_2)\vee 1}^{t - \lambda_1}c_{tt_1}\left(\mathbf{z}_{t}^\top \mathbf{I}_{d_1}\mathbf{z}_{t_1} - \mathbf{E}\left[\mathbf{z}_{t}^\top \mathbf{I}_{d_1}\mathbf{z}_{t_1}\mid\mathcal{F}_{t,S}\right]\right)\right)^{M/2}\right]\\
        &\leq \left(\frac{CT\sqrt{T}}{(S - \lambda_2)^{\alpha - 2}}\max_{t = 1+\lambda_1,\cdots,T;\ t_1 = (t-\lambda_2)\vee 1,\cdots, t-\lambda_1}\left\vert c_{tt_1}\right\vert\right)^{M/2},
    \end{align*}
    and notice that $M > 8,$ which makes 
    \begin{equation}
    \begin{aligned}
        &\frac{1}{\mathcal{U}}\left\Vert
        \sum_{t = 1 + \lambda_1}^T\sum_{t_1 = (t - \lambda_2)\vee 1}^{t - \lambda_1}\gamma_{tt_1}\left(\mathbf{x}_{t}^\top\mathbf{x}_{t_1} - \mathbf{E}\left[\mathbf{x}_{t}^\top\mathbf{x}_{t_1}\mid\mathcal{F}_{t,S}\right]\right)
        \right\Vert_{M/2}\\
        &\leq \frac{CT\sqrt{T}}{\mathcal{U}(S - \lambda_2)^{\alpha - 2}}\left\Vert\max_{t = 1+\lambda_1,\cdots,T;\ t_1 = (t-\lambda_2)\vee 1,\cdots, t-\lambda_1}\left\vert \gamma_{tt_1,S}\right\vert\ \right\Vert_{M/2}\\
        &\leq \frac{C_1\sqrt{d}T^{2/M}(\lambda_2 - \lambda_1)^{2/M - 1/2}}{(S-\lambda_2)^{\alpha - 2}}\leq \frac{C_2T^{3/4}}{(S - \lambda_2 - h)^{\alpha - 2}}.
    \end{aligned}
     \label{eq.truncate_sec_half}
    \end{equation}
    From equations \eqref{eq.truncate_fir_half}, \eqref{eq.truncate_sec_half}, we have 
    \begin{align*}
        &\frac{1}{\mathcal{U}}\left\Vert
        \sum_{t = 1 + \lambda_1}^T\sum_{t_1 = (t - \lambda_2)\vee 1}^{t - \lambda_1}\sum_{s = 0}^{h\wedge (t_1 - 1)}\left(\mathbf{x}_{t}^\top\mathbf{x}_{t_1} - \mathbf{E}\left[\mathbf{x}_{t}^\top\mathbf{x}_{t_1}\right]\right)\left(\mathbf{y}_{t - s}^\top\mathbf{y}_{t_1 - s} - \mathbf{E}\left[\mathbf{y}_{t - s}^\top\mathbf{y}_{t_1 - s}\right]\right)\right.\\
        &\left. -\sum_{t = 1 + \lambda_1}^T\sum_{t_1 = (t - \lambda_2)\vee 1}^{t - \lambda_1}\sum_{s = 0}^{h\wedge (t_1 - 1)}\left(\mathbf{E}\left[\mathbf{x}_{t}^\top\mathbf{x}_{t_1}\mid\mathcal{F}_{t,S}\right] - \mathbf{E}\left[\mathbf{x}_{t}^\top\mathbf{x}_{t_1}\right]\right)\left(\mathbf{E}\left[\mathbf{y}_{t - s}^\top\mathbf{y}_{t_1 - s}\mid\mathcal{F}_{t,S}\right] - \mathbf{E}\left[\mathbf{y}_{t - s}^\top\mathbf{y}_{t_1 - s}\right]\right)
        \right\Vert_{M/2}\\
        &\leq \frac{CT^{3/4}}{(S - \lambda_2 - h)^{\alpha - 2}}.
    \end{align*}
    Similarly, we can derive 
    \begin{align*}
        &\frac{1}{\mathcal{U}}\left\Vert\sum_{t = 1 + \lambda_1}^T\sum_{t_1 = (t - \lambda_2)\vee 1}^{t - \lambda_1}\sum_{s = 1}^{h\wedge(t_1 - 1)}\left(\mathbf{x}_{t - s}^\top\mathbf{x}_{t_1 - s} - \mathbf{E}\left[\mathbf{x}_{t - s}^\top\mathbf{x}_{t_1 - s}\right]\right)\left(\mathbf{y}_{t}^\top\mathbf{y}_{t_1} - \mathbf{E}\left[\mathbf{y}_{t}^\top\mathbf{y}_{t_1}\right]\right)\right.\\
        &\left. - \sum_{t = 1 + \lambda_1}^T\sum_{t_1 = (t - \lambda_2)\vee 1}^{t - \lambda_1}\sum_{s = 1}^{h\wedge(t_1 - 1)}\left(\mathbf{E}\left[\mathbf{x}_{t - s}^\top\mathbf{x}_{t_1 - s}\mid\mathcal{F}_{t,S}\right] - \mathbf{E}\left[\mathbf{x}_{t - s}^\top\mathbf{x}_{t_1 - s}\right]\right)\left(\mathbf{E}\left[\mathbf{y}_{t}^\top\mathbf{y}_{t_1}\mid\mathcal{F}_{t,S}\right] - \mathbf{E}\left[\mathbf{y}_{t}^\top\mathbf{y}_{t_1}\right]\right)
        \right\Vert_{M/2}\\
        & \leq \frac{CT^{3/4}}{(S - \lambda_2 - h)^{\alpha - 2}},
    \end{align*}
    which implies that
    \begin{equation}
    \begin{aligned}
        \left\Vert
            \widehat{R}^\dagger - \widehat{R}^\dagger_S
            \right\Vert_{M/2}\leq \frac{CT^{3/4}}{(S - \lambda_2 - h)^{\alpha - 2}},
    \end{aligned}
    \label{eq.R_R_s}
    \end{equation}
and 
    \begin{equation}
        \begin{aligned}
            \sup_{x\in\mathbf{R}}\left\vert
            \mathbf{E}\left[g_{\psi,x}\left( \widehat{R}^\dagger\right)\right] - \mathbf{E}\left[g_{\psi,x}\left( \widehat{R}^\dagger_S\right)\right]
            \right\vert\leq \frac{C\psi T^{3/4}}{(S - \lambda_2 - h)^{\alpha - 2}}.
        \end{aligned}
        \label{eq.R_dagger_R_s}
    \end{equation}

Choose $W > S$ and define the ``big-block''
\begin{align*}
    A_z = \sum_{t = (1+\lambda_1)\vee ((z-1)(W+S) + 1)}^{T\wedge ((z-1)(W+S) + W)}O^\dagger_{t,S},\quad a_z = \sum_{t = (1+\lambda_1)\vee ((z-1)(W+S) + W+1)}^{T\wedge (z(W+S))}O^\dagger_{t,S}.
\end{align*}
Since $O^\dagger_{t, S}$ is measurable in $\mathcal{F}_{t,S},$ $A_z$ is measurable in the $\sigma$-field generated by $e_{T\wedge ((z-1)(W+S) + W)},\cdots, e_{(1+\lambda_1)\vee ((z-1)(W+S) + 1) - S},$ $A_z$ are mutually independent. Similarly, 
$a_z$ is measurable in the $\sigma$-field generated by $e_{T\wedge (z(W+S))},\cdots, e_{(1+\lambda_1)\vee ((z-1)(W+S) + W+1) - S},$ so  $a_z$ are mutually independent. Furthermore, from Corollary \ref{corollary.moment_expectation},
\begin{align*}
    &\left\Vert
    \frac{1}{\mathcal{U}}\sum_{t = (1+\lambda_1)\vee ((z-1)(W+S) + 1)}^{T\wedge ((z-1)(W+S) + W)}\sum_{t_1 = (t - \lambda_2)\vee 1}^{t - \lambda_1}\sum_{s = 0}^{h\wedge (t_1 - 1)}\left(\mathbf{E}\left[\mathbf{x}_{t}^\top\mathbf{x}_{t_1}\mid\mathcal{F}_{t,S}\right] - \mathbf{E}\left[\mathbf{x}_{t}^\top\mathbf{x}_{t_1}\right]\right)\right.\\
    &\left(\mathbf{E}\left[\mathbf{y}_{t - s}^\top\mathbf{y}_{t_1 - s}\mid\mathcal{F}_{t,S}\right] - \mathbf{E}\left[\mathbf{y}_{t - s}^\top\mathbf{y}_{t_1 - s}\right]\right)
    \Bigg\Vert_{M/2}\\
    & = \left\Vert
    \frac{1}{\mathcal{U}}\sum_{t = (1+\lambda_1)\vee ((z-1)(W+S) + 1)}^{T\wedge ((z-1)(W+S) + W)}\sum_{t_1 = (t - \lambda_2)\vee 1}^{t - \lambda_1}\gamma_{tt_1,S}\left(\mathbf{E}\left[\mathbf{x}_{t}^\top\mathbf{x}_{t_1}\mid\mathcal{F}_{t,S}\right] - \mathbf{E}\left[\mathbf{x}_{t}^\top\mathbf{x}_{t_1}\right]\right)
    \right\Vert_{M/2}\\
    &\leq \frac{C}{\mathcal{U}}\sqrt{d\sum_{t = (1+\lambda_1)\vee ((z-1)(W+S) + 1)}^{T\wedge ((z-1)(W+S) + W)}\sum_{t_1 = (t - \lambda_2)\vee 1}^{t - \lambda_1}\left\Vert \gamma_{tt_1,S}\right\Vert^2_{M/2}}\leq C_1\sqrt{\frac{W}{T}},
\end{align*}
similarly 
\begin{align*}
    &\left\Vert
    \frac{1}{\mathcal{U}}\sum_{t = (1+\lambda_1)\vee ((z-1)(W+S) + 1)}^{T\wedge ((z-1)(W+S) + W)}\sum_{t_1 = (t - \lambda_2)\vee 1}^{t - \lambda_1}\sum_{s = 0}^{h\wedge (t_1 - 1)}\left(\mathbf{E}\left[\mathbf{x}_{t - s}^\top\mathbf{x}_{t_1 - s}\mid\mathcal{F}_{t,S}\right] - \mathbf{E}\left[\mathbf{x}_{t - s}^\top\mathbf{x}_{t_1 - s}\right]\right)\right.\\
    &\left(\mathbf{E}\left[\mathbf{y}_{t}^\top\mathbf{y}_{t_1}\mid\mathcal{F}_{t,S}\right] - \mathbf{E}\left[\mathbf{y}_{t}^\top\mathbf{y}_{t_1}\right]\right)
    \Bigg\Vert_{M/2}\\
    &\leq C\sqrt{\frac{W}{T}},
\end{align*}
which makes 
\begin{equation}
    \begin{aligned}
        \left\Vert A_z\right\Vert_{M/2}\leq C\sqrt{\frac{W}{T}}.
    \end{aligned}
    \label{eq.A_z_moment}
\end{equation}
Also from Corollary \ref{corollary.moment_expectation}, 
\begin{align*}
&\left\Vert
    \frac{1}{\mathcal{U}}\sum_{t = (1+\lambda_1)\vee ((z-1)(W+S) + W+1)}^{T\wedge (z(W+S))} \sum_{t_1 = (t - \lambda_2)\vee 1}^{t - \lambda_1}\sum_{s = 0}^{h\wedge (t_1 - 1)}\left(\mathbf{E}\left[\mathbf{x}_{t}^\top\mathbf{x}_{t_1}\mid\mathcal{F}_{t,S}\right] - \mathbf{E}\left[\mathbf{x}_{t}^\top\mathbf{x}_{t_1}\right]\right)\right.\\
    &\left(\mathbf{E}\left[\mathbf{y}_{t - s}^\top\mathbf{y}_{t_1 - s}\mid\mathcal{F}_{t,S}\right] - \mathbf{E}\left[\mathbf{y}_{t - s}^\top\mathbf{y}_{t_1 - s}\right]\right)
\Bigg\Vert_{M/2}\\
& = \left\Vert
    \frac{1}{\mathcal{U}}\sum_{t = (1+\lambda_1)\vee ((z-1)(W+S) + W+1)}^{T\wedge (z(W+S))} \sum_{t_1 = (t - \lambda_2)\vee 1}^{t - \lambda_1}\gamma_{tt_1,S}\left(\mathbf{E}\left[\mathbf{x}_{t}^\top\mathbf{x}_{t_1}\mid\mathcal{F}_{t,S}\right] - \mathbf{E}\left[\mathbf{x}_{t}^\top\mathbf{x}_{t_1}\right]\right)
\right\Vert_{M/2}\\
&\leq \frac{C}{\mathcal{U}}\sqrt{d\sum_{t = (1+\lambda_1)\vee ((z-1)(W+S) + W+1)}^{T\wedge (z(W+S))} \sum_{t_1 = (t - \lambda_2)\vee 1}^{t - \lambda_1}\left\Vert\gamma_{tt_1,S}\right\Vert^2_{M/2}}\leq C_1\sqrt{\frac{S}{T}},
\end{align*}
which makes 
\begin{equation}
\begin{aligned}
    \left\Vert
    a_{z}
    \right\Vert_{M/2}\leq C\sqrt{\frac{S}{T}}.
\end{aligned}
\end{equation}
Define $P = \lceil\frac{T}{W+S}\rceil,$ where $\lceil x\rceil$ stands for the smallest integer that is larger than or equal to $x,$ then 
\begin{align*}
    \sum_{z = 1}^P A_z + \sum_{z = 1}^P a_z = \sum_{t  = 1+\lambda_1}^T O^\dagger_{t,S},
\end{align*}
so from Theorem 2 of \cite{MR0133849},
\begin{equation}
\begin{aligned}
    \left\Vert \widehat{R}^\dagger_S - \sum_{z = 1}^P A_z\right\Vert_{M/2} = \left\Vert\sum_{z = 1}^P a_z\right\Vert_{M/2} \leq C\sqrt{\sum_{z = 1}^P\left\Vert a_z\right\Vert^2_{M/2}}\leq C_1\sqrt{\frac{S}{W}},
\end{aligned}
\label{eq.R_to_A_z}
\end{equation}
and 
\begin{equation}
    \begin{aligned}
        \sup_{x\in\mathbf{R}}\left\vert
            \mathbf{E}\left[g_{\psi,x}\left( \widehat{R}^\dagger_S\right)\right] - \mathbf{E}\left[g_{\psi,x}\left(\sum_{z = 1}^P A_z\right)\right]
            \right\vert\leq C\psi\left\Vert \widehat{R}^\dagger_S - \sum_{z = 1}^P A_z\right\Vert_{M/2}\leq C\psi\sqrt{\frac{S}{W}}.
    \end{aligned}
    \label{eq.R_s_to_A_z}
\end{equation}
Define $A_z^*$ as independent random variables with $\mathbf{E}\left[A_z^*\right] = 0$ and $\mathrm{Var}\left(A_z^*\right) = \mathrm{Var}\left(A_z\right),$ and $A_z^*,z = 1,\cdots,P$ are independent of $\mathbf{z}_t, t\in\mathbf{Z}.$ For any $u = 1,\cdots, P,$ define 
\begin{align*}
    B_u = \sum_{z = 1}^{u - 1}A_z + \sum_{z = u+1}^P A^*_z,
\end{align*}
then 
\begin{align*}
    B_u + A_u = B_{u+1} + A_{u+1}^*.
\end{align*}
From Taylor's expansion,
\begin{align*}
    g_{\psi,x}\left(B_u + A_u\right) = g_{\psi,x}\left(B_u\right) + g_{\psi,x}^\prime\left(B_u\right)A_u + \frac{1}{2}g_{\psi,x}^{\prime\prime}\left(B_u\right)A_u^2 + \frac{1}{2}\int_{B_u}^{B_u+A_u} g_{\psi,x}^{\prime\prime\prime}(t)\left(B_u + A_u - t\right)^2\mathrm{d}t,
\end{align*}
and 
\begin{align*}
    g_{\psi,x}\left(B_u + A_u^*\right) = g_{\psi,x}\left(B_u\right) + g_{\psi,x}^\prime\left(B_u\right)A_u^* + \frac{1}{2}g_{\psi,x}^{\prime\prime}\left(B_u\right)A_u^{*2} + \frac{1}{2}\int_{B_u}^{B_u+A_u^*} g_{\psi,x}^{\prime\prime\prime}(t)\left(B_u + A^*_u - t\right)^2\mathrm{d}t.
\end{align*}
Therefore, 
\begin{align*}
    &\left\vert\mathbf{E}\left[g_{\psi,x}\left(B_u + A_u\right) - g_{\psi,x}\left(B_u + A_u^*\right)\mid B_u\right]
    \right\vert\\
    & = 
    \left\vert
    g_{\psi,x}^\prime\left(B_u\right)\left(\mathbf{E}\left[A_u\right] - \mathbf{E}\left[A_u^*\right]\right)
    + \frac{1}{2}g_{\psi,x}^{\prime\prime}\left(B_u\right)\left(\mathrm{Var}\left(A_u\right) - \mathrm{Var}\left(A^*_u\right)\right)\right.\\
    &\left. + \mathbf{E}\left[\frac{1}{2}\int_{B_u}^{B_u+A_u} g_{\psi,x}^{\prime\prime\prime}(t)\left(B_u + A_u - t\right)^2\mathrm{d}t\mid B_u\right] - \mathbf{E}\left[\frac{1}{2}\int_{B_u}^{B_u+A_u^*} g_{\psi,x}^{\prime\prime\prime}(t)\left(B_u + A^*_u - t\right)^2\mathrm{d}t\mid B_u\right]\right\vert\\
    &\leq  \frac{1}{2}\mathbf{E}\left[\int_{B_u}^{B_u+A_u} \left\vert g_{\psi,x}^{\prime\prime\prime}(t)\right\vert\left(B_u + A_u - t\right)^2\mathrm{d}t\mid B_u\right]
    + \frac{1}{2}\mathbf{E}\left[\int_{B_u}^{B_u + A^*_u} \left\vert g_{\psi,x}^{\prime\prime\prime}(t)\right\vert \left(B_u + A^*_u - t\right)^2\mathrm{d}t\mid B_u\right]\\
    &\leq C\psi^3\mathbf{E}\left[\left\vert A_u\right\vert^3\right] +  C\psi^3\mathbf{E}\left[\left\vert A_u^*\right\vert^3\right]\leq C\psi^3\left\Vert A_u\right\Vert_{M/2}^3  +  C\psi^3\left\Vert A_u^*\right\Vert_{M/2}^3.
\end{align*}
For $A^*_u$ has normal distribution, there exists a constant $C>0$ such that 
\begin{align*}
    \left\Vert A_u^*\right\Vert_{M/2}\leq C\left\Vert A_u^*\right\Vert_2 =  C\left\Vert A_u\right\Vert_2\leq C\left\Vert A_u\right\Vert_{M/2},
\end{align*}
so from \eqref{eq.A_z_moment}, 
\begin{align*}
    \left\vert\mathbf{E}\left[g_{\psi,x}\left(B_u + A_u\right) - g_{\psi,x}\left(B_u + A_u^*\right)\right]
    \right\vert &\leq \mathbf{E}\left[\left\vert\mathbf{E}\left[g_{\psi,x}\left(B_u + A_u\right) - g_{\psi,x}\left(B_u + A_u^*\right)\mid B_u\right]
    \right\vert\right]\\
    &\leq C\psi^3\left(\frac{W}{T}\right)^{3/2}.
\end{align*}
By taking summations, we have 
\begin{equation}
\begin{aligned}
         &\left\vert
     \mathbf{E}\left[g_{\psi,x}\left(\sum_{z = 1}^P A_z\right)\right] - \mathbf{E}\left[g_{\psi,x}\left(\sum_{z = 1}^P A_z^*\right)\right]
            \right\vert\\
            &= \left\vert \mathbf{E}\left[g_{\psi,x}\left(B_P +A_P\right)\right] - \mathbf{E}\left[g_{\psi,x}\left(B_1 +A_1^*\right)\right]\right\vert\\
            &\leq \left\vert\mathbf{E}\left[g_{\psi,x}\left(B_1 +A_1\right)\right] - \mathbf{E}\left[g_{\psi,x}\left(B_1 +A_1^*\right)\right]\right\vert\\
            &+ \sum_{u = 2}^P\left\vert
            \mathbf{E}\left[g_{\psi,x}\left(B_u +A_u\right)\right] - \mathbf{E}\left[g_{\psi,x}\left(B_{u-1} +A_{u-1}\right)\right]
            \right\vert\\
            & = \sum_{u = 1}^P\left\vert
            \mathbf{E}\left[g_{\psi,x}\left(B_u +A_u\right)\right] - \mathbf{E}\left[g_{\psi,x}\left(B_{u} +A^*_{u}\right)\right]
            \right\vert\leq C\psi^3\left(\frac{W}{T}\right)^{3/2}P\leq C_1\psi^3\sqrt{\frac{W}{T}}.
\end{aligned}
\label{eq.R_3_W_T}
\end{equation}
Finally, since $A^*_z$ are mutually independent, 
\begin{align*}
    \left\Vert \sum_{z = 1}^P A_z^*\right\Vert_{2} = \sqrt{\mathrm{Var}\left(\sum_{z = 1}^P A_z^*\right)} = \sqrt{\sum_{z=  1}^P\mathrm{Var}\left(A_z^*\right)} = \sqrt{\sum_{z=  1}^P\mathrm{Var}\left(A_z\right)} = \left\Vert\sum_{z = 1}^P A_z\right\Vert_2,
\end{align*}
so from \eqref{eq.R_R_s}, and \eqref{eq.R_to_A_z},
\begin{align*}
    \left\vert\ \left\Vert\sum_{z = 1}^P A_z^*\right\Vert_{2} - \left\Vert\widehat{R}^\dagger\right\Vert_{2}\ \right\vert
    &\leq \left\vert \ \left\Vert\sum_{z = 1}^P A_z\right\Vert_2 - \left\Vert\sum_{z = 1}^P A_z +\sum_{z=1}^Pa_z\right\Vert_2\ \right\vert
    + \left\vert
    \ \left\Vert\sum_{t  = 1+\lambda_1}^T O^\dagger_{t,S}\right\Vert_2 - \left\Vert\widehat{R}^\dagger\right\Vert_2\ 
    \right\vert\\
    &\leq \left\Vert \sum_{z=1}^Pa_z\right\Vert_{M/2}
    + \left\Vert \widehat{R}^\dagger - \sum_{t  = 1+\lambda_1}^T O^\dagger_{t,S}\right\Vert_{M/2}
    \leq \frac{CT^{3/4}}{(S - \lambda_2 - h)^{\alpha - 2}} + C\sqrt{\frac{S}{W}}.
\end{align*}
Define $Z$ as a standard normal random variable (i.e., $\mathbf{E}\left[Z\right] = 0$ and $\mathrm{Var}\left(Z\right) = 1$). For  $\sum_{z = 1}^P A_z^*$ has normal distribution with mean 0, the distribution of $\sum_{z = 1}^P A_z^*$ coincide with $\left\Vert\sum_{z = 1}^P A_z^*\right\Vert_{2}Z.$ Similarly, $\epsilon$ has normal distribution with mean 0 and variance $\mathrm{Var}\left(\widehat{R}^\dagger\right),$ so the distribution of $\epsilon$ coincide with $\sqrt{\mathrm{Var}\left(\widehat{R}^\dagger\right)}Z.$ From this observation, we have 
\begin{equation}
\begin{aligned}
    \sup_{x\in\mathbf{R}}\left\vert
    \mathbf{E}\left[g_{\psi,x}\left(\sum_{z = 1}^P A_z^*\right)\right] - \mathbf{E}\left[g_{\psi,x}\left(\epsilon\right)\right]
    \right\vert &= \sup_{x\in\mathbf{R}}\left\vert
    \mathbf{E}\left[g_{\psi,x}\left(\left\Vert\sum_{z = 1}^P A_z^*\right\Vert_{2}Z\right)\right] - \mathbf{E}\left[g_{\psi,x}\left(\left\Vert\widehat{R}^\dagger\right\Vert_{2}Z\right)\right]
    \right\vert\\
    &\leq C\psi\left\Vert Z\right\Vert_1\left\vert\  \left\Vert\sum_{z = 1}^P A_z^*\right\Vert_{2} - \left\Vert\widehat{R}^\dagger\right\Vert_{2}\ \right\vert\\
    &\leq \frac{C_1\psi T^{3/4}}{(S - \lambda_2 - h)^{\alpha - 2}} + C_1\psi\sqrt{\frac{S}{W}}.
\end{aligned}
\label{eq.g_psi_x_to_epsilon}
\end{equation}
Choose $\psi = \log(T), S = \lfloor h + 3\lambda_2\rfloor, W = \lfloor\sqrt{TS}\rfloor,$ where $\lfloor x\rfloor$ stands for the largest integer that is smaller than or equal to $x.$ Since $S = o(T), $ $W < T$ for sufficiently large $T,$ 
\begin{align*}
    &\frac{\psi T^{3/4}}{(S - \lambda_2 - h)^{\alpha - 2}}\leq \frac{C\log(T) T^{3/4}}{T^{5/2}} = \frac{C\log(T)}{T^{7/4}},\quad \psi\sqrt{\frac{S}{W}}\leq C\log(T)\left(\frac{1}{T^{1-\kappa_h}} + \frac{1}{T^{1-\kappa_2}}\right)^{1/4},\\
    &\psi^3\sqrt{\frac{W}{T}} \leq C\log^3(T)\left(\frac{1}{T^{1-\kappa_h}} + \frac{1}{T^{1-\kappa_2}}\right)^{1/4}.
\end{align*}
From \eqref{eq.Prob_R_to_epsilon}, \eqref{eq.R_to_R_dagger}, \eqref{eq.R_dagger_R_s}, \eqref{eq.R_s_to_A_z}, \eqref{eq.R_3_W_T}, \eqref{eq.g_psi_x_to_epsilon}, we prove \eqref{eq.Prob_close}.
\end{proof}

We then prove the validity of the bootstrap algorithm \ref{algorithm.dependent_wild_bootstrap}, which is summarized in Theorem \ref{theorem.validity_of_bootstrap_algorithm}. Since Algorithm \ref{algorithm.dependent_wild_bootstrap} samples joint normal random variables and constructs the bootstrap statistic as a linear combination of these sampled normal random variables, its validity actually relies on the fact that the variance of the  bootstrap statistic in the bootstrap world closely approximates that of the estimator $\widehat{R},$ which is proved below.

\begin{assumption}
Define $\alpha$ and $\kappa_z, z = 1,2,h$ as those in Assumption \ref{assumption.independent_short_range_structure} and Assumption \ref{assumption.lambda_choices}. 
Assume that $\alpha > 17$ and $ 1/ 6> \kappa_2 >\kappa_1 > \frac{5}{2(\alpha - 2)}$ and $1/6 > \kappa_h > 0.$ Suppose the bandwidth $\mathcal{K}_T\asymp T^{\kappa_k},$ where $ 1 / 2 - (\kappa_2\vee \kappa_h) > \kappa_k > 2(\kappa_2\vee \kappa_h).$
\label{assumption.for_variance_estimation}
\end{assumption}
Assumption \ref{assumption.for_variance_estimation} imposes a stronger condition than Assumption \ref{assumption.independent_short_range_structure} and Assumption \ref{assumption.lambda_choices}. This trade-off is required for the consistent estimation of the variance of the estimator $\widehat{R}.$

\begin{proof}[Proof of Theorem \ref{theorem.consistency_bootstrap}]
    Define ``the expectation in the bootstrap world'' as the conditional expectation conditional on $\mathbf{z}_t, t = 1,\cdots,T,$ i.e., 
    $$
    \mathbf{E}^*\left[\cdot\right] = \mathbf{E}\left[\cdot\mid\mathbf{z}_t, t = 1,\cdots,T\ \right].
    $$
    In the bootstrap world, since $e^*_t, t = 1,\cdots, T$ has joint normal distribution, $\widehat{R}^*_b = \sum_{t = 1+ \lambda_1}^T O_te^*_t$ also has normal distribution with 
    \begin{align*}
        \mathbf{E}^*\left[\widehat{R}^*_b\right] = \sum_{t = 1+ \lambda_1}^T O_t\mathbf{E}^*\left[e^*_t\right] = 0,
    \end{align*}
    and 
    \begin{align*}
       \mathbf{E}^*\left[\widehat{R}^{*2}_b\right] = \sum_{t_1 = 1+ \lambda_1}^T\sum_{t_2 = 1+ \lambda_1}^T O_{t_1}O_{t_2}\mathbf{E}^*\left[e^*_{t_1}e^*_{t_2}\right] = \sum_{t_1 = 1+ \lambda_1}^T\sum_{t_2 = 1+ \lambda_1}^T O_{t_1}O_{t_2}K\left(\frac{t_1  -  t_2}{\mathcal{K}_T}\right).
    \end{align*}
    For 
    \begin{align*}
        \mathrm{Var}\left(\widehat{R}^\dagger\right) = \mathbf{E}\left[\left(\sum_{t = 1 + \lambda_1}^T O_t^\dagger\right)^2\right] = \sum_{t_1 = 1 + \lambda_1}^T\sum_{t_2 = 1 + \lambda_1}^T\mathbf{E}\left[O_{t_1}^\dagger O_{t_2}^\dagger\right],
    \end{align*}
    we have 
    \begin{align*}
        \left\Vert
        \mathbf{E}^*\left[\widehat{R}^{*2}_b\right] -  \mathrm{Var}\left(\widehat{R}^\dagger\right)
        \right\Vert_{M/4}&\leq \left\Vert \sum_{t_1 = 1+ \lambda_1}^T\sum_{t_2 = 1+ \lambda_1}^T K\left(\frac{t_1  -  t_2}{\mathcal{K}_T}\right)\left(O_{t_1}O_{t_2} - O_{t_1}^\dagger O_{t_2}^\dagger\right)\right\Vert_{M/4}\\
        & + \left\Vert
        \sum_{t_1 = 1+ \lambda_1}^T\sum_{t_2 = 1+ \lambda_1}^T K\left(\frac{t_1  -  t_2}{\mathcal{K}_T}\right)\left(O_{t_1}^\dagger O_{t_2}^\dagger - \mathbf{E}\left[O_{t_1}^\dagger O_{t_2}^\dagger\right]\right)
        \right\Vert_{M/4}\\
        & + \sum_{t_1 = 1+ \lambda_1}^T\sum_{t_2 = 1+ \lambda_1}^T\left(1 - K\left(\frac{t_1  -  t_2}{\mathcal{K}_T}\right)\right)\left\vert \mathbf{E}\left[O_{t_1}^\dagger O_{t_2}^\dagger\right]\right\vert.
    \end{align*}
    For the first term,
    \begin{align*}
        &\left\Vert \sum_{t_1 = 1+ \lambda_1}^T\sum_{t_2 = 1+ \lambda_1}^T K\left(\frac{t_1  -  t_2}{\mathcal{K}_T}\right)\left(O_{t_1}O_{t_2} - O_{t_1}^\dagger O_{t_2}^\dagger\right)\right\Vert_{M/4}\\
        &\leq \sum_{t_1 = 1+ \lambda_1}^T\sum_{t_2 = 1+ \lambda_1}^T K\left(\frac{t_1  -  t_2}{\mathcal{K}_T}\right)\left\Vert O_{t_1} - O_{t_1}^\dagger\right\Vert_{M/2}\left\Vert O_{t_2}^\dagger\right\Vert_{M/2}\\
        & + \sum_{t_1 = 1+ \lambda_1}^T\sum_{t_2 = 1+ \lambda_1}^T K\left(\frac{t_1  -  t_2}{\mathcal{K}_T}\right)\left\Vert O_{t_1}^\dagger\right\Vert_{M/2}\left\Vert  O_{t_2} - O_{t_2}^\dagger\right\Vert_{M/2}\\
        & + \sum_{t_1 = 1+ \lambda_1}^T\sum_{t_2 = 1+ \lambda_1}^T K\left(\frac{t_1  -  t_2}{\mathcal{K}_T}\right)\left\Vert O_{t_1} - O_{t_1}^\dagger\right\Vert_{M/2}\left\Vert  O_{t_2} - O_{t_2}^\dagger\right\Vert_{M/2}.
    \end{align*}
Define $\gamma_{tt_1}$ as in \eqref{eq.def_gamma_tt1}, from Theorem \ref{theorem.linear_and_quadratic} and \eqref{eq.gamma_norm},
\begin{align*}
    &\frac{1}{\mathcal{U}}\left\Vert \sum_{t_1 = (t - \lambda_2)\vee 1}^{t - \lambda_1}\sum_{s = 0}^{h\wedge (t_1 - 1)}\left(\mathbf{x}_{t}^\top\mathbf{x}_{t_1} - \mathbf{E}\left[\mathbf{x}_{t}^\top\mathbf{x}_{t_1}\right]\right)\left(\mathbf{y}_{t - s}^\top\mathbf{y}_{t_1 - s} - \mathbf{E}\left[\mathbf{y}_{t - s}^\top\mathbf{y}_{t_1 - s}\right]\right)\right\Vert_{M/2}\\
    & = \frac{1}{\mathcal{U}}\left\Vert \sum_{t_1 = (t - \lambda_2)\vee 1}^{t - \lambda_1}\gamma_{tt_1}\left(\mathbf{z}_{t}^\top\mathbf{I}_{d_1}\mathbf{z}_{t_1} - \mathbf{E}\left[\mathbf{z}_{t}^\top \mathbf{I}_{d_1}\mathbf{z}_{t_1}\right]\right)\right\Vert_{M/2}\\
    &\leq \frac{C}{\mathcal{U}}\sqrt{d\sum_{t_1 = (t - \lambda_2)\vee 1}^{t - \lambda_1}\left\Vert\gamma_{tt_1}\right\Vert^2_{M/2}}\leq \frac{C_1}{\sqrt{T}}.
\end{align*}
Similarly, 
\begin{align*}
    \frac{1}{\mathcal{U}}\left\Vert\sum_{t_1 = (t - \lambda_2)\vee 1}^{t - \lambda_1}\sum_{s = 1}^{h\wedge(t_1 - 1)}\left(\mathbf{x}_{t - s}^\top\mathbf{x}_{t_1 - s} - \mathbf{E}\left[\mathbf{x}_{t - s}^\top\mathbf{x}_{t_1 - s}\right]\right)\left(\mathbf{y}_{t}^\top\mathbf{y}_{t_1} - \mathbf{E}\left[\mathbf{y}_{t}^\top\mathbf{y}_{t_1}\right]\right)\right\Vert_{M/2}\leq \frac{C}{\sqrt{T}},
\end{align*}
and 
\begin{equation}
    \left\Vert O_t^\dagger\right\Vert_{M/2}\leq \frac{C}{\sqrt{T}}.
    \label{eq.moment_O_t}
\end{equation}
For 
\begin{align*}
        &\frac{1}{\mathcal{U}}\left\Vert\sum_{t_1 = (t - \lambda_2)\vee 1}^{t - \lambda_1}\sum_{s = 0}^{h\wedge (t_1 - 1)}\mathbf{x}_{t}^\top\mathbf{x}_{t_1}\mathbf{y}_{t - s}^\top\mathbf{y}_{t_1 - s}\right.\\
        &\left.- \sum_{t_1 = (t - \lambda_2)\vee 1}^{t - \lambda_1}\sum_{s = 0}^{h\wedge (t_1 - 1)}\left(\mathbf{x}_{t}^\top\mathbf{x}_{t_1} - \mathbf{E}\left[\mathbf{x}_{t}^\top\mathbf{x}_{t_1}\right]\right)\left(\mathbf{y}_{t - s}^\top\mathbf{y}_{t_1 - s} - \mathbf{E}\left[\mathbf{y}_{t - s}^\top\mathbf{y}_{t_1 - s}\right]\right)\right\Vert_{M/2}\\
        &\leq I + II + III,\\
    \end{align*}
    where 
    \begin{align*}
        & I =  \frac{1}{\mathcal{U}}\left\Vert
        \sum_{t_1 = (t - \lambda_2)\vee 1}^{t - \lambda_1}\sum_{s = 0}^{h\wedge (t_1 - 1)}\mathbf{E}\left[\mathbf{x}_{t}^\top\mathbf{x}_{t_1}\right]\left(\mathbf{y}_{t - s}^\top\mathbf{y}_{t_1 - s} - \mathbf{E}\left[\mathbf{y}_{t - s}^\top\mathbf{y}_{t_1 - s}\right]\right)
        \right\Vert_{M/2},\\
        & II  = \frac{1}{\mathcal{U}}\left\Vert
        \sum_{t_1 = (t - \lambda_2)\vee 1}^{t - \lambda_1}\sum_{s = 0}^{h\wedge (t_1 - 1)}
        \mathbf{E}\left[\mathbf{y}_{t - s}^\top\mathbf{y}_{t_1 - s}\right]
        \left(\mathbf{x}_{t}^\top\mathbf{x}_{t_1} - \mathbf{E}\left[\mathbf{x}_{t}^\top\mathbf{x}_{t_1}\right]\right)
        \right\Vert_{M/2},\\
        & III = \frac{1}{\mathcal{U}}\sum_{t_1 = (t - \lambda_2)\vee 1}^{t - \lambda_1}\sum_{s = 0}^{h\wedge (t_1 - 1)}\left\vert \mathbf{E}\left[\mathbf{x}_{t}^\top\mathbf{x}_{t_1}\right]\right\vert\left\vert \mathbf{E}\left[\mathbf{y}_{t - s}^\top\mathbf{y}_{t_1 - s}\right]\right\vert.
    \end{align*}
From \eqref{eq.covariances_x_y} and \eqref{eq.gamma_norm}
\begin{align*}
    I &\leq \frac{1}{\mathcal{U}}\sum_{t_1 = (t - \lambda_2)\vee 1}^{t - \lambda_1}\left\vert \mathbf{E}\left[\mathbf{x}_{t}^\top\mathbf{x}_{t_1}\right]\right\vert\left\Vert
        \sum_{s = 0}^{h\wedge (t_1 - 1)}\left(\mathbf{y}_{t - s}^\top\mathbf{y}_{t_1 - s} - \mathbf{E}\left[\mathbf{y}_{t - s}^\top\mathbf{y}_{t_1 - s}\right]\right)
        \right\Vert_{M/2}\\
        &\leq \frac{C}{\mathcal{U}}\sum_{t_1 = (t - \lambda_2)\vee 1}^{t - \lambda_1}\frac{d_1\sqrt{dh}}{(t - t_1)^\alpha}\leq \frac{C}{\sqrt{\lambda_2 - \lambda_1}\lambda_1^{\alpha - 1}}\leq \frac{C_1}{T^{5/2}}.
\end{align*}
From Theorem \ref{theorem.linear_and_quadratic}, 
\begin{align*}
    II\leq \frac{C}{\mathcal{U}}\sqrt{d\sum_{t_1 = (t-\lambda_2)\vee 1}^{t-\lambda_1}\left(\sum_{s = 0}^{h\wedge (t_1 - 1)}
        \mathbf{E}\left[\mathbf{y}_{t - s}^\top\mathbf{y}_{t_1 - s}\right]\right)^2}\leq \frac{C_1\sqrt{h}}{\sqrt{(\lambda_2 - \lambda_1)}\lambda_1^{\alpha - 1/2}}\leq \frac{C_2}{T^2}.
\end{align*}
From \eqref{eq.covariances_x_y}, 
\begin{align*}
    III\leq \frac{1}{\mathcal{U}}\sum_{t_1 = (t - \lambda_2)\vee 1}^{t - \lambda_1}\sum_{s = 0}^{h\wedge (t_1 - 1)}\frac{Cd^2}{(t - t_1)^{2\alpha}}\leq \frac{C_1\sqrt{Th}}{\sqrt{(\lambda_2 - \lambda_1)}\lambda_1^{2\alpha - 1}}
    \leq \frac{C_2}{T^4}.
\end{align*}
Therefore, for any $t = 1 + \lambda_1, \cdots,T,$ 
\begin{align*}
     &\frac{1}{\mathcal{U}}\left\Vert\sum_{t_1 = (t - \lambda_2)\vee 1}^{t - \lambda_1}\sum_{s = 0}^{h\wedge (t_1 - 1)}\mathbf{x}_{t}^\top\mathbf{x}_{t_1}\mathbf{y}_{t - s}^\top\mathbf{y}_{t_1 - s}\right.\\
        &\left.- \sum_{t_1 = (t - \lambda_2)\vee 1}^{t - \lambda_1}\sum_{s = 0}^{h\wedge (t_1 - 1)}\left(\mathbf{x}_{t}^\top\mathbf{x}_{t_1} - \mathbf{E}\left[\mathbf{x}_{t}^\top\mathbf{x}_{t_1}\right]\right)\left(\mathbf{y}_{t - s}^\top\mathbf{y}_{t_1 - s} - \mathbf{E}\left[\mathbf{y}_{t - s}^\top\mathbf{y}_{t_1 - s}\right]\right)\right\Vert_{M/2}\leq \frac{C}{T^2}.
\end{align*}
Similarly, 
\begin{align*}
    &\frac{1}{\mathcal{U}}\left\Vert
    \sum_{t_1 = (t - \lambda_2)\vee 1}^{t - \lambda_1}\sum_{s = 1}^{h\wedge(t_1 - 1)}\mathbf{x}_{t - s}^\top\mathbf{x}_{t_1 - s}\mathbf{y}_{t}^\top\mathbf{y}_{t_1}\right.\\
    &\left.
    - \sum_{t_1 = (t - \lambda_2)\vee 1}^{t - \lambda_1}\sum_{s = 1}^{h\wedge(t_1 - 1)}\left(\mathbf{x}_{t - s}^\top\mathbf{x}_{t_1 - s} - \mathbf{E}\left[\mathbf{x}_{t - s}^\top\mathbf{x}_{t_1 - s}\right]\right)\left(\mathbf{y}_{t}^\top\mathbf{y}_{t_1} - \mathbf{E}\left[\mathbf{y}_{t}^\top\mathbf{y}_{t_1}\right]\right)
    \right\Vert_{M/2}\leq  \frac{C}{T^2},
\end{align*}
and 
\begin{equation}
    \left\Vert O_{t} - O_{t}^\dagger\right\Vert_{M/2}\leq \frac{C}{T^2}.
    \label{eq.delta_O_O_t}
\end{equation}
From \eqref{eq.moment_O_t}, \eqref{eq.delta_O_O_t}, 
\begin{align*}
    &\sum_{t_1 = 1+ \lambda_1}^T\sum_{t_2 = 1+ \lambda_1}^T K\left(\frac{t_1  -  t_2}{\mathcal{K}_T}\right)\left\Vert O_{t_1} - O_{t_1}^\dagger\right\Vert_{M/2}\left\Vert O_{t_2}^\dagger\right\Vert_{M/2}\\
    &\leq \frac{C}{T^2\sqrt{T}}\sum_{t_1 = 1+ \lambda_1}^T\sum_{t_2 = 1+ \lambda_1}^T K\left(\frac{t_1  -  t_2}{\mathcal{K}_T}\right)\\
    & \leq  \frac{C_1}{T^2\sqrt{T}}\sum_{u = 0}^{T - 1 - \lambda_1} K\left(\frac{u}{\mathcal{K}_T}\right)\sum_{t_1 = 1 + \lambda_1}^{T-u}1\\
    &\leq \frac{C_1}{T\sqrt{T}}\sum_{u = 0}^\infty K\left(\frac{u}{\mathcal{K}_T}\right)\leq \frac{C_2\mathcal{K}_T}{T\sqrt{T}}\leq \frac{C_3}{T},
\end{align*}
similarly 
\begin{align*}
    \sum_{t_1 = 1+ \lambda_1}^T\sum_{t_2 = 1+ \lambda_1}^T K\left(\frac{t_1  -  t_2}{\mathcal{K}_T}\right)\left\Vert O_{t_1}^\dagger\right\Vert_{M/2}\left\Vert  O_{t_2} - O_{t_2}^\dagger\right\Vert_{M/2}\leq \frac{C}{T},
\end{align*}
and 
\begin{align*}
    &\sum_{t_1 = 1+ \lambda_1}^T\sum_{t_2 = 1+ \lambda_1}^T K\left(\frac{t_1  -  t_2}{\mathcal{K}_T}\right)\left\Vert O_{t_1} - O_{t_1}^\dagger\right\Vert_{M/2}\left\Vert  O_{t_2} - O_{t_2}^\dagger\right\Vert_{M/2}\\
    &\leq \frac{C}{T^4}\sum_{t_1 = 1+ \lambda_1}^T\sum_{t_2 = 1+ \lambda_1}^T K\left(\frac{t_1  -  t_2}{\mathcal{K}_T}\right)\leq \frac{C_1\mathcal{K}_T}{T^3}\leq \frac{C_2}{T^{5/2}}.
\end{align*}
Therefore, we have 
\begin{equation}
    \begin{aligned}
        \left\Vert \sum_{t_1 = 1+ \lambda_1}^T\sum_{t_2 = 1+ \lambda_1}^T K\left(\frac{t_1  -  t_2}{\mathcal{K}_T}\right)\left(O_{t_1}O_{t_2} - O_{t_1}^\dagger O_{t_2}^\dagger\right)\right\Vert_{M/4}\leq \frac{C}{T}.
    \end{aligned}
    \label{eq.OO_to_OOS}
\end{equation}
For any $S\geq 0$ define $O^\dagger_{t, S} $ as in equation \eqref{eq.def_O_t_S}, i.e., 
\begin{align*}
    O^\dagger_{t, S} &= \frac{1}{\mathcal{U}}\sum_{t_1 = (t - \lambda_2)\vee 1}^{t - \lambda_1}\sum_{s = 0}^{h\wedge (t_1 - 1)}\left(\mathbf{E}\left[\mathbf{x}_{t}^\top\mathbf{x}_{t_1}\mid\mathcal{F}_{t,S}\right] - \mathbf{E}\left[\mathbf{x}_{t}^\top\mathbf{x}_{t_1}\right]\right)\left(\mathbf{E}\left[\mathbf{y}_{t - s}^\top\mathbf{y}_{t_1 - s}\mid\mathcal{F}_{t,S}\right] - \mathbf{E}\left[\mathbf{y}_{t - s}^\top\mathbf{y}_{t_1 - s}\right]\right)\\
     & + \frac{1}{\mathcal{U}}\sum_{t_1 = (t - \lambda_2)\vee 1}^{t - \lambda_1}\sum_{s = 1}^{h\wedge(t_1 - 1)}\left(\mathbf{E}\left[\mathbf{x}_{t - s}^\top\mathbf{x}_{t_1 - s}\mid\mathcal{F}_{t,S}\right] - \mathbf{E}\left[\mathbf{x}_{t - s}^\top\mathbf{x}_{t_1 - s}\right]\right)\left(\mathbf{E}\left[\mathbf{y}_{t}^\top\mathbf{y}_{t_1}\mid\mathcal{F}_{t,S}\right] - \mathbf{E}\left[\mathbf{y}_{t}^\top\mathbf{y}_{t_1}\right]\right).
\end{align*}
From this definition, $O_{t,S}^\dagger$ is measurable in $\mathcal{F}_{t,S}.$ After that, we define $P_{t,S}^\dagger = O_t^\dagger - O_{t,S}^\dagger,$ and define   $\gamma_{tt_1,S}$ as in equation \eqref{eq.def_gamma_tt1S}, i.e., 
\begin{align*}
    \gamma_{tt_1,S} = \sum_{s = 0}^{h\wedge (t_1 - 1)}\left(\mathbf{E}\left[\mathbf{y}_{t - s}^\top\mathbf{y}_{t_1 - s}\mid\mathcal{F}_{t,S}\right] - \mathbf{E}\left[\mathbf{y}_{t - s}^\top\mathbf{y}_{t_1 - s}\right]\right).
\end{align*}
Define $O^\dagger_{t, S} = 0$ if $S < 0,$ in such case, $P_{t,S}^\dagger = O_t^\dagger - O_{t,S}^\dagger = O_t^\dagger.$ Define $\gamma_{tt_1}$ as in \eqref{eq.def_gamma_tt1}, then  
\begin{align*}
   &\sum_{t_1 = (t - \lambda_2)\vee 1}^{t - \lambda_1}\sum_{s = 0}^{h\wedge (t_1 - 1)}\left(\mathbf{x}_{t}^\top\mathbf{x}_{t_1} - \mathbf{E}\left[\mathbf{x}_{t}^\top\mathbf{x}_{t_1}\right]\right)\left(\mathbf{y}_{t - s}^\top\mathbf{y}_{t_1 - s} - \mathbf{E}\left[\mathbf{y}_{t - s}^\top\mathbf{y}_{t_1 - s}\right]\right)\\
   & = \sum_{t_1 = (t - \lambda_2)\vee 1}^{t - \lambda_1}\gamma_{tt_1}\left(\mathbf{x}_{t}^\top\mathbf{x}_{t_1} - \mathbf{E}\left[\mathbf{x}_{t}^\top\mathbf{x}_{t_1}\right]\right).
\end{align*}
Define $\mathcal{G}_{t,s}$ as the $\sigma$-field generated by $\cdots, e_{t-1}^{(2)}, e_t^{(2)}, e^{(1)}_t, e^{(1)}_{t-1},\cdots, e^{(1)}_{t - S},$ with $e_t^{(i)}, i = 1,2$ defined in Assumption \ref{assumption.independent_short_range_structure}. Then $\mathcal{F}_{t,S}\subset \mathcal{G}_{t,s},$ and $\gamma_{tt_1}$ is measurable in $\mathcal{G}_{t,s}.$ Since 
\begin{align*}
    &\mathbf{E}\left[\sum_{t_1 = (t - \lambda_2)\vee 1}^{t - \lambda_1}\gamma_{tt_1}\left(\mathbf{x}_{t}^\top\mathbf{x}_{t_1} - \mathbf{E}\left[\mathbf{x}_{t}^\top\mathbf{x}_{t_1}\right]\right)\mid\mathcal{G}_{t,S}\right]\\
    & = \sum_{t_1 = (t - \lambda_2)\vee 1}^{t - \lambda_1}\gamma_{tt_1}\left(\mathbf{E}\left[\mathbf{x}_{t}^\top\mathbf{x}_{t_1}\mid\mathcal{G}_{t,S}\right] - \mathbf{E}\left[\mathbf{x}_{t}^\top\mathbf{x}_{t_1}\right]\right)\\
    & = \sum_{t_1 = (t - \lambda_2)\vee 1}^{t - \lambda_1}\gamma_{tt_1}\left(\mathbf{E}\left[\mathbf{x}_{t}^\top\mathbf{x}_{t_1}\mid\mathcal{F}_{t,S}\right] - \mathbf{E}\left[\mathbf{x}_{t}^\top\mathbf{x}_{t_1}\right]\right),
\end{align*}
from the tower property of conditional expectation, we have  
\begin{align*}
    &\mathbf{E}\left[\sum_{t_1 = (t - \lambda_2)\vee 1}^{t - \lambda_1}\gamma_{tt_1}\left(\mathbf{x}_{t}^\top\mathbf{x}_{t_1} - \mathbf{E}\left[\mathbf{x}_{t}^\top\mathbf{x}_{t_1}\right]\right)\mid\mathcal{F}_{t,S}\right]\\
    & = \mathbf{E}\left[\sum_{t_1 = (t - \lambda_2)\vee 1}^{t - \lambda_1}\gamma_{tt_1}\left(\mathbf{E}\left[\mathbf{x}_{t}^\top\mathbf{x}_{t_1}\mid\mathcal{F}_{t,S}\right] - \mathbf{E}\left[\mathbf{x}_{t}^\top\mathbf{x}_{t_1}\right]\right)\mid\mathcal{F}_{t,S}\right]\\
    & = \sum_{t_1 = (t - \lambda_2)\vee 1}^{t - \lambda_1}\left(\mathbf{E}\left[\mathbf{x}_{t}^\top\mathbf{x}_{t_1}\mid\mathcal{F}_{t,S}\right] - \mathbf{E}\left[\mathbf{x}_{t}^\top\mathbf{x}_{t_1}\right]\right)\sum_{s = 0}^{h\wedge (t_1 - 1)}\left(\mathbf{E}\left[\mathbf{y}_{t - s}^\top\mathbf{y}_{t_1 - s}\mid \mathcal{F}_{t,S}\right] - \mathbf{E}\left[\mathbf{y}_{t - s}^\top\mathbf{y}_{t_1 - s}\right]\right).
\end{align*}
Similarly 
\begin{align*}
    &\mathbf{E}\left[\sum_{t_1 = (t - \lambda_2)\vee 1}^{t - \lambda_1}\sum_{s = 1}^{h\wedge(t_1 - 1)}\left(\mathbf{x}_{t - s}^\top\mathbf{x}_{t_1 - s} - \mathbf{E}\left[\mathbf{x}_{t - s}^\top\mathbf{x}_{t_1 - s}\right]\right)\left(\mathbf{y}_{t}^\top\mathbf{y}_{t_1} - \mathbf{E}\left[\mathbf{y}_{t}^\top\mathbf{y}_{t_1}\right]\right)\mid\mathcal{F}_{t,S}\right]\\
    & = \sum_{t_1 = (t - \lambda_2)\vee 1}^{t - \lambda_1}\sum_{s = 1}^{h\wedge(t_1 - 1)}\left(\mathbf{E}\left[\mathbf{y}_{t}^\top\mathbf{y}_{t_1}\mid\mathcal{F}_{t,S}\right] - \mathbf{E}\left[\mathbf{y}_{t}^\top\mathbf{y}_{t_1}\right]\right)\left(\mathbf{E}\left[\mathbf{x}_{t - s}^\top\mathbf{x}_{t_1 - s}\mid\mathcal{F}_{t,S}\right] - \mathbf{E}\left[\mathbf{x}_{t - s}^\top\mathbf{x}_{t_1 - s}\right]\right),
\end{align*}
so 
\begin{equation}
\begin{aligned}
    \mathbf{E}\left[O^\dagger_{t}\mid \mathcal{F}_{t,S}\right] = O^\dagger_{t,S}.
\end{aligned}
\label{eq.O_t_F_s}
\end{equation}

Suppose $S\geq 0.$ For any real values $c_{tt_1}\in\mathbf{R}, t_1 = (t-\lambda_2)\vee 1,\cdots, t - \lambda_1,$ from Theorem \ref{theorem.linear_and_quadratic}, 
\begin{align*}
    &\left\Vert\frac{1}{\mathcal{U}}\sum_{t_1 = (t - \lambda_2)\vee 1}^{t - \lambda_1}c_{tt_1}\left(\mathbf{E}\left[\mathbf{x}_{t}^\top\mathbf{x}_{t_1}\mid\mathcal{F}_{t,S}\right] - \mathbf{E}\left[\mathbf{x}_{t}^\top\mathbf{x}_{t_1}\right]\right)\right\Vert_{M/2}\\
    &= \left\Vert\mathbf{E}\left[\frac{1}{\mathcal{U}}\sum_{t_1 = (t - \lambda_2)\vee 1}^{t - \lambda_1}c_{tt_1}\left(\mathbf{x}_{t}^\top\mathbf{x}_{t_1} - \mathbf{E}\left[\mathbf{x}_{t}^\top\mathbf{x}_{t_1}\right]\right)\mid\mathcal{F}_{t,S}\right]\right\Vert_{M/2}\\
    &\leq \frac{C}{\mathcal{U}}\sqrt{d\sum_{t_1 = (t - \lambda_2)\vee 1}^{t - \lambda_1}c_{tt_1}^2},
\end{align*}
therefore
\begin{align*}
    &\frac{1}{\mathcal{U}}\left\Vert \sum_{t_1 = (t - \lambda_2)\vee 1}^{t - \lambda_1}\sum_{s = 0}^{h\wedge (t_1 - 1)}\left(\mathbf{E}\left[\mathbf{x}_{t}^\top\mathbf{x}_{t_1}\mid\mathcal{F}_{t,S}\right] - \mathbf{E}\left[\mathbf{x}_{t}^\top\mathbf{x}_{t_1}\right]\right)\left(\mathbf{E}\left[\mathbf{y}_{t - s}^\top\mathbf{y}_{t_1 - s}\mid\mathcal{F}_{t,S}\right] - \mathbf{E}\left[\mathbf{y}_{t - s}^\top\mathbf{y}_{t_1 - s}\right]\right)\right\Vert_{M/2}\\
    & = \frac{1}{\mathcal{U}}\left\Vert \sum_{t_1 = (t - \lambda_2)\vee 1}^{t - \lambda_1}\gamma_{tt_1,S}\left(\mathbf{E}\left[\mathbf{x}_{t}^\top\mathbf{x}_{t_1}\mid\mathcal{F}_{t,S}\right] - \mathbf{E}\left[\mathbf{x}_{t}^\top\mathbf{x}_{t_1}\right]\right)\right\Vert_{M/2}\\
    & \leq \frac{C}{\mathcal{U}}\sqrt{d\sum_{t_1 = (t - \lambda_2)\vee 1}^{t - \lambda_1}\left\Vert\gamma_{tt_1,S}\right\Vert^2_{M/2}}.
\end{align*}
From Corollary \ref{corollary.combination_lag}, 
\begin{align*}
    \left\Vert
    \gamma_{tt_1,S}
    \right\Vert_{M/2} &= 
    \left\Vert
    \mathbf{E}\left[\sum_{s = 0}^{h\wedge (t_1 - 1)}\left(\mathbf{y}_{t - s}^\top\mathbf{y}_{t_1 - s} - \mathbf{E}\left[\mathbf{y}_{t - s}^\top\mathbf{y}_{t_1 - s}\right]\right)\mid\mathcal{F}_{t,S}\right]
    \right\Vert_{M/2}\\
    &\leq \left\Vert \sum_{s = 0}^{h\wedge (t_1 - 1)}\left(\mathbf{y}_{t - s}^\top\mathbf{y}_{t_1 - s} - \mathbf{E}\left[\mathbf{y}_{t - s}^\top\mathbf{y}_{t_1 - s}\right]\right)\right\Vert_{M/2}\leq C\sqrt{dh},
\end{align*}
which makes 
\begin{align*}
    &\frac{1}{\mathcal{U}}\left\Vert \sum_{t_1 = (t - \lambda_2)\vee 1}^{t - \lambda_1}\sum_{s = 0}^{h\wedge (t_1 - 1)}\left(\mathbf{E}\left[\mathbf{x}_{t}^\top\mathbf{x}_{t_1}\mid\mathcal{F}_{t,S}\right] - \mathbf{E}\left[\mathbf{x}_{t}^\top\mathbf{x}_{t_1}\right]\right)\left(\mathbf{E}\left[\mathbf{y}_{t - s}^\top\mathbf{y}_{t_1 - s}\mid\mathcal{F}_{t,S}\right] - \mathbf{E}\left[\mathbf{y}_{t - s}^\top\mathbf{y}_{t_1 - s}\right]\right)\right\Vert_{M/2}\\
    &\leq \frac{C}{d\sqrt{T(\lambda_2 - \lambda_1)h}}\sqrt{d(\lambda_2 - \lambda_1)dh}\leq \frac{C_1}{\sqrt{T}}.
\end{align*}
Similarly
\begin{align*}
    &\frac{1}{\mathcal{U}}\left\Vert
    \sum_{t_1 = (t - \lambda_2)\vee 1}^{t - \lambda_1}\sum_{s = 1}^{h\wedge(t_1 - 1)}\left(\mathbf{E}\left[\mathbf{x}_{t - s}^\top\mathbf{x}_{t_1 - s}\mid\mathcal{F}_{t,S}\right] - \mathbf{E}\left[\mathbf{x}_{t - s}^\top\mathbf{x}_{t_1 - s}\right]\right)\left(\mathbf{E}\left[\mathbf{y}_{t}^\top\mathbf{y}_{t_1}\mid\mathcal{F}_{t,S}\right] - \mathbf{E}\left[\mathbf{y}_{t}^\top\mathbf{y}_{t_1}\right]\right)
    \right\Vert_{M/2}\\
    &\leq \frac{C}{\sqrt{T}},
\end{align*}
and 
\begin{equation}
    \begin{aligned}
        \left\Vert O_{t,S}^\dagger \right\Vert_{M/2}\leq \frac{C}{\sqrt{T}}.
        \label{eq.size_O_t_S}
    \end{aligned}
\end{equation}
For any $S\in\mathbf{Z},$ 
\begin{equation}
\begin{aligned}
    \left\Vert
    P^\dagger_{t,S}
    \right\Vert_{M/2}\leq \left\Vert O^\dagger _t\right\Vert_{M/2} + \left\Vert O^\dagger _{t,S} \right\Vert_{M/2}\leq \frac{C}{\sqrt{T}}.
\end{aligned}
\label{eq.size_P_t_S}
\end{equation}
For any $S > 2\lambda_2 + h,$ from Theorem \ref{theorem.linear_and_quadratic} and Corollary \ref{corollary.combination_lag},
\begin{align*}
&\frac{1}{\mathcal{U}}\left\Vert
\sum_{t_1 = (t - \lambda_2)\vee 1}^{t - \lambda_1}\sum_{s = 0}^{h\wedge (t_1 - 1)}\left(\mathbf{x}_{t}^\top\mathbf{x}_{t_1} - \mathbf{E}\left[\mathbf{x}_{t}^\top\mathbf{x}_{t_1}\right]\right)\left(\mathbf{y}_{t - s}^\top\mathbf{y}_{t_1 - s} - \mathbf{E}\left[\mathbf{y}_{t - s}^\top\mathbf{y}_{t_1 - s}\right]\right)\right.\\
&\left. - \sum_{t_1 = (t - \lambda_2)\vee 1}^{t - \lambda_1}\sum_{s = 0}^{h\wedge (t_1 - 1)}\left(\mathbf{E}\left[\mathbf{x}_{t}^\top\mathbf{x}_{t_1}\mid\mathcal{F}_{t,S}\right] - \mathbf{E}\left[\mathbf{x}_{t}^\top\mathbf{x}_{t_1}\right]\right)\left(\mathbf{E}\left[\mathbf{y}_{t - s}^\top\mathbf{y}_{t_1 - s}\mid\mathcal{F}_{t,S}\right] - \mathbf{E}\left[\mathbf{y}_{t - s}^\top\mathbf{y}_{t_1 - s}\right]\right) 
\right\Vert_{M/2}\\
& = \frac{1}{\mathcal{U}}\left\Vert
\sum_{t_1 = (t - \lambda_2)\vee 1}^{t - \lambda_1}\gamma_{tt_1}\left(\mathbf{x}_{t}^\top\mathbf{x}_{t_1} - \mathbf{E}\left[\mathbf{x}_{t}^\top\mathbf{x}_{t_1}\right]\right)- \sum_{t_1 = (t - \lambda_2)\vee 1}^{t - \lambda_1}\gamma_{tt_1,S}\left(\mathbf{E}\left[\mathbf{x}_{t}^\top\mathbf{x}_{t_1}\mid\mathcal{F}_{t,S}\right] - \mathbf{E}\left[\mathbf{x}_{t}^\top\mathbf{x}_{t_1}\right]\right) 
\right\Vert_{M/2}\\
&\leq \frac{1}{\mathcal{U}}\left\Vert\sum_{t_1 = (t - \lambda_2)\vee 1}^{t - \lambda_1}\left(\gamma_{tt_1} - \gamma_{tt_1,S}\right)\left(\mathbf{x}_{t}^\top\mathbf{x}_{t_1} - \mathbf{E}\left[\mathbf{x}_{t}^\top\mathbf{x}_{t_1}\right]\right)\right\Vert_{M/2}\\
&+ \frac{1}{\mathcal{U}}\left\Vert\sum_{t_1 = (t - \lambda_2)\vee 1}^{t - \lambda_1} \gamma_{tt_1,S}\left(\mathbf{x}_{t}^\top\mathbf{x}_{t_1} - \mathbf{E}\left[\mathbf{x}_{t}^\top\mathbf{x}_{t_1}\mid\mathcal{F}_{t,S}\right]\right)\right\Vert_{M/2}\\
&\leq \frac{C}{\mathcal{U}}\sqrt{d\sum_{t_1 = (t - \lambda_2)\vee 1}^{t - \lambda_1}\left\Vert \gamma_{tt_1} - \gamma_{tt_1,S}\right\Vert^2_{M/2}} + \frac{CT\sqrt{T}}{\mathcal{U}(S - \lambda_2)^{\alpha - 2}}\left\Vert\max_{t = 1,\cdots,T;\  t_1 = (t-\lambda_2)\vee 1,\cdots, t-\lambda_1}\left\vert \gamma_{tt_1,S}\right\vert\ \right\Vert_{M/2}\\
&\leq \frac{C_1}{\sqrt{T}(S - h - \lambda_2)^\alpha} + \frac{C_1\sqrt{T}}{\sqrt{(\lambda_2 - \lambda_1)}(S - \lambda_2)^{\alpha - 2}}T^{2/M}(\lambda_2 - \lambda_1)^{2/M}.
\end{align*}
Similarly,
\begin{align*}
    &\frac{1}{\mathcal{U}}\left\Vert
    \sum_{t_1 = (t - \lambda_2)\vee 1}^{t - \lambda_1}\sum_{s = 1}^{h\wedge(t_1 - 1)}\left(\mathbf{x}_{t - s}^\top\mathbf{x}_{t_1 - s} - \mathbf{E}\left[\mathbf{x}_{t - s}^\top\mathbf{x}_{t_1 - s}\right]\right)\left(\mathbf{y}_{t}^\top\mathbf{y}_{t_1} - \mathbf{E}\left[\mathbf{y}_{t}^\top\mathbf{y}_{t_1}\right]\right)\right.\\
    &\left.
    - \sum_{t_1 = (t - \lambda_2)\vee 1}^{t - \lambda_1}\sum_{s = 1}^{h\wedge(t_1 - 1)}\left(\mathbf{E}\left[\mathbf{x}_{t - s}^\top\mathbf{x}_{t_1 - s}\mid\mathcal{F}_{t,S}\right] - \mathbf{E}\left[\mathbf{x}_{t - s}^\top\mathbf{x}_{t_1 - s}\right]\right)\left(\mathbf{E}\left[\mathbf{y}_{t}^\top\mathbf{y}_{t_1}\mid\mathcal{F}_{t,S}\right] - \mathbf{E}\left[\mathbf{y}_{t}^\top\mathbf{y}_{t_1}\right]\right)
    \right\Vert_{M/2}\\
    &\leq \frac{C}{\sqrt{T}(S - h - \lambda_2)^\alpha} + \frac{C\sqrt{T}}{\sqrt{(\lambda_2 - \lambda_1)}(S - \lambda_2)^{\alpha - 2}}T^{2/M}(\lambda_2 - \lambda_1)^{2/M},
\end{align*}
and therefore
\begin{equation}
    \begin{aligned}
        \left\Vert
        P_{t,S}^\dagger 
        \right\Vert_{M/2} & = \left\Vert O_t^\dagger - O_{t,S}^\dagger\right\Vert_{M/2}\\
        &\leq \frac{C}{\sqrt{T}(S - h - \lambda_2)^\alpha} + \frac{C\sqrt{T}}{\sqrt{(\lambda_2 - \lambda_1)}(S - \lambda_2)^{\alpha - 2}}T^{2/M}(\lambda_2 - \lambda_1)^{2/M}.
    \end{aligned}
    \label{eq.def_P_t_S}
\end{equation}

Notice that 
\begin{align*}
    &\left\Vert
        \sum_{t_1 = 1+ \lambda_1}^T\sum_{t_2 = 1+ \lambda_1}^T K\left(\frac{t_1  -  t_2}{\mathcal{K}_T}\right)\left(O_{t_1}^\dagger O_{t_2}^\dagger - \mathbf{E}\left[O_{t_1}^\dagger O_{t_2}^\dagger\right]\right)
        \right\Vert_{M/4}\\
    &\leq \sum_{s = 0}^{T - 1 - \lambda_1} K\left(\frac{s}{\mathcal{K}_T}\right)\left\Vert\sum_{t_1 = 1 + \lambda_1}^{T - s}\left(O_{t_1}^\dagger O_{t_1 + s}^\dagger - \mathbf{E}\left[O_{t_1}^\dagger O_{t_1 + s}^\dagger\right]\right)\right\Vert_{M/4}\\
    & + \sum_{s = 1}^{T - 1 - \lambda_1} K\left(\frac{s}{\mathcal{K}_T}\right)\left\Vert\sum_{t_2 = 1 + \lambda_1}^{T - s}\left(O_{t_2 + s}^\dagger O_{t_2}^\dagger - \mathbf{E}\left[O_{t_2 + s}^\dagger O_{t_2}^\dagger\right]\right)\right\Vert_{M/4}.
\end{align*}

For any $s\geq 0,$   $O^\dagger_{t_1 + s} = O^\dagger_{t_1 + s, s - 1} + P^\dagger_{t_1 + s, s - 1},$ and $O^\dagger_{t_1 + s, s - 1}$ is independent of $O^\dagger_{t_1},$ which makes
\begin{align*}
    \mathbf{E}\left[O^\dagger_{t_1 + s, s - 1}O^\dagger_{t_1}\right] = \mathbf{E}\left[O^\dagger_{t_1 + s, s - 1}\right]\mathbf{E}\left[O^\dagger_{t_1}\right] = 0,
\end{align*}
and 
\begin{align*}
    \mathbf{E}\left[O^\dagger_{t_1 + s}O^\dagger_{t_1}\right] = \mathbf{E}\left[P^\dagger_{t_1 + s, s - 1}O^\dagger_{t_1}\right].
\end{align*}
Therefore,
\begin{align*}
    &\left\Vert\sum_{t_1 = 1 + \lambda_1}^{T - s}\left(O_{t_1}^\dagger O_{t_1 + s}^\dagger - \mathbf{E}\left[O_{t_1}^\dagger O_{t_1 + s}^\dagger\right]\right)\right\Vert_{M/4}\\
    &\leq \left\Vert\sum_{t_1 = 1 + \lambda_1}^{T - s}O^\dagger_{t_1 + s, s - 1}O^\dagger_{t_1}\right\Vert_{M/4} + \left\Vert \sum_{t_1 = 1 + \lambda_1}^{T - s}\left(O_{t_1}^\dagger P_{t_1 + s, s - 1}^\dagger - \mathbf{E}\left[O_{t_1}^\dagger P_{t_1 + s, s - 1}^\dagger\right]\right)\right\Vert_{M/4}.
\end{align*}
From \eqref{eq.O_t_F_s},
\begin{align*}
   \left\Vert\sum_{t_1 = 1 + \lambda_1}^{T - s}O^\dagger_{t_1 + s, s - 1}O^\dagger_{t_1}\right\Vert_{M/4} 
   \leq  \sum_{j = 0}^\infty\left\Vert\sum_{t_1 = 1 + \lambda_1}^{T - s}O^\dagger_{t_1 + s, s - 1}\left(\mathbf{E}\left[O^\dagger_{t_1}\mid\mathcal{F}_{t_1, j}\right] - \mathbf{E}\left[O^\dagger_{t_1}\mid\mathcal{F}_{t_1, j - 1}\right]\right)\right\Vert_{M/4}, 
\end{align*}
where $\mathbf{E}\left[O^\dagger_{t_1}\mid\mathcal{F}_{t_1, - 1}\right] = \mathbf{E}\left[O^\dagger_{t_1}\right] = 0.$
For any $j = 0,1,\cdots$ and $z = 1,\cdots, T - s - \lambda_1,$ define 
\begin{align*}
    V_{z,j} = \sum_{t_1 = T - s - z + 1}^{T - s}O^\dagger_{t_1 + s, s - 1}\left(\mathbf{E}\left[O^\dagger_{t_1}\mid\mathcal{F}_{t_1, j}\right] - \mathbf{E}\left[O^\dagger_{t_1}\mid\mathcal{F}_{t_1, j - 1}\right]\right)
\end{align*}
and $\mathcal{V}_{z,j}$ as the $\sigma$-field generated by $e_T, e_{T-1},\cdots, e_{T-s-z+1-j}.$ Then $V_{z,j}$ is measurable in  $\mathcal{V}_{z,j},$  $\mathcal{V}_{z,j}\subset\mathcal{V}_{z + 1,j},$ and 
\begin{align*}
    \mathbf{E}\left[V_{z + 1,j} - V_{z,j}\mid \mathcal{V}_{z,j}\right] &= 
    \mathbf{E}\left[O^\dagger_{T-z, s - 1}\left(\mathbf{E}\left[O^\dagger_{T-s-z}\mid\mathcal{F}_{T-s-z, j}\right] - \mathbf{E}\left[O^\dagger_{T-s-z}\mid\mathcal{F}_{T-s-z, j - 1}\right]\right)\mid \mathcal{V}_{z,j}\right]\\
    & = O^\dagger_{T-z, s - 1}\left(\mathbf{E}\left[O^\dagger_{T-s-z}\mid\mathcal{F}_{T-s-z, j - 1}\right] - \mathbf{E}\left[O^\dagger_{T-s-z}\mid\mathcal{F}_{T-s-z, j - 1}\right]\right) = 0.
\end{align*}
Therefore, $V_{z,j}$ forms a martingale, and according to Theorem 1.1 of \cite{MR0400380}, 
\begin{align*}
    &\left\Vert\sum_{t_1 = 1 + \lambda_1}^{T - s}O^\dagger_{t_1 + s, s - 1}\left(\mathbf{E}\left[O^\dagger_{t_1}\mid\mathcal{F}_{t_1, j}\right] - \mathbf{E}\left[O^\dagger_{t_1}\mid\mathcal{F}_{t_1, j - 1}\right]\right)\right\Vert_{M/4}\\
    &\leq C\sqrt{\sum_{t_1 = 1 + \lambda_1}^{T - s}\left\Vert O^\dagger_{t_1 + s, s - 1}\right\Vert^2_{M/2}\left\Vert \mathbf{E}\left[O^\dagger_{t_1}\mid\mathcal{F}_{t_1, j}\right] - \mathbf{E}\left[O^\dagger_{t_1}\mid\mathcal{F}_{t_1, j - 1}\right]\right\Vert^2_{M/2}}\\
    & = C\sqrt{\sum_{t_1 = 1 + \lambda_1}^{T - s}\left\Vert O^\dagger_{t_1 + s, s - 1}\right\Vert^2_{M/2}\left\Vert P^\dagger_{t_1,j - 1} - P^\dagger_{t_1,j}\right\Vert^2_{M/2}}.
\end{align*}
From \eqref{eq.size_O_t_S}, \eqref{eq.size_P_t_S}, and \eqref{eq.def_P_t_S}, if $j\leq 2\lambda_2 + h + 1,$
\begin{align*}
    \sqrt{\sum_{t_1 = 1 + \lambda_1}^{T - s}\left\Vert O^\dagger_{t_1 + s, s - 1}\right\Vert^2_{M/2}\left\Vert P^\dagger_{t_1,j - 1} - P^\dagger_{t_1,j}\right\Vert^2_{M/2}}\leq \frac{C}{\sqrt{T}}.
\end{align*}
On the other hand, if $j >  2\lambda_2 + h + 1,$
\begin{align*}
    &\sqrt{\sum_{t_1 = 1 + \lambda_1}^{T - s}\left\Vert O^\dagger_{t_1 + s, s - 1}\right\Vert^2_{M/2}\left\Vert P^\dagger_{t_1,j - 1} - P^\dagger_{t_1,j}\right\Vert^2_{M/2}}\\
    &\leq \frac{C}{\sqrt{T}(j - 1 - h - \lambda_2)^\alpha} + \frac{C\sqrt{T}}{\sqrt{(\lambda_2 - \lambda_1)}(j - 1 - \lambda_2)^{\alpha - 2}}T^{2/M}(\lambda_2 - \lambda_1)^{2/M}.
\end{align*}
Therefore,
\begin{equation}
\begin{aligned}
   \left\Vert\sum_{t_1 = 1 + \lambda_1}^{T - s}O^\dagger_{t_1 + s, s - 1}O^\dagger_{t_1}\right\Vert_{M/4} &\leq \sum_{j = 0}^{2\lambda_2 + h + 1}\frac{C}{\sqrt{T}} + \sum_{j = 2\lambda_2 + h + 2}^{\infty}\frac{C}{\sqrt{T}(j - 1 - h - \lambda_2)^\alpha} \\
   &+ \sum_{j = 2\lambda_2 + h + 2}^{\infty}\frac{C\sqrt{T}}{\sqrt{(\lambda_2 - \lambda_1)}(j - 1 - \lambda_2)^{\alpha - 2}}T^{2/M}(\lambda_2 - \lambda_1)^{2/M}\\
   &\leq \frac{C_1\lambda_2 + C_1h}{\sqrt{T}} + \frac{C_1}{\sqrt{T}\lambda_2^{\alpha - 1}} + \frac{C_1\sqrt{T}}{\sqrt{\lambda_2 - \lambda_1}(\lambda_2 + h)^{\alpha - 3}}T^{2/M}(\lambda_2 - \lambda_1)^{2/M}\\
   &\leq \frac{C_2\lambda_2 + C_2h}{\sqrt{T}}.
\end{aligned}
\label{eq.OO_lag}
\end{equation}
Since 
\begin{align*}
    &\left\Vert \sum_{t_1 = 1 + \lambda_1}^{T - s}\left(O_{t_1}^\dagger P_{t_1 + s, s - 1}^\dagger - \mathbf{E}\left[O_{t_1}^\dagger P_{t_1 + s, s - 1}^\dagger\right]\right)\right\Vert_{M/4}\\
    &\leq \sum_{j = 0}^\infty\left\Vert \sum_{t_1 = 1 + \lambda_1}^{T - s}\left(\mathbf{E}\left[O_{t_1}^\dagger P_{t_1 + s, s - 1}^\dagger\mid\mathcal{F}_{t_1 +s, j}\right] - \mathbf{E}\left[O_{t_1}^\dagger P_{t_1 + s, s - 1}^\dagger\mid\mathcal{F}_{t_1 + s, j - 1}\right]\right)\right\Vert_{M/4},
\end{align*}
where we define $\mathbf{E}\left[O_{t_1}^\dagger P_{t_1 + s, s - 1}^\dagger\mid\mathcal{F}_{t_1 + s, - 1}\right] = \mathbf{E}\left[O_{t_1}^\dagger P_{t_1 + s, s - 1}^\dagger\right].$ For any $z = 1,\cdots, T - s - \lambda_1,$ define 
\begin{align*}
    W_{z,j} = \sum_{t_1 = T - s - z + 1}^{T - s}\left(\mathbf{E}\left[O_{t_1}^\dagger P_{t_1 + s, s - 1}^\dagger\mid\mathcal{F}_{t_1 +s, j}\right] - \mathbf{E}\left[O_{t_1}^\dagger P_{t_1 + s, s - 1}^\dagger\mid\mathcal{F}_{t_1 + s, j - 1}\right]\right),
\end{align*}
and $\mathcal{W}_{z,j}$ the $\sigma$-field generated by $e_T,e_{T-1},\cdots, e_{T-z+1 - j}.$ Then $W_{z,j}$ is measurable in $\mathcal{W}_{z,j},$ $\mathcal{W}_{z,j}\subset \mathcal{W}_{z + 1,j},$ and 
\begin{align*}
    &\mathbf{E}\left[W_{z + 1,j} - W_{z,j}\mid \mathcal{W}_{z,j}\right]\\
    &= 
    \mathbf{E}\left[\left(\mathbf{E}\left[O_{T-s-z}^\dagger P_{T-z, s - 1}^\dagger\mid\mathcal{F}_{T-z, j}\right] - \mathbf{E}\left[O_{T-s-z}^\dagger P_{T-z, s - 1}^\dagger\mid\mathcal{F}_{T-z, j - 1}\right]\right)\mid \mathcal{W}_{z,j}\right]\\
    & =  \mathbf{E}\left[O_{T-s-z}^\dagger P_{T-z, s - 1}^\dagger\mid\mathcal{F}_{T-z, j - 1}\right] - \mathbf{E}\left[O_{T-s-z}^\dagger P_{T-z, s - 1}^\dagger\mid\mathcal{F}_{T-z, j - 1}\right] = 0,
\end{align*}
so $W_{z,j} $ forms a martingale. According to Theorem 1.1 of \cite{MR0400380}, 
\begin{equation}
\begin{aligned}
    &\left\Vert \sum_{t_1 = 1 + \lambda_1}^{T - s}\left(\mathbf{E}\left[O_{t_1}^\dagger P_{t_1 + s, s - 1}^\dagger\mid\mathcal{F}_{t_1 +s, j}\right] - \mathbf{E}\left[O_{t_1}^\dagger P_{t_1 + s, s - 1}^\dagger\mid\mathcal{F}_{t_1 + s, j - 1}\right]\right)\right\Vert_{M/4}\\
    &\leq C\sqrt{\sum_{t_1 = 1 + \lambda_1}^{T - s}\left\Vert \mathbf{E}\left[O_{t_1}^\dagger P_{t_1 + s, s - 1}^\dagger\mid\mathcal{F}_{t_1 +s, j}\right] - \mathbf{E}\left[O_{t_1}^\dagger P_{t_1 + s, s - 1}^\dagger\mid\mathcal{F}_{t_1 + s, j - 1}\right]\right\Vert^2_{M/4}}.
\end{aligned}
\label{eq.moment_inequality_X}
\end{equation}
For any $j\in\mathbf{Z},$
\begin{align*}
    &\left\Vert \mathbf{E}\left[O_{t_1}^\dagger P_{t_1 + s, s - 1}^\dagger\mid\mathcal{F}_{t_1 +s, j}\right] - \mathbf{E}\left[O_{t_1}^\dagger P_{t_1 + s, s - 1}^\dagger\mid\mathcal{F}_{t_1 + s, j - 1}\right]\right\Vert_{M/4}
    \leq 2\left\Vert O_{t_1}^\dagger\right\Vert_{M/2}\left\Vert P_{t_1 + s, s - 1}^\dagger\right\Vert_{M/2}.
\end{align*}
On the other hand, if $j > s,$ we have 
\begin{equation}
\begin{aligned}
     P_{t_1 + s, s - 1}^\dagger &= O_{t_1 + s}^\dagger - O^\dagger_{t_1 + s, s - 1} \\
     &=  O_{t_1 + s}^\dagger - O^\dagger_{t_1 + s, j - 1} +O^\dagger_{t_1 + s, j - 1} - O^\dagger_{t_1 + s, s - 1}\\
    &= P^\dagger_{t_1 + s, j - 1} + O^\dagger_{t_1 + s, j - 1} - O^\dagger_{t_1 + s, s - 1},
\end{aligned}
\label{eq.P_separate}
\end{equation}
and $O^\dagger_{t_1 + s, j - 1} - O^\dagger_{t_1 + s, s - 1}$ is measurable in $\mathcal{F}_{t_1 + s, j - 1},$ which implies  
\begin{align*}
    &\left\Vert \mathbf{E}\left[O_{t_1}^\dagger P_{t_1 + s, s - 1}^\dagger\mid\mathcal{F}_{t_1 +s, j}\right] - \mathbf{E}\left[O_{t_1}^\dagger P_{t_1 + s, s - 1}^\dagger\mid\mathcal{F}_{t_1 + s, j - 1}\right]\right\Vert_{M/4}\\
    &\leq \left\Vert
    \mathbf{E}\left[O_{t_1}^\dagger P_{t_1 + s, j - 1}^\dagger\mid\mathcal{F}_{t_1 +s, j}\right] - \mathbf{E}\left[O_{t_1}^\dagger P_{t_1 + s, j - 1}^\dagger\mid\mathcal{F}_{t_1 + s, j - 1}\right]\right\Vert_{M/4}\\
    &+ \left\Vert
    \left(O^\dagger_{t_1 + s, j - 1} - O^\dagger_{t_1 + s, s - 1}\right)\left(\mathbf{E}\left[O_{t_1}^\dagger \mid\mathcal{F}_{t_1 +s, j}\right] - \mathbf{E}\left[O_{t_1}^\dagger \mid\mathcal{F}_{t_1 + s, j - 1}\right]\right)
    \right\Vert_{M/4}\\
    &\leq 2\left\Vert O_{t_1}^\dagger\right\Vert_{M/2}\left\Vert P^\dagger_{t_1 + s, j - 1}\right\Vert_{M/2} + \left\Vert O^\dagger_{t_1 + s, j - 1} - O^\dagger_{t_1 + s, s - 1}\right\Vert_{M/2}\left\Vert\mathbf{E}\left[O_{t_1}^\dagger\mid\mathcal{F}_{t_1 +s, j}\right] - \mathbf{E}\left[O_{t_1}^\dagger\mid\mathcal{F}_{t_1 + s, j - 1}\right]\right\Vert_{M/2}\\
    & = 2\left\Vert O_{t_1}^\dagger\right\Vert_{M/2}\left\Vert P^\dagger_{t_1 + s, j - 1}\right\Vert_{M/2} + \left\Vert P^\dagger_{t_1 + s, s - 1} - P^\dagger_{t_1 + s, j - 1}\right\Vert_{M/2}\left\Vert\mathbf{E}\left[O_{t_1}^\dagger\mid\mathcal{F}_{t_1, j - s}\right] - \mathbf{E}\left[O_{t_1}^\dagger\mid\mathcal{F}_{t_1, j - s - 1}\right]\right\Vert_{M/2}\\
    & =  2\left\Vert O_{t_1}^\dagger\right\Vert_{M/2}\left\Vert P^\dagger_{t_1 + s, j - 1}\right\Vert_{M/2} + \left\Vert P^\dagger_{t_1 + s, s - 1} - P^\dagger_{t_1 + s, j - 1}\right\Vert_{M/2}\left\Vert
    P^\dagger_{t_1, j - s - 1} - P^\dagger_{t_1, j - s}\right\Vert_{M/2}.
\end{align*}
Suppose $s\leq 2\lambda_2 + h + 1.$ In this case, if $j\leq 4\lambda_2 + 2h + 2,$ then from \eqref{eq.moment_O_t} and \eqref{eq.size_P_t_S}
\begin{align*}
    &\left\Vert \mathbf{E}\left[O_{t_1}^\dagger P_{t_1 + s, s - 1}^\dagger\mid\mathcal{F}_{t_1 +s, j}\right] - \mathbf{E}\left[O_{t_1}^\dagger P_{t_1 + s, s - 1}^\dagger\mid\mathcal{F}_{t_1 + s, j - 1}\right]\right\Vert_{M/4}\\
    &\leq 2\left\Vert O_{t_1}^\dagger\right\Vert_{M/2}\left\Vert P_{t_1 + s, s - 1}^\dagger\right\Vert_{M/2}\leq \frac{C}{T},
\end{align*}
which makes
\begin{equation}
\begin{aligned}
    &\left\Vert \sum_{t_1 = 1 + \lambda_1}^{T - s}\left(\mathbf{E}\left[O_{t_1}^\dagger P_{t_1 + s, s - 1}^\dagger\mid\mathcal{F}_{t_1 +s, j}\right] - \mathbf{E}\left[O_{t_1}^\dagger P_{t_1 + s, s - 1}^\dagger\mid\mathcal{F}_{t_1 + s, j - 1}\right]\right)\right\Vert_{M/4}
    \leq \frac{C}{\sqrt{T}}.
\end{aligned}
\label{eq.sum_moment}
\end{equation}
If $j > 4\lambda_2 + 2h + 2,$ then $j - s - 1 > 4\lambda_2 + 2h + 2  - (2\lambda_2 + h + 1) = 2\lambda_2 +h + 1,$ and \eqref{eq.def_P_t_S} makes 
\begin{align*}
    &\left\Vert \mathbf{E}\left[O_{t_1}^\dagger P_{t_1 + s, s - 1}^\dagger\mid\mathcal{F}_{t_1 +s, j}\right] - \mathbf{E}\left[O_{t_1}^\dagger P_{t_1 + s, s - 1}^\dagger\mid\mathcal{F}_{t_1 + s, j - 1}\right]\right\Vert_{M/4}\\
    &\leq \frac{C}{T(j - 1 - h - \lambda_2)^\alpha} + \frac{C}{\sqrt{(\lambda_2 - \lambda_1)}(j - 1 - \lambda_2)^{\alpha - 2}}T^{2/M}(\lambda_2 - \lambda_1)^{2/M}\\
    & + \frac{C}{T(j -s - 1 - h - \lambda_2)^\alpha} + \frac{C}{\sqrt{(\lambda_2 - \lambda_1)}(j - s - 1 - \lambda_2)^{\alpha - 2}}T^{2/M}(\lambda_2 - \lambda_1)^{2/M}\\
    &\leq \frac{C_1}{T(j -s - 1 - h - \lambda_2)^\alpha} + \frac{C_1}{\sqrt{(\lambda_2 - \lambda_1)}(j - s - 1 - \lambda_2)^{\alpha - 2}}T^{2/M}(\lambda_2 - \lambda_1)^{2/M}.
\end{align*}
Applying this inequality to \eqref{eq.moment_inequality_X}, we derive
\begin{align*}
    &\left\Vert \sum_{t_1 = 1 + \lambda_1}^{T - s}\left(\mathbf{E}\left[O_{t_1}^\dagger P_{t_1 + s, s - 1}^\dagger\mid\mathcal{F}_{t_1 +s, j}\right] - \mathbf{E}\left[O_{t_1}^\dagger P_{t_1 + s, s - 1}^\dagger\mid\mathcal{F}_{t_1 + s, j - 1}\right]\right)\right\Vert_{M/4}\\
    &\leq \frac{C_1}{\sqrt{T}(j -s - 1 - h - \lambda_2)^\alpha} + \frac{C_1\sqrt{T}}{\sqrt{(\lambda_2 - \lambda_1)}(j - s - 1 - \lambda_2)^{\alpha - 2}}T^{2/M}(\lambda_2 - \lambda_1)^{2/M}.
\end{align*}
Combine this inequality with \eqref{eq.sum_moment} and take summation, 
\begin{equation}
\begin{aligned}
    &\left\Vert \sum_{t_1 = 1 + \lambda_1}^{T - s}\left(O_{t_1}^\dagger P_{t_1 + s, s - 1}^\dagger - \mathbf{E}\left[O_{t_1}^\dagger P_{t_1 + s, s - 1}^\dagger\right]\right)\right\Vert_{M/4}\\
    &\leq \sum_{j = 0}^{4\lambda_2 + 2h + 2}\frac{C}{\sqrt{T}} + \sum_{j = 4\lambda_2 + 2h + 3}^{\infty}\frac{C}{\sqrt{T}(j -s - 1 - h - \lambda_2)^\alpha}\\
    &+ \sum_{j = 4\lambda_2 + 2h + 3}^{\infty}\frac{C\sqrt{T}}{\sqrt{(\lambda_2 - \lambda_1)}(j - s - 1 - \lambda_2)^{\alpha - 2}}T^{2/M}(\lambda_2 - \lambda_1)^{2/M}\\
    &\leq \frac{C_1\lambda_2 + C_1h}{\sqrt{T}} + \frac{C_1}{\sqrt{T}\lambda_2^{\alpha  -  1}} + \frac{C_1\sqrt{T}}{\sqrt{(\lambda_2 - \lambda_1)}(\lambda_2 + h)^{\alpha - 3}}T^{2/M}(\lambda_2 - \lambda_1)^{2/M}\\
    &\leq \frac{C_2\lambda_2 + C_2h}{\sqrt{T}}.
\end{aligned}
\label{eq.OP_first_part}
\end{equation}
Equation \eqref{eq.OP_first_part} derives the moment bound for  $\left\Vert \sum_{t_1 = 1 + \lambda_1}^{T - s}\left(O_{t_1}^\dagger P_{t_1 + s, s - 1}^\dagger - \mathbf{E}\left[O_{t_1}^\dagger P_{t_1 + s, s - 1}^\dagger\right]\right)\right\Vert_{M/4}$ under the setting $s\leq 2\lambda_2 + h + 1.$ We next turn to the complementary case where $s > 2\lambda_2 + h + 1.$

In this case, if $j\leq 2s + 1,$  from \eqref{eq.moment_O_t} and \eqref{eq.def_P_t_S},
\begin{align*}
    &\left\Vert \mathbf{E}\left[O_{t_1}^\dagger P_{t_1 + s, s - 1}^\dagger\mid\mathcal{F}_{t_1 +s, j}\right] - \mathbf{E}\left[O_{t_1}^\dagger P_{t_1 + s, s - 1}^\dagger\mid\mathcal{F}_{t_1 + s, j - 1}\right]\right\Vert_{M/4}\\
    &\leq 2\left\Vert O_{t_1}^\dagger P_{t_1 + s, s - 1}^\dagger\right\Vert_{M/4}\\
    &\leq \frac{C_1}{T(s - 1 - h - \lambda_2)^\alpha} + \frac{C_1}{\sqrt{(\lambda_2 - \lambda_1)}(s - 1 - \lambda_2)^{\alpha - 2}}T^{2/M}(\lambda_2 - \lambda_1)^{2/M},
\end{align*}
which, according to \eqref{eq.moment_inequality_X}, makes
\begin{align*}
    &\left\Vert \sum_{t_1 = 1 + \lambda_1}^{T - s}\left(\mathbf{E}\left[O_{t_1}^\dagger P_{t_1 + s, s - 1}^\dagger\mid\mathcal{F}_{t_1 +s, j}\right] - \mathbf{E}\left[O_{t_1}^\dagger P_{t_1 + s, s - 1}^\dagger\mid\mathcal{F}_{t_1 + s, j - 1}\right]\right)\right\Vert_{M/4}\\
    &\leq \frac{C}{\sqrt{T}(s - 1 - h - \lambda_2)^\alpha} + \frac{C\sqrt{T}}{\sqrt{(\lambda_2 - \lambda_1)}(s - 1 - \lambda_2)^{\alpha - 2}}T^{2/M}(\lambda_2 - \lambda_1)^{2/M}.
\end{align*}
If $j > 2s + 1,$ then $j - s - 1 > s  > 2\lambda_2 + h + 1.$ We apply \eqref{eq.P_separate} and derive 
\begin{align*}
    &\left\Vert \mathbf{E}\left[O_{t_1}^\dagger P_{t_1 + s, s - 1}^\dagger\mid\mathcal{F}_{t_1 +s, j}\right] - \mathbf{E}\left[O_{t_1}^\dagger P_{t_1 + s, s - 1}^\dagger\mid\mathcal{F}_{t_1 + s, j - 1}\right]\right\Vert_{M/4}\\
    &\leq 2\left\Vert O_{t_1}^\dagger\right\Vert_{M/2}\left\Vert P^\dagger_{t_1 + s, j - 1}\right\Vert_{M/2} + \left\Vert P^\dagger_{t_1 + s, s - 1} - P^\dagger_{t_1 + s, j - 1}\right\Vert_{M/2}\left\Vert
    P^\dagger_{t_1, j - s - 1} - P^\dagger_{t_1, j - s}\right\Vert_{M/2}\\
    &\leq \frac{C}{T(j - 1 - h - \lambda_2)^\alpha} + \frac{C}{\sqrt{(\lambda_2 - \lambda_1)}(j - 1 - \lambda_2)^{\alpha - 2}}T^{2/M}(\lambda_2 - \lambda_1)^{2/M}\\
    & + \left(\left(\frac{C}{\sqrt{T}(s-1 - h - \lambda_2)^\alpha} + \frac{CT^{1/2 +2/M}(\lambda_2 - \lambda_1)^{2/M}}{\sqrt{(\lambda_2 - \lambda_1)}(s-1 - \lambda_2)^{\alpha - 2}}\right)\right.\\
    &\left.\left(\frac{1}{\sqrt{T}(j-s-1 - h - \lambda_2)^\alpha} + \frac{T^{1/2 + 2/M}(\lambda_2 - \lambda_1)^{2/M}}{\sqrt{(\lambda_2 - \lambda_1)}(j-s-1 - \lambda_2)^{\alpha - 2}}\right)\right).
\end{align*}
Apply \eqref{eq.moment_inequality_X}, it makes 
\begin{align*}
    &\left\Vert \sum_{t_1 = 1 + \lambda_1}^{T - s}\left(\mathbf{E}\left[O_{t_1}^\dagger P_{t_1 + s, s - 1}^\dagger\mid\mathcal{F}_{t_1 +s, j}\right] - \mathbf{E}\left[O_{t_1}^\dagger P_{t_1 + s, s - 1}^\dagger\mid\mathcal{F}_{t_1 + s, j - 1}\right]\right)\right\Vert_{M/4}\\
    &\leq \frac{C}{\sqrt{T}(j - 1 - h - \lambda_2)^\alpha} + \frac{C\sqrt{T}}{\sqrt{(\lambda_2 - \lambda_1)}(j - 1 - \lambda_2)^{\alpha - 2}}T^{2/M}(\lambda_2 - \lambda_1)^{2/M}\\
    & + \left(\left(\frac{C}{(s-1 - h - \lambda_2)^\alpha} + \frac{CT^{1 +2/M}(\lambda_2 - \lambda_1)^{2/M}}{\sqrt{(\lambda_2 - \lambda_1)}(s-1 - \lambda_2)^{\alpha - 2}}\right)\right.\\
    &\left.\left(\frac{1}{\sqrt{T}(j-s-1 - h - \lambda_2)^\alpha} + \frac{T^{1/2 + 2/M}(\lambda_2 - \lambda_1)^{2/M}}{\sqrt{(\lambda_2 - \lambda_1)}(j-s-1 - \lambda_2)^{\alpha - 2}}\right)\right).
\end{align*}
From these observations,  
\begin{equation}
\begin{aligned}
    &\left\Vert \sum_{t_1 = 1 + \lambda_1}^{T - s}\left(O_{t_1}^\dagger P_{t_1 + s, s - 1}^\dagger - \mathbf{E}\left[O_{t_1}^\dagger P_{t_1 + s, s - 1}^\dagger\right]\right)\right\Vert_{M/4}\\
    &\leq \sum_{j = 0}^{2s + 1}\frac{C}{\sqrt{T}(s - 1 - h - \lambda_2)^\alpha} + \frac{C\sqrt{T}}{\sqrt{(\lambda_2 - \lambda_1)}(s - 1 - \lambda_2)^{\alpha - 2}}T^{2/M}(\lambda_2 - \lambda_1)^{2/M}\\
    & + \sum_{j = 2s+2}^\infty \frac{C}{\sqrt{T}(j - 1 - h - \lambda_2)^\alpha} + \frac{C\sqrt{T}}{\sqrt{(\lambda_2 - \lambda_1)}(j - 1 - \lambda_2)^{\alpha - 2}}T^{2/M}(\lambda_2 - \lambda_1)^{2/M}\\   
    & + \sum_{j = 2s+2}^\infty\left(\left(\frac{C}{(s-1 - h - \lambda_2)^\alpha} + \frac{CT^{1 +2/M}(\lambda_2 - \lambda_1)^{2/M}}{\sqrt{(\lambda_2 - \lambda_1)}(s-1 - \lambda_2)^{\alpha - 2}}\right)\right.\\
    &\left.\left(\frac{1}{\sqrt{T}(j-s-1 - h - \lambda_2)^\alpha} + \frac{T^{1/2 + 2/M}(\lambda_2 - \lambda_1)^{2/M}}{\sqrt{(\lambda_2 - \lambda_1)}(j-s-1 - \lambda_2)^{\alpha - 2}}\right)\right)\\
    &\leq \frac{C_1s}{\sqrt{T}(s - 1 - h - \lambda_2)^\alpha} + \frac{C_1s\sqrt{T}}{\sqrt{(\lambda_2 - \lambda_1)}(s - 1 - \lambda_2)^{\alpha - 2}}T^{2/M}(\lambda_2 - \lambda_1)^{2/M}\\
    & + \frac{C_1}{\sqrt{T}(2s+ 1 - h - \lambda_2)^{\alpha - 1}} + \frac{C_1\sqrt{T}}{\sqrt{(\lambda_2  - \lambda_1)}(2s+1-\lambda_2)^{\alpha - 3}}T^{2/M}(\lambda_2 - \lambda_1)^{2/M}\\
    & + \left(\left(\frac{C_1}{(s-1 - h - \lambda_2)^\alpha} + \frac{C_1T^{1 +2/M}(\lambda_2 - \lambda_1)^{2/M}}{\sqrt{(\lambda_2 - \lambda_1)}(s-1 - \lambda_2)^{\alpha - 2}}\right)\right.\\
    &\left.\left(\frac{1}{\sqrt{T}(s + 1 - h - \lambda_2)^{\alpha - 1}} + \frac{T^{1/2 + 2/M}(\lambda_2 - \lambda_1)^{2/M}}{\sqrt{(\lambda_2  - \lambda_1)}(s + 1 - \lambda_2)^{\alpha - 3}}\right)\right)\\
    &\leq \frac{C_2sT^{1/2 + 2/M}(\lambda_2 - \lambda_1)^{2/M}}{\sqrt{(\lambda_2 - \lambda_1)}(s - 1 - h - \lambda_2)^{\alpha - 2}}\\
    &+ \frac{C_2T^{1 +2/M}(\lambda_2 - \lambda_1)^{2/M}}{\sqrt{(\lambda_2 - \lambda_1)}(s-1 - h -  \lambda_2)^{\alpha - 2}}\frac{T^{1/2 + 2/M}(\lambda_2 - \lambda_1)^{2/M}}{\sqrt{(\lambda_2  - \lambda_1)}(s + 1 - h -  \lambda_2)^{\alpha - 3}}.
\end{aligned}
\label{eq.OP_second_part}
\end{equation}
This inequality applies the following facts:  
\begin{align*}
    \frac{1}{(s-1 - h - \lambda_2)^\alpha} &\leq \frac{T^{1 + 2/M}(\lambda_2 - \lambda_1)^{2/M}}{\sqrt{(\lambda_2 - \lambda_1)}(s-1 - h - \lambda_2)^{\alpha - 2}},\\
    \text{and}\quad  
    \frac{1}{(s + 1 - h - \lambda_2)^{\alpha - 1}} &\leq \frac{T^{1 + 2/M}(\lambda_2 - \lambda_1)^{2/M}}{\sqrt{(\lambda_2  - \lambda_1)}(s + 1 - h -  \lambda_2)^{\alpha - 3}}
\end{align*}
for sufficiently large $T.$  For $s > 2\lambda_2 + h + 1$ and $M > 8,$
\begin{align*}
    &\frac{sT^{1/2 + 2/M}(\lambda_2 - \lambda_1)^{2/M}}{\sqrt{(\lambda_2 - \lambda_1)}(s - 1 - h - \lambda_2)^{\alpha - 2}}\\ 
    &= \frac{T^{1/2 + 2/M}(\lambda_2 - \lambda_1)^{2/M}}{\sqrt{(\lambda_2 - \lambda_1)}(s - 1 - h - \lambda_2)^{\alpha - 3}} 
    + \frac{(1 + h + \lambda_2)T^{1/2 + 2/M}(\lambda_2 - \lambda_1)^{2/M}}{\sqrt{(\lambda_2 - \lambda_1)}(s - 1 - h - \lambda_2)^{\alpha - 2}}\\
    & \leq \frac{T^{1/2 + 2/M}}{\lambda_2^{\alpha - 3}} + \frac{(1 + h + \lambda_2)T^{1/2 + 2/M}}{\lambda_2^{\alpha - 2}}\leq \frac{C\lambda_2 + Ch}{\sqrt{T}},
\end{align*}
and 
\begin{align*}
    \frac{T^{3/2 + 4/M}(\lambda_2 - \lambda_1)^{4/M}}{(\lambda_2 - \lambda_1)(s - 1 - h - \lambda_2)^{2\alpha - 5}}\leq \frac{T^{3/2 + 4/M}}{\lambda_2^{2\alpha - 5}}\leq \frac{C\lambda_2 + Ch}{\sqrt{T}}.
\end{align*}

From \eqref{eq.OO_lag}, \eqref{eq.OP_first_part}, and \eqref{eq.OP_second_part}, 
\begin{align*}
    \left\Vert\sum_{t_1 = 1 + \lambda_1}^{T - s}\left(O_{t_1}^\dagger O_{t_1 + s}^\dagger - \mathbf{E}\left[O_{t_1}^\dagger O_{t_1 + s}^\dagger\right]\right)\right\Vert_{M/4}\leq \frac{C\lambda_2 + Ch}{\sqrt{T}},
\end{align*}
and 
\begin{align*}
    &\sum_{s = 0}^{T - 1 - \lambda_1} K\left(\frac{s}{\mathcal{K}_T}\right)\left\Vert\sum_{t_1 = 1 + \lambda_1}^{T - s}\left(O_{t_1}^\dagger O_{t_1 + s}^\dagger - \mathbf{E}\left[O_{t_1}^\dagger O_{t_1 + s}^\dagger\right]\right)\right\Vert_{M/4}\\
    &\leq C\sum_{s = 0}^{T - 1 - \lambda_1} K\left(\frac{s}{\mathcal{K}_T}\right)\frac{\lambda_2 + h}{\sqrt{T}}\leq \frac{C_1\mathcal{K}_T(\lambda_2 + h)}{\sqrt{T}}.
\end{align*}
Since 
\begin{align*}
    \left\Vert\sum_{t_2 = 1 + \lambda_1}^{T - s}\left(O_{t_2 + s}^\dagger O_{t_2}^\dagger - \mathbf{E}\left[O_{t_2 + s}^\dagger O_{t_2}^\dagger\right]\right)\right\Vert_{M/4} & = \left\Vert\sum_{t_1 = 1 + \lambda_1}^{T - s}\left(O_{t_1 + s}^\dagger O_{t_1}^\dagger - \mathbf{E}\left[O_{t_1 + s}^\dagger O_{t_1}^\dagger\right]\right)\right\Vert_{M/4}\leq \frac{C\lambda_2 + Ch}{\sqrt{T}},
\end{align*}
we have 
\begin{align*}
     &\sum_{s = 1}^{T - 1 - \lambda_1} K\left(\frac{s}{\mathcal{K}_T}\right)\left\Vert\sum_{t_2 = 1 + \lambda_1}^{T - s}\left(O_{t_2 + s}^\dagger O_{t_2}^\dagger - \mathbf{E}\left[O_{t_2 + s}^\dagger O_{t_2}^\dagger\right]\right)\right\Vert_{M/4}\\
     &\leq \sum_{s = 1}^{T - 1 - \lambda_1} K\left(\frac{s}{\mathcal{K}_T}\right)\frac{C\lambda_2 + Ch}{\sqrt{T}}\leq \frac{C_1\mathcal{K}_T(\lambda_2 + h)}{\sqrt{T}},
\end{align*}
and 
\begin{equation}
    \begin{aligned}
        &\left\Vert
        \sum_{t_1 = 1+ \lambda_1}^T\sum_{t_2 = 1+ \lambda_1}^T K\left(\frac{t_1  -  t_2}{\mathcal{K}_T}\right)\left(O_{t_1}^\dagger O_{t_2}^\dagger - \mathbf{E}\left[O_{t_1}^\dagger O_{t_2}^\dagger\right]\right)
        \right\Vert_{M/4}\leq \frac{C\mathcal{K}_T(\lambda_2 + h)}{\sqrt{T}}.
    \end{aligned}
    \label{eq.K_OO_EOO}
\end{equation}
Since 
\begin{align*}
    &\sum_{t_1 = 1+ \lambda_1}^T\sum_{t_2 = 1+ \lambda_1}^T\left(1 - K\left(\frac{t_1  -  t_2}{\mathcal{K}_T}\right)\right)\left\vert \mathbf{E}\left[O_{t_1}^\dagger O_{t_2}^\dagger\right]\right\vert\\
    & = \sum_{s = 0}^{T - 1 - \lambda_1}\left(1 - K\left(\frac{s}{\mathcal{K}_T}\right)\right)\sum_{t_1 = 1+ \lambda_1}^{T-s}\left\vert \mathbf{E}\left[O_{t_1}^\dagger O_{t_1 + s}^\dagger\right]\right\vert\\
    & + \sum_{s = 1}^{T - 1 - \lambda_1}\left(1 - K\left(\frac{s}{\mathcal{K}_T}\right)\right)\sum_{t_2 = 1+ \lambda_1}^{T-s}\left\vert \mathbf{E}\left[O_{t_2}^\dagger O_{t_2 + s}^\dagger\right]\right\vert
\end{align*}
If $s \leq 2\lambda_2 + h + 1,$ from \eqref{eq.moment_O_t},  
\begin{align*}
    \left\vert \mathbf{E}\left[O_{t_2}^\dagger O_{t_2 + s}^\dagger\right]\right\vert\leq \left\Vert O_{t_2}^\dagger\right\Vert_{M/2}\left\Vert  O_{t_2 + s}^\dagger\right\Vert_{M/2}\leq \frac{C}{T}.
\end{align*}
On the other hand, if $s > 2\lambda_2 + h + 1,$ since $\mathbf{E}\left[O_{t_1 + s}^\dagger\mid\mathcal{F}_{t_1 + s, s - 1}\right]$ is independent of $O_{t_1}^\dagger,$ we have 
\begin{align*}
    \mathbf{E}\left[O_{t_1}^\dagger\left(\mathbf{E}\left[O_{t_1 + s}^\dagger\mid\mathcal{F}_{t_1 + s, s - 1}\right]\right)\right] = \mathbf{E}\left[O_{t_1}^\dagger\right]\mathbf{E}\left[\left(\mathbf{E}\left[O_{t_1 + s}^\dagger\mid\mathcal{F}_{t_1 + s, s - 1}\right]\right)\right] = 0,
\end{align*}
and applying \eqref{eq.def_P_t_S} leads to 
\begin{align*}
    \left\vert \mathbf{E}\left[O_{t_1}^\dagger O_{t_1 + s}^\dagger\right]\right\vert &= \left\vert \mathbf{E}\left[O_{t_1}^\dagger \left(O_{t_1 + s}^\dagger - \mathbf{E}\left[O_{t_1 + s}^\dagger\mid\mathcal{F}_{t_1 + s, s - 1}\right]\right)\right]\right\vert\\
    &\leq \left\Vert O_{t_1}^\dagger\right\Vert_{M/2}\left\Vert O_{t_1 + s}^\dagger - \mathbf{E}\left[O_{t_1 + s}^\dagger\mid\mathcal{F}_{t_1 + s, s - 1}\right]\right\Vert_{M/2}\\
    &\leq \frac{C}{T(s - 1 - h - \lambda_2)^\alpha} + \frac{CT^{2/M}(\lambda_2 - \lambda_1)^{2/M}}{\sqrt{(\lambda_2 - \lambda_1)}(s - 1 - \lambda_2)^{\alpha - 2}}.
\end{align*}
From Definition \ref{def.kernel}, $\left\vert K^\prime(\cdot)\right\vert$ is bounded by a constant, so for $0\leq s\leq \mathcal{K}_T,$ 
\begin{align*}
    1 - K\left(\frac{s}{\mathcal{K}_T}\right) = K(0) - K\left(\frac{s}{\mathcal{K}_T}\right)\leq \frac{Cs}{\mathcal{K}_T}.
\end{align*}
From Assumption \ref{assumption.for_variance_estimation}, $2\lambda_2 + h + 1\leq \lfloor\mathcal{K}_T\rfloor$  for sufficiently large sample size $T, $ where $\lfloor\mathcal{K}_T\rfloor$ denotes the largest integer that is smaller than or equal to $\mathcal{K}_T.$ Therefore, 
\begin{align*}
    &\sum_{s = 0}^{T - 1 - \lambda_1}\left(1 - K\left(\frac{s}{\mathcal{K}_T}\right)\right)\sum_{t_1 = 1+ \lambda_1}^{T-s}\left\vert \mathbf{E}\left[O_{t_1}^\dagger O_{t_1 + s}^\dagger\right]\right\vert\\
    &\leq \sum_{s = 0}^{2\lambda_2 + h + 1}\frac{Cs}{\mathcal{K}_T} + \sum_{s = 2\lambda_2 + h + 2}^{\lfloor\mathcal{K}_T\rfloor}\frac{Cs}{\mathcal{K}_T(s - 1 - h - \lambda_2)^\alpha} + \sum_{s = 2\lambda_2 + h + 2}^{\lfloor\mathcal{K}_T\rfloor}\frac{Cs}{\mathcal{K}_T}\frac{T^{1 +2/M}(\lambda_2 - \lambda_1)^{2/M}}{\sqrt{(\lambda_2 - \lambda_1)}(s - 1 - \lambda_2)^{\alpha - 2}}\\
    & +  \sum_{s = \lfloor\mathcal{K}_T\rfloor + 1}^{T-1-\lambda_1}\frac{C}{(s - 1 - h - \lambda_2)^\alpha} + \sum_{s = \lfloor\mathcal{K}_T\rfloor + 1}^{T-1-\lambda_1}\frac{CT^{1 +2/M}(\lambda_2 - \lambda_1)^{2/M}}{\sqrt{(\lambda_2 - \lambda_1)}(s - 1 - \lambda_2)^{\alpha - 2}}.
\end{align*}
 Since 
\begin{align*}
    \sum_{s = 0}^{2\lambda_2 + h + 1}\frac{s}{\mathcal{K}_T}\leq \frac{C(\lambda_2 + h)^2}{\mathcal{K}_T},
\end{align*}
\begin{align*}
    &\sum_{s = 2\lambda_2 + h + 2}^{\lfloor\mathcal{K}_T\rfloor}\frac{s}{\mathcal{K}_T(s - 1 - h - \lambda_2)^\alpha}\\
    &= \sum_{s = 2\lambda_2 + h + 2}^{\lfloor\mathcal{K}_T\rfloor}\frac{1}{\mathcal{K}_T(s - 1 - h - \lambda_2)^{\alpha - 1}} + \sum_{s = 2\lambda_2 + h + 2}^{\lfloor\mathcal{K}_T\rfloor}\frac{1 + h +\lambda_2}{\mathcal{K}_T(s - 1 - h - \lambda_2)^\alpha}\\
    &\leq \frac{C}{\mathcal{K}_T\lambda_2^{\alpha - 2}} + \frac{C(1 + h +\lambda_2)}{\mathcal{K}_T\lambda_2^{\alpha - 1}}\leq \frac{C_1(\lambda_2 + h)^2}{\mathcal{K}_T},
\end{align*}
\begin{align*}
    \sum_{s = 2\lambda_2 + h + 2}^{\lfloor\mathcal{K}_T\rfloor}\frac{s}{\mathcal{K}_T}\frac{T^{1 +2/M}(\lambda_2 - \lambda_1)^{2/M}}{\sqrt{(\lambda_2 - \lambda_1)}(s - 1 - \lambda_2)^{\alpha - 2}} &= \sum_{s = 2\lambda_2 + h + 2}^{\lfloor\mathcal{K}_T\rfloor}\frac{T^{1 +2/M}(\lambda_2 - \lambda_1)^{2/M}}{\mathcal{K}_T\sqrt{(\lambda_2 - \lambda_1)}(s - 1 - \lambda_2)^{\alpha - 3}}\\
    &+  \sum_{s = 2\lambda_2 + h + 2}^{\lfloor\mathcal{K}_T\rfloor}\frac{(1 + \lambda_2)}{\mathcal{K}_T}\frac{T^{1 +2/M}(\lambda_2 - \lambda_1)^{2/M}}{\sqrt{(\lambda_2 - \lambda_1)}(s - 1 - \lambda_2)^{\alpha - 2}}\\
    &\leq \frac{CT^{1 +2/M}}{\mathcal{K}_T(\lambda_2 + h)^{\alpha - 4}} + \frac{(1 + \lambda_2)}{\mathcal{K}_T}\frac{T^{1 +2/M}}{(\lambda_2 + h)^{\alpha - 3}}\leq \frac{C(\lambda_2 + h)^2}{\mathcal{K}_T},
\end{align*}
and 
\begin{align*}
    \sum_{s = \lfloor\mathcal{K}_T\rfloor + 1}^{T-1-\lambda_1}\frac{1}{(s - 1 - h - \lambda_2)^\alpha}\leq \frac{C}{(\mathcal{K}_T - h - \lambda_2)^{\alpha - 1}}\leq \frac{C_1(\lambda_2 + h)^2}{\mathcal{K}_T},
\end{align*}
\begin{align*}
    \sum_{s = \lfloor\mathcal{K}_T\rfloor + 1}^{T-1-\lambda_1}\frac{T^{1 +2/M}(\lambda_2 - \lambda_1)^{2/M}}{\sqrt{(\lambda_2 - \lambda_1)}(s - 1 - \lambda_2)^{\alpha - 2}}\leq \frac{CT^{1 +2/M}}{(\mathcal{K}_T  - \lambda_2)^{\alpha - 3}}\leq \frac{C_1(\lambda_2 + h)^2}{\mathcal{K}_T}.
\end{align*}
We have 
\begin{equation}
    \begin{aligned}
         &\sum_{s = 0}^{T - 1 - \lambda_1}\left(1 - K\left(\frac{s}{\mathcal{K}_T}\right)\right)\sum_{t_1 = 1+ \lambda_1}^{T-s}\left\vert \mathbf{E}\left[O_{t_1}^\dagger O_{t_1 + s}^\dagger\right]\right\vert\leq \frac{C(\lambda_2 + h)^2}{\mathcal{K}_T}.
    \end{aligned}
    \label{eq.bias_O_tO_ts}
\end{equation}
From \eqref{eq.OO_to_OOS}, \eqref{eq.K_OO_EOO}, \eqref{eq.bias_O_tO_ts}, and Assumption \ref{assumption.for_variance_estimation}, 
\begin{equation}
    \begin{aligned}
        \left\Vert
        \mathbf{E}^*\left[\widehat{R}^{*2}_b\right] -  \mathrm{Var}\left(\widehat{R}^\dagger\right)
        \right\Vert_{M/4}&\leq \frac{C\mathcal{K}_T(\lambda_2 + h)}{\sqrt{T}} + \frac{C(\lambda_2 + h)^2}{\mathcal{K}_T}\\
        & = O\left(T^{\kappa_k + (\kappa_2\vee \kappa_h) - 1/2} + T^{2(\kappa_2\vee \kappa_h) - \kappa_k}\right).
    \end{aligned}
    \label{eq.difference_variance}
\end{equation}
Define $Z$ as a standard normal random variable that is independent of $\mathbf{z}_t, t = 1,\cdots,T,$ then 
\begin{align*}
    \mathbf{pr}^*\left(\widehat{R}^*_b\leq x\right) - \mathbf{pr}\left(\epsilon\leq x\right) &= \mathbf{pr}^*\left(\widehat{R}^*_b\leq x\right) - \mathbf{pr}\left(\epsilon\leq x + 1/\psi\right) + \mathbf{pr}\left(\epsilon\leq x + 1/\psi\right) - \mathbf{pr}\left(\epsilon\leq x\right)\\
    &\leq \mathbf{E}^*\left[g_{\psi,x}\left(\widehat{R}^*_b\right)\right] - \mathbf{E}\left[g_{\psi,x}\left(\epsilon\right)\right] + \frac{C}{\psi}\\
    & = \mathbf{E}^*\left[g_{\psi,x}\left(\sqrt{\mathbf{E}^*\left[\widehat{R}^{*2}_b\right]}Z\right)\right] - \mathbf{E}^*\left[g_{\psi,x}\left(\sqrt{\mathrm{Var}\left(\widehat{R}^\dagger\right)}Z\right)\right] + \frac{C}{\psi}\\
    &\leq \frac{C}{\psi} + C\psi\left\vert \sqrt{\mathbf{E}^*\left[\widehat{R}^{*2}_b\right]} - \sqrt{\mathrm{Var}\left(\widehat{R}^\dagger\right)}\right\vert
\end{align*}
and 
\begin{align*}
    \mathbf{pr}^*\left(\widehat{R}^*_b\leq x\right) - \mathbf{pr}\left(\epsilon\leq x\right) &= \mathbf{pr}^*\left(\widehat{R}^*_b\leq x\right) - \mathbf{pr}\left(\epsilon\leq x - 1/\psi\right) + \mathbf{pr}\left(\epsilon\leq x - 1/\psi\right) - \mathbf{pr}\left(\epsilon\leq x\right)\\
    &\geq \mathbf{E}\left[g_{\psi, x - 1/\psi}\left(\widehat{R}^*_b\right)\right] - \mathbf{E}\left[g_{\psi, x - 1/\psi}\left(\epsilon\right)\right] - \frac{C}{\psi}\\
    & = \mathbf{E}^*\left[g_{\psi,x - 1/\psi}\left(\sqrt{\mathbf{E}^*\left[\widehat{R}^{*2}_b\right]}Z\right)\right] - \mathbf{E}^*\left[g_{\psi, x - 1/\psi}\left(\sqrt{\mathrm{Var}\left(\widehat{R}^\dagger\right)}Z\right)\right] - \frac{C}{\psi}\\
    &\geq  -\frac{C}{\psi} - C\psi\left\vert \sqrt{\mathbf{E}^*\left[\widehat{R}^{*2}_b\right]} - \sqrt{\mathrm{Var}\left(\widehat{R}^\dagger\right)}\right\vert.
\end{align*}
These two inequalities make 
\begin{align*}
    &\sup_{x\in\mathbf{R}}\left\vert \mathbf{pr}^*\left(\widehat{R}^*_b\leq x\right) - \mathbf{pr}\left(\epsilon\leq x\right)\right\vert\\
    &\leq \frac{C}{\psi} + C\psi\left\vert \sqrt{\mathbf{E}^*\left[\widehat{R}^{*2}_b\right]} - \sqrt{\mathrm{Var}\left(\widehat{R}^\dagger\right)}\right\vert\\
    & = \frac{C}{\psi} + C\psi\frac{\left\vert \mathbf{E}^*\left[\widehat{R}^{*2}_b\right] - \mathrm{Var}\left(\widehat{R}^\dagger\right)\right\vert}{\sqrt{\mathbf{E}^*\left[\widehat{R}^{*2}_b\right]} + \sqrt{\mathrm{Var}\left(\widehat{R}^\dagger\right)}}\leq \frac{C}{\psi} + C_1\psi\left\vert \mathbf{E}^*\left[\widehat{R}^{*2}_b\right] - \mathrm{Var}\left(\widehat{R}^\dagger\right)\right\vert.
\end{align*}
From \eqref{eq.difference_variance}, $\left\vert \mathbf{E}^*\left[\widehat{R}^{*2}_b\right] - \mathrm{Var}\left(\widehat{R}^\dagger\right)\right\vert = O_p\left(T^{\kappa_k + (\kappa_2\vee \kappa_h) - 1/2} + T^{2(\kappa_2\vee \kappa_h) - \kappa_k}\right).$ Choose $\psi = \log(T),$ we have 
\begin{align*}
    \sup_{x\in\mathbf{R}}\left\vert \mathbf{pr}^*\left(\widehat{R}^*_b\leq x\right) - \mathbf{pr}\left(\epsilon\leq x\right)\right\vert = o_p(1),
\end{align*}
and we prove \eqref{eq.bootstrap_validity}.

\end{proof}



\scriptsize
  \bibliography{Arxivs}

\end{document}